\documentclass[11pt,reqno]{amsart}
\usepackage[colorlinks, linkcolor=magenta, citecolor=magenta, urlcolor=red!80!gray, pagebackref]{hyperref}
\usepackage{amssymb}
\usepackage{amsmath,pifont}
\usepackage{mathtools}
\usepackage[all]{xy}
\usepackage{mdwtab}
\usepackage{subcaption}
\usepackage{stmaryrd}
\usepackage{tikz}
\usepackage{mathrsfs}

\usetikzlibrary{arrows.meta}
\usetikzlibrary{positioning,fit,calc}
\usepackage[indent]{parskip}
\usepackage{url}
\usepackage{enumitem}
\usepackage{fullpage}
\usepackage[mark=o,edge length=1.5cm,root-radius=0.18cm]{dynkin-diagrams}
\usepackage{bm}
\usepackage{threeparttable}
\usepackage{pifont}
\usepackage{amssymb}
\usepackage{makecell}

\usepackage{caption}

\usepackage{graphicx}
\usepackage{booktabs}
\usepackage{comment, amsgen, amscd, xspace, epsfig, float, epic}

\usepackage[left=1.2in, right=1.2in, top=1in, bottom=1in]{geometry}

\usepackage{multirow}
\usepackage{tikz}
\usepackage{threeparttable}
\usepackage{enumitem}
\usepackage{fancyhdr}

\usepackage{etoolbox}
\patchcmd{\section}{\scshape}{\bfseries}{}{}
\patchcmd{\subsubsection}{\itshape}{\bfseries}{}{}
\makeatletter
\renewcommand{\@secnumfont}{\bfseries}
\makeatother
\xyoption{all}
\theoremstyle{plain}

\theoremstyle{definition}
\newtheorem{Thm}[equation]{Theorem}
\newtheorem{Lem}[equation]{Lemma}
\newtheorem{Cor}[equation]{Corollary}
\newtheorem{Rem}[equation]{Remark}
\newtheorem{Prop}[equation]{Proposition}
\newtheorem{Exam}[equation]{Example}
\newtheorem{Def}[equation]{Definition}
\newtheorem*{MainThmA}{Theorem A}
\newtheorem*{MainThmB}{Theorem B}
\newtheorem*{MainThmC}{Theorem C}

\newcommand{\hobox}[3]{\draw (0+#1,0-#2) rectangle (1+#1,-1-#2)++(-0.5,+0.5) node {$ #3$};}

\newcommand{\domscale}{0.4}

\numberwithin{equation}{section}

\DeclareMathAlphabet\mathsf{OT1}{cmss}{m}{n}

\newcommand{\frb}{{\mathfrak b}}

\newcommand{\frh}{{\mathfrak h}}
\newcommand{\frk}{{\mathfrak k}}

\newcommand{\frn}{{\mathfrak n}}

\newcommand{\frp}{{\mathfrak p}}

\newcommand{\frt}{{\mathfrak t}}
\newcommand{\fru}{{\mathfrak u}}

\newcommand{\frgl}{\mathfrak{gl}}
\newcommand{\frg}{\mathfrak{g}}

\newcommand{\frso}{\mathfrak{so}}
\newcommand{\frsp}{\mathfrak{sp}}

\newcommand{\Z}{\mathbb Z}      

\newcommand{\GKD}{\operatorname{GKdim}}

\def\frsl{\operatorname{SL}}

\newcommand{\eps}{\epsilon}

\newcommand{\gl}{\mathfrak{gl}}

\newcommand{\so}{\mathfrak{so}}
\renewcommand{\frsl}{\mathfrak{sl}}
\newcommand{\g}{\mathfrak{g}}

\newcommand{\h}{\mathfrak{h}}

\DeclareMathOperator{\ch}{ch}

\newcommand{\AV}{\mathrm{AV}}

\newcommand{\ev}{\mathrm{ev}}
\newcommand{\odd}{\mathrm{od}}

\renewcommand{\arraystretch}{1.5}

\usepackage{enumitem}
\setlist[enumerate,1]{font=\textup,
	leftmargin=5mm,labelsep=0.2em,topsep=1mm,itemsep=1mm,itemindent=1em,listparindent=1em}
\setlist[enumerate,2]{font=\textup, leftmargin=3mm,labelsep=0.2mm,topsep=1mm,itemsep=1mm,itemindent=1em}
\setlist[enumerate,3]{font=\textup, leftmargin=4mm,labelsep=1mm,topsep=1mm,itemsep=1mm,itemindent=1em,listparindent=1em}
\setlist[enumerate,4]{font=\textup, leftmargin=5mm,labelsep=1mm,topsep=1mm,itemsep=1mm,itemindent=1em,listparindent=1em}

\begin{document}
	
\title[GK dimensions of HWM]{Gelfand--Kirillov dimensions of highest weight modules for basic classical Lie superalgebras}


\author{Jing Jiang}
\address{
	School of Mathematical Sciences, East China Normal University, Shanghai 200241, China} 
\email{jjsd6514@163.com}

\begin{abstract}
We obtain explicit formulas for the Gelfand--Kirillov (GK) dimensions of simple highest weight modules over basic classical Lie superalgebras. For the type II orthosymplectic superalgebras $D(m,n)$ and $B(m,n)$, we describe the associated varieties determined by highest weights relative to two Borel subalgebras. This gives algorithms valid for arbitrary complex highest weights, without integrality or typicality hypotheses. In type $B$, we further derive a formula for $\mathfrak{osp}(1|2n)$, distinguishing between integral and half-integral weights. For $G(3)$, $F(4)$ and $D(2,1;\alpha)$, finite odd-reflection procedures give exact GK dimension formulas and finite-dimensionality criteria. We also describe the dependence on $\alpha$ for $D(2,1;\alpha)$. In type I, the GK dimension equals that of the simple even highest weight module for the distinguished Borel subalgebra. The computations rely on established ordinary Lie algebra formulas involving Lusztig's $\mathbf a$-function and Robinson--Schensted (RS) insertion.

\end{abstract}

\subjclass[2020]{17A70, 17B10, 17B20}

\keywords{Gelfand--Kirillov dimension; basic classical Lie superalgebra; highest weight module}

\maketitle

\tableofcontents

\section{Introduction}
\subsection{Gelfand--Kirillov dimension}
Let $\mathfrak{g}_0$ be a complex semisimple Lie algebra and $U(\mathfrak{g}_0)$ its enveloping algebra. The GK dimension, introduced as a measure of growth of algebras and their modules, has become a fundamental invariant in the representation theory of infinite-dimensional algebras. Roughly speaking, the GK dimension provides a quantitative description of the asymptotic behavior of the dimensions of the graded components of a module, and thereby serves as a crucial tool in distinguishing between different classes of representations. More precisely, given an infinite-dimensional $U(\mathfrak{g}_0)$-module $M$, $\operatorname{GKdim}(M)$ captures the polynomial rate of growth of $M$, and can be regarded as an analogue of the Krull dimension in commutative algebra (cf. \cite{AJ1, Vo78}).

The study of GK dimension has a long history. A remarkable result of Joseph \cite{AJ1} established an explicit formula for the GK dimension of regular integral highest weight modules for type $A$. Building on Joseph's foundational work and exploiting Lusztig's theory of the $\mathbf{a}$-function \cite{lusztig1985cellsI,lusztig1985A_n,lusztig2003hecke}, Bai--Xie \cite{BX1} subsequently extended the computation to singular and non-integral cases in type $A$. For the other classical types, Bai--Xiao--Xie \cite{BXX} developed an efficient algorithm that systematically computes the GK dimensions of highest weight modules over classical simple Lie algebras. In the representation theory of Lie algebras, GK dimension is used to measure the size of representations of Lie algebras and Lie groups. Furthermore, it can determine the reducibility of scalar generalized Verma modules; see \cite{BX, BJ24, JJ, JW}. Closely related to the GK dimension is the notion of the associated variety of a module $M$, whose dimension is known to coincide with $\operatorname{GK dim}M$. These two invariants, though defined in distinct ways, are deeply interconnected and have played a significant role in the development of modern representation theory; see \cite{AJ1,AJ2,AJ3,AJ6}.

This work is motivated by extending these ideas to Lie superalgebras. While the GK dimensions of highest weight modules over classical simple Lie algebras have been extensively studied and explicit formulas are now available \cite{BX1,BXX}, the situation for Lie superalgebras is considerably more intricate. The additional structure introduced by the presence of odd roots, as well as the non-conjugacy of Borel subalgebras in the super setting, poses new challenges in understanding the growth properties of highest weight modules. Nevertheless, highest weight theory remains a cornerstone in the study of representations of Lie superalgebras, and the determination of GK dimensions in this context is both natural and important.

An implementation of  Lie algebra algorithm is described in \cite{BGXW} and is available at
\begin{center}
\textcolor{magenta}{http://test.slashblade.top:5000/lie/GKdim}.
\end{center}


\subsection{Main results}
The basic classical Lie superalgebras considered in this paper are listed in \eqref{classical-type}. Here $\operatorname{AV}_{\mathfrak{k}}M$ denotes the associated variety of a finitely generated $U(\mathfrak{k})$-module $M$. For a supermodule, the associated variety is taken after restriction to the even part; see \cite{Mu06} and Subsection \ref{section-GK-dimension}.

\begin{MainThmA}[See Propositions \ref{D(m,n)} and \ref{prop:B-two-Borel}]
Let $V$ be a simple highest weight module for $D(m,n)$, $m\geq2$, $n\geq1$, or $B(m,n)$, $m,n\geq1$. Let
$\Lambda_{\Ydown}=(\lambda_{\Ydown}\mid\mu_{\Ydown})$
and $\Lambda_{\mid}=(\lambda_{\mid}\mid\mu_{\mid})$ be its highest weights for the two compatible systems specified in Section \ref{section-GK-basic}. Put
\[
(\frk,\frt)=
\begin{cases}
(\frso(2m),\frsp(2n)),&\frg=D(m,n),\\
(\frso(2m+1),\mathfrak{osp}(1|2n)),&\frg=B(m,n).
\end{cases}
\]
Then
\[
\begin{aligned}
\AV_{\frg_0}V
&=\AV_{\frk}L_{\frk}(\lambda_{\Ydown})
 \times\AV_{\frsp(2n)}L_{\frt}(\mu_{\mid}),\\
\GKD V&=\GKD L_{\frk}(\lambda_{\Ydown})
        +\GKD L_{\frt}(\mu_{\mid}).
\end{aligned}
\]
No integrality or typicality hypothesis is imposed.
\end{MainThmA}

Gorelik--Grantcharov \cite{GG21} use related factor reductions to study boundedness of highest weight modules. Theorem A identifies the associated variety as a product of factor supports and gives an exact GK dimension formula. Combined with RS insertion, this gives the explicit algorithms in Theorems \ref{thm:D-GK-all} and \ref{thm:B-GK-all}.

\begin{MainThmB}[See Lemma \ref{lem:osp-terminal-GK}]
For $\frt=\mathfrak{osp}(1|2n)$ and any $\mu\in\mathbb C^n$, put
$\xi=\mu+\rho_{\frt}$, where
$\rho_{\frt}=(n-\frac12,n-\frac32,\ldots,\frac12)$. Then
\[
\GKD L_{\frt}(\mu)
=n^2-F_B(\xi_{\frac12}^-)-F_D(\xi_0^-)
 -\sum_{\xi}F_A(\widetilde\xi_z).
\]
Here $F_A,F_B,F_D$ are the functions defined in Definition \ref{F_X}.
\end{MainThmB}

\begin{MainThmC}[See Theorems \ref{thm:GK-G3-exceptional}, \ref{thm:GK-F4-exceptional} and \ref{thm:GK-D21-exceptional}]
Let $\varkappa(c)=0$ for $c\in\mathbb Z_{\geq0}$ and $\varkappa(c)=1$ otherwise. Then we have the following table.
\begin{table}[htbp]
    \centering
    \begin{tabular}{cl}
        \toprule
        $\frg$ & $\GKD\widetilde{L}(\Lambda)$ \\ 
        \midrule
        $G(3)$ & $\GKD L_{G_2}(\mu)+\varkappa(\lambda_4)$, \\
        $F(4)$ & $\GKD L_{\frso(7)}(\lambda)+\varkappa(2\mu^{(4)})$, \\
        $D(2,1;\alpha)$ & $\varkappa(\lambda_1-d_1-d_2)+ \varkappa(\lambda_2)+\varkappa(\lambda_3)$. \\
        \bottomrule
    \end{tabular}
\end{table}
\end{MainThmC}

For $D(2,1;\alpha)$, Corollary \ref{cor:D21-GK-parameter} also describes exactly when the GK dimension of a fixed highest weight changes with $\alpha$. The finite-dimensionality criteria below recover the classical classifications \cite{Mu12,CW12,ML} from the GK dimension.

\subsection{A uniform  formula for Type I }\label{A(Y)}
Lusztig's $\mathbf a$-function is defined using the Kazhdan--Lusztig basis of the Hecke algebra; see Subsection \ref{a_function}. For an ordinary highest weight module with positive roots $\Phi^+$, the uniform formula is
\[
\GKD L(\lambda)=|\Phi^+|-
\mathbf a_{[\lambda^{\rho_0}]}(w_{\lambda^{\rho_0}}),
\]
where the $\mathbf a$-function belongs to the integral Weyl group and $w_{\lambda^{\rho_0}}$ is the minimal length element of Lemma \ref{GKdim-uniform1}; see \cite{BX1}.
The even Weyl group $W_0$ is the product of the Weyl groups of the semisimple factors of $\frg_0$, and the $\mathbf a$-function is additive on this product. In type I with the distinguished Borel, the Kac-module bounds reduce the supermodule calculation to these ordinary factors, as follows.

\begin{Prop}[See Proposition \ref{semisimple-GK-dim}]
If $\frg$ is of type I with the distinguished Borel subalgebra, then $\text{GK}\dim \widetilde{L}(\lambda|\mu)=\text{GK}\dim L(\lambda)+\text{GK}\dim L(\mu)$.
\end{Prop}

For type $A$, let $\Lambda=(\lambda\mid\mu)$ and consider 
$\lambda^{\rho_0}=\lambda+\left(\frac{m-1}{2},\frac{m-3}{2},\ldots,\frac{1-m}{2}\right)$.
For each coset $z+\mathbb Z$ containing an entry of $\lambda^{\rho_0}$, take the subsequence consisting of all entries in that coset, preserving their original order. Applying RS insertion to each such subsequence gives a set $P(\lambda)$ of Young tableaux. Similarly, we can define $P(\mu)$ and set $P(\Lambda)=P(\lambda\mid\mu)=P(\lambda)\sqcup P(\mu)$. Thus entries from different cosets or different even factors are never inserted together. For a Young tableau $Y$, define $A(Y)=\sum_{i\geq1}{\bf q}_i({\bf q}_i-1)/2$, where ${\bf q}_i$ is the number of entries in the $i$-th column of $Y$, and set $A(P(\Lambda))=\sum_{Y\in P(\Lambda)}A(Y)$.

\begin{Thm}[See Theorem \ref{GK-dim-gl(m|n)}]
  Let $\frg=\frsl(m|n)$ with $m\neq n$ and the distinguished Borel. For any $\Lambda\in\frh^*$, we have
 \begin{align*}
   \text{GK}\dim \widetilde{L}(\lambda|\mu)&=|\Delta_0^+|-A(P(\lambda|\mu)) \\&=\frac{m(m-1)}{2}-A(P(\lambda))+\frac{n(n-1)}{2}-A(P(\mu)).
 \end{align*}  
  
\end{Thm}


\begin{Exam}
Take $\frg=\frgl(6|7)$ and
\[
\Lambda^{\rho_0}:=\Lambda+\rho_0
=(3,3.2,2,-1,2.2,4\mid2,1,2.5,3,4.5,1.5,-1).
\]
 $P(\lambda)$ and $P(\mu)$ contain two Young tableaux:
\[\underbrace{\tiny{{\begin{tikzpicture}[scale=\domscale,baseline=-19pt]
			\hobox{0}{0}{-1}
			\hobox{1}{0}{4}
			\hobox{0}{1}{2}
        \hobox{0}{2}{3}
	\end{tikzpicture}}}\hspace{2em}
  \tiny{\begin{tikzpicture}[scale=\domscale+0.05,baseline=-19pt]
			\hobox{0}{0}{2.2}
			\hobox{0}{1}{3.2}
	\end{tikzpicture}}}_{P(\lambda)}\hspace{5em}
  \underbrace{\tiny{{\begin{tikzpicture}[scale=\domscale,baseline=-19pt]
			\hobox{0}{0}{-1}
			\hobox{1}{0}{3}
			\hobox{0}{1}{1}
        \hobox{0}{2}{2}
	\end{tikzpicture}}}\hspace{2em}
  \tiny{\begin{tikzpicture}[scale=\domscale+0.05,baseline=-19pt]
			\hobox{0}{0}{1.5}
        \hobox{1}{0}{4.5}
			\hobox{0}{1}{2.5}
	\end{tikzpicture}}}_{P(\mu)}
	\]
By the above theorem we have $\text{GK}\dim \widetilde{L}(\lambda|\mu)=(\frac{6\cdot 5}{2}-\frac{3\cdot 2}{2}-\frac{2\cdot 1}{2})+(\frac{7\cdot 6}{2}-\frac{3\cdot 2}{2}-\frac{2\cdot 1}{2})=28$.
\end{Exam}

For other basic classical Lie superalgebras, we recall the RS insertion algorithm and some notations which will be used in our paper. Some details can be found in \cite{BX1,BXX}.

For a totally ordered set $ \Gamma $, we denote by $ \mathrm{Seq}_n (\Gamma)$ the set of sequences $ x=(x_1,x_2,\dots, x_n) $  of length $ n $ with $ x_i\in\Gamma $. In our paper, we usually take $\Gamma$ to be $\mathbb{Z}$. Then we have a Young tableau $Y(x)$ obtained by applying the following RS insertion algorithm to $x\in \mathrm{Seq}_n (\Gamma)$. Whenever a sequence lies in a coset \(z + \mathbb{Z} \subseteq \mathbb{C}\), we use the order \(a < b\) if \(b - a \in \mathbb{Z}_{>0}\). Equivalently, we apply RS insertion after subtracting \(z\) from every entry. This convention also applies to the sequence \(\widetilde{x}_z\).
 \begin{Def}[RS insertion algorithm]
For an element $ x \in \mathrm{Seq}_n (\Gamma)$, we write $x=(x_1,\dots,x_n)$. We associate to $x $ a Young tableau $ Y(x) $ as follows: Let $ Y_0 $ be an empty Young tableau. Assume that we have constructed the Young tableau $ Y_k $ associated to $ (x_1,\dots,x_k) $, $ 0\leq k<n $. Then $ Y_{k+1} $ is obtained by adding $ x_{k+1} $ to $ Y_k $ as follows. First, we add $ x_{k+1} $ to the first row of $ Y_k $ by replacing the leftmost entry $ x_j $ in the first row which is \textit{strictly} bigger than $ x_{k+1} $. (If there is no such entry $ x_j $, we just add a box with entry $x_{k+1} $ to the right side of the first row, and end this process). Then add this $ x_j $ to the next row as the same way of adding $x_{k+1} $ to the first row. Finally we put $Y(x)=Y_n$.

\end{Def}

Denote the shape of the Young tableau $Y(x)$ by ${\bf p}(x)=\mathrm{sh}(Y(x))=[{\bf p}_1,{\bf p}_2,...,{\bf p}_k]$, where ${\bf p}_i$ is the number of boxes in the $i$-th
row of $Y(x)$. Then ${\bf p}(x)$ is a partition of $n=\sum_{1\leq i\leq k} {\bf p}_i$. Denote by ${\bf q}(x)={\bf p}'(x)=[{\bf q}_1,{\bf q}_2,\cdots,{\bf q}_N]$, the dual partition of ${\bf p}(x)$, if ${\bf q}_i$ is the number of boxes in the $i$-th
column of $Y(x)$.


\begin{Exam}
    Suppose $x=(2,1, 3, -2, -1, 1)$. Usually we write $$x=(2,1, 3, -2, -1, 1')$$ and regard $1<1'$. Then from the RS algorithm, we have
$$
\tiny{\begin{tikzpicture}[scale=\domscale,baseline=-18pt]
		\hobox{0}{0}{2}	
		\end{tikzpicture}}\stackrel{1}{\to}
 \begin{tikzpicture}[scale=\domscale+0.1,baseline=-18pt]
		\hobox{0}{0}{1}
		\hobox{0}{1}{2}
 
		\end{tikzpicture}\stackrel{3}{\to} 
 \begin{tikzpicture}[scale=\domscale+0.1,baseline=-18pt]
		\hobox{0}{0}{1}
		\hobox{1}{0}{3}
      \hobox{0}{1}{2}
 
		\end{tikzpicture}\stackrel{-2}{\to}
 \begin{tikzpicture}[scale=\domscale+0.1,baseline=-18pt]
		\hobox{0}{0}{-2}
		\hobox{1}{0}{3}
		\hobox{0}{1}{1}
      \hobox{0}{2}{2}
		\end{tikzpicture}\stackrel{-1}{\to}
\begin{tikzpicture}[scale=\domscale+0.1,baseline=-18pt]
		\hobox{0}{0}{-2}
		\hobox{1}{0}{-1}
		\hobox{0}{1}{1}
      \hobox{1}{1}{3}
      \hobox{0}{2}{2}
		\end{tikzpicture}\stackrel{1'}{\to} 
  \begin{tikzpicture}[scale=\domscale+0.1,baseline=-18pt]
		\hobox{0}{0}{-2}
		\hobox{1}{0}{-1}
      \hobox{2}{0}{1'}
		\hobox{0}{1}{1}
      \hobox{1}{1}{3}
      \hobox{0}{2}{2}
		\end{tikzpicture}=Y(x).$$
 Thus we have ${\bf p}(x)=[3,2,1]$, which is a partition of $6$.
\end{Exam}

 For a Young diagram $Y$, use $ (k,l) $ to denote the box in the $ k $-th row and the $ l $-th column.
We say the box $ (k,l) $ is {even} (resp. {odd}) if $ k+l $ is even (resp. odd). Let $ {\bf p}_i ^{\ev}$ (resp. $ {\bf p}_i^{\odd} $) be the number of even (resp. odd) boxes in the $ i $-th row of the Young diagram $ Y $.
The row counts are \begin{equation}\label{eq:ev-od}
	{\bf p}_i^{\ev}=\begin{cases}
		\left\lceil \frac{{\bf p}_i}{2} \right\rceil,&\text{ if } i \text{ is odd},\\
		\left\lfloor \frac{{\bf p}_i}{2} \right\rfloor,&\text{ if } i \text{ is even},
	\end{cases}
	\quad {\bf p}_i^{\odd}=\begin{cases}
		\left\lfloor \frac{{\bf p}_i}{2} \right\rfloor,&\text{ if } i \text{ is odd},\\
		\left\lceil \frac{{\bf p}_i}{2} \right\rceil,&\text{ if } i \text{ is even}.
	\end{cases}
\end{equation}
Here for $ a\in \mathbb{R} $, $ \lfloor a \rfloor $ is the largest integer $ n $ such that $ n\leq a $, and $ \lceil a \rceil$ is the smallest integer $n$ such that $ n\geq a $. For convenience, we set
\begin{equation*}
	{\bf p}^{\ev}=[{\bf p}_1^{\ev},{\bf p}_2^{\ev},\cdots]\quad\mbox{and}\quad {\bf p}^{\odd}=[{\bf p}_1^{\odd},{\bf p}_2^{\odd},\cdots].
\end{equation*}

\begin{Exam}
	Let ${\bf p}=[6,5,4,3,2,1]$ be the shape of the Young diagram $Y$. Then odd and even boxes in $Y$ are marked as follows.
	\[
	\tiny\begin{tikzpicture}[scale=0.4,baseline=-40pt]
			\hobox{0}{0}{E}
			\hobox{0}{1}{O}
			\hobox{0}{2}{E}
			\hobox{0}{3}{O} 
			\hobox{0}{4}{E}	
			\hobox{0}{5}{O}		
			\hobox{1}{0}{O}
			\hobox{1}{1}{E}
			\hobox{1}{2}{O}
			\hobox{1}{3}{E} 
			\hobox{1}{4}{O}
			\hobox{2}{0}{E}
			\hobox{2}{1}{O}
			\hobox{2}{2}{E}
			\hobox{2}{3}{O} 
			\hobox{3}{0}{O}
			\hobox{3}{1}{E}
			\hobox{3}{2}{O}
			\hobox{4}{0}{E}
			\hobox{4}{1}{O}
			\hobox{5}{0}{O}		
	\end{tikzpicture}
	\] 
	Then ${\bf p}^{\ev}=[3,2,2,1,1]$ and ${\bf p}^{\odd}=[3,3,2,2,1,1]$.
	
\end{Exam}

For convenience, if $ x=(x_1,x_2,\cdots,x_n)\in \mathrm{Seq}_n (\Gamma) $, set
\begin{equation*}
	\begin{aligned}
		{x}^-=&(x_1,x_2,\cdots,x_{n-1}, x_n,-x_n,-x_{n-1},\cdots,-x_2,-x_1),\\
		{}^-{x}=&(-x_n,-x_{n-1},\cdots, -x_2,-x_1,x_1,x_2,\cdots, x_{n-1}, x_n),
	\end{aligned}
\end{equation*}
and we call ${\bf p}(x^-)$ the {\it doubled RS shape} of $x$.

Let $\mathfrak{S}_m$ be the symmetric group in $m$ letters. Then $\mathfrak{S}_m$ is the Weyl group of $\gl(m)$ via $$w(\epsilon_i)=\epsilon_{w(i)}\text{ for }1\leq i\leq m\text{ and }w\in\mathfrak{S}_m.$$ The simple reflections of $\mathfrak{S}_m$ are $(k,k+1)$ for $1\leq k\leq m-1$. For $w\in \mathfrak{S}_m$, we use the notation
$$
\left(\begin{array}{cccc}
   1  & 2 & \cdots & m \\
   w(1)  & w(2) & \cdots & w(m)
\end{array}\right)
$$
to denote $w$. For simplicity, we denote $w$ by the one-line notation $(w(1), w(2), \cdots, w(m))$. Moreover, for $\Phi=B_n$, $C_n$ or $D_n$, then
$$w(\epsilon_{n+1-k}) = \operatorname{sgn}(w(k))\epsilon_{n+1-|w(k)|}, \qquad 1 \le k \le n.$$

\begin{Thm}[See Thm. \ref{a-function-gl(m|n)-osp}]
For $w=(u,v)\in W_0$, we have
\[
{\bf a}(w) = 
\begin{cases} 
\sum_{i\geq 1}(i-1){\bf p}(u)_i+\sum_{j\geq 1}(j-1){\bf p}(v)_j & \text{ if } \frg=\frsl(m|n), \\
\sum_{i\geq 1}(i-1){\bf p}({}^-u)_i^{\odd}+\sum_{j\geq 1}(j-1){\bf p}({}^-v)_j^{\odd} & \text{ if } \frg=\mathfrak{osp}(2m+1|2n), \\
\sum_{i\geq 1}(i-1){\bf p}({}^-u)_i^{\ev}+\sum_{j\geq 1}(j-1){\bf p}({}^-v)_j^{\odd} & \text{ if } \frg=\mathfrak{osp}(2m|2n)\text{ for }m\geq2, \\
\sum_{j\geq 1}(j-1){\bf p}({}^-v)_j^{\odd} & \text{ if } \frg=\mathfrak{osp}(2|2n).
\end{cases}
\]
\end{Thm}

\begin{Exam}\label{a-w}
    Let $\frg=\mathfrak{osp}(11|8)$ and $w=(\underbrace{3,4,-1,5,2}_u,\underbrace{3,-1, -4, 2}_v)$. We have
    \[
\scalebox{0.85}{$Y({}^-u)$}={\tiny\begin{tikzpicture}[scale=\domscale+0.1,baseline=-19pt]
			\hobox{0}{0}{-5}
			\hobox{1}{0}{-4}
                \hobox{2}{0}{-3}
                \hobox{3}{0}{-1}
                \hobox{4}{0}{2}
                \hobox{5}{0}{5}
			\hobox{0}{1}{-2}
			\hobox{1}{1}{1}
                \hobox{2}{1}{3}
                \hobox{3}{1}{4}
	\end{tikzpicture}}\hspace{5em}
 \scalebox{0.85}{$Y({}^-v)$}={\tiny\begin{tikzpicture}[scale=\domscale+0.1,baseline=-19pt]
			\hobox{0}{0}{-4}
			\hobox{1}{0}{-1}
            \hobox{2}{0}{2}
			\hobox{0}{1}{-3}
            \hobox{1}{1}{1}
            \hobox{2}{1}{3}
            \hobox{0}{2}{-2}
            \hobox{0}{3}{4}
	\end{tikzpicture}}.
	\]
Hence ${\bf p}({}^-u)^{\odd}=[3,2]$, ${\bf p}({}^-v)^{\odd}=[1,2,0,1]$ and ${\bf a}(w)={\bf a}(u)+{\bf a}(v)=2+5=7$.
    
\end{Exam}

For type I basic classical Lie superalgebras with the distinguished Borel subalgebra, based on the above result for $\mathbf{a}$-functions, we can deduce an efficient algorithm for the GK dimensions of simple highest weight modules.

\begin{Thm}[See Theorem \ref{GK-dim-type I}]
Let $\frg$ be of type I with the distinguished Borel subalgebra, and let $\Lambda=(\lambda_1,\ldots,\lambda_m\mid\mu_1,\ldots,\mu_n)$ be integral. Then
$$
\scalebox{0.8}{$
\text{GK}\dim \widetilde{L}(\lambda|\mu) = 
\begin{cases} 
\frac{m(m-1)}{2}+\frac{n(n-1)}{2}-\bigg( \sum_{i\geq 1}(i-1){\bf p}(\lambda^{\rho_0})_i+\sum_{j\geq 1}(j-1){\bf p}(\mu^{\rho_0})_j\bigg) & \text{ if } \frg=\frsl(m|n), \\
n^2-\sum_{i\geq 1}(i-1){\bf p}((\mu^{\rho_0})^-)_i^{\odd} & \text{ if } \frg=\mathfrak{osp}(2|2n).
\end{cases}
$}
$$
\end{Thm}

\begin{Exam}
    Let $\frg=\mathfrak{osp}(2|8)$ and $\Lambda^{\rho_0}:=\Lambda+\rho_0=(3 |-2,1,3,-1)$. Then
 \[
 \scalebox{0.85}{$Y((\mu^{\rho_0})^-)$}={\tiny\begin{tikzpicture}[scale=\domscale+0.1,baseline=-19pt]
			\hobox{0}{0}{-3}
			\hobox{1}{0}{-1}
                \hobox{2}{0}{-1}
                \hobox{3}{0}{2}
			\hobox{0}{1}{-2}
                \hobox{1}{1}{1}
                \hobox{0}{2}{1}
                \hobox{1}{2}{3}
	\end{tikzpicture}}.
	\]
Therefore ${\bf p}((\mu^{\rho_0})^-)^{\odd}=[2,1,1]$, we have
$$\text{GK}\dim \widetilde{L}(\lambda|\mu)=4^2-(0\cdot2+1\cdot1+2\cdot1)=13.$$

\end{Exam}


\subsection{Organization}
This paper is organized as follows. Section \ref{pre} recalls the basic definitions and notation, including root systems, induced modules, associated varieties, and Lusztig's ${\bf a}$-function. Section \ref{section-to-GK-on-HWM} establishes the relations of GK dimensions for type I Lie superalgebras to those of their even subalgebras. Section \ref{section-a-fun} presents the formulas for Lusztig's ${\bf a}$-function for classical Weyl groups and $G_2$. Section \ref{section-GK-formula} determines the GK dimensions of simple highest weight modules over the even subalgebras. Finally, Section \ref{section-GK-basic} establishes the GK dimension formulas for simple highest weight modules over basic classical Lie superalgebras. Section \ref{sec:BD-finite-dimensional} presents the criterion for finite-dimensionality  of irreducible highest weight modules over type II basic classical Lie superalgebras.


\section*{Acknowledgments}
The author thanks Chih-Whi Chen for discussions on Kac modules and GK dimension, and Li Luo for pointing out errors in an earlier version. The author thanks a previous anonymous referee for providing some ideas on the GK dimension of highest weight modules over type II basic classical Lie superalgebra.

\section*{Generative AI Usage}
ChatGPT (OpenAI) was not involved in the first arXiv version of the paper.  Compared to the first version, this revised version includes the formulas for the GK dimension of the highest weight modules over type II Lie superalgebras. The proof strategy was suggested by an anonymous referee: {\it Note that in the type II case there are usually two different distinguished triangular decompositions ($\mathfrak{g}_{0}$ has two or three components and the corresponding set of simple roots may include the set of simple roots for each of the components). Both $\dim \mathfrak{g}(2)$ and the highest weight of a simple module depend on the choice of triangular decomposition. It would therefore be interesting to compare the results obtained from the two choices}. This led to the final version. In preparing this revised version, ChatGPT was used primarily to check minor mathematical or typographical details. The mathematical ideas and conceptual framework were developed by the author, who verified the AI-assisted content and takes full responsibility for the manuscript.

\section{Preliminaries}\label{pre}
In this section,  we give full preliminaries on basic classical Lie superalgebras, induced modules, GK dimension and Lusztig's {\bf a}-function. See \cite{Mu12,BX1,KL,lusztig1985cellsI,FSS2000,Geck} for more details.
\subsection{Definitions and Notations}
\begin{itemize}
    \item $\dot{\frg}$ : a finite-dimensional complex reductive Lie algebra.

    \item $\frg=\frg_0\oplus\frg_1$ : a basic classical Lie superalgebra in the scope of \eqref{classical-type}.

    \item $\frg=\frh\oplus \sum_{\alpha\in \Delta}\frg^{\alpha}$ : root space decomposition with Cartan subalgebra $\frh$, root system $\Delta$ and root space $\frg^{\alpha}$.

    \item $\frg=\frn^-\oplus\frh\oplus\frn^+$ : triangular decomposition of $\frg$ with standard Borel subalgebra $\frb=\frh\oplus\frn^+$.

    \item $\Delta=\Delta_0\bigcup\Delta_1$ : disjoint union of even roots $\Delta_0$ and odd roots $\Delta_1$ of $\frg$.

    \item $\Phi=\Phi^-\bigcup\Phi^+$ : disjoint union of negative roots $\Phi^-$ and positive roots $\Phi^+$ of $\dot{\frg}$.

    \item $\Phi_X^+$ : the positive roots of the ordinary Lie algebra of type $X\in\{A,B,C,D,G\}$.

    \item $\rho$ : the half graded sum of positive roots $\rho_0-\rho_1$.

    \item $\Lambda^{\rho_0}$ : $\Lambda+\rho_0$.

    \item $\widetilde{L}(\Lambda)$ : simple highest weight $\frg$-module with highest weight $\Lambda=(\lambda|\mu)$.

    \item $L_{\frg_0}(\Lambda)$ : simple highest weight $\frg_0$-module.

    \item $L_0(\Lambda)$ : simple highest weight $\frg(0)$-module, where $\frg=\oplus_{i\in\mathbb{Z}}\frg(i)$. 

    \item $L(\lambda)$ (resp. $L(\mu)$) : the simple highest weight module of the semisimple factor $\frg_0'$ (resp. $\frg_0''$), where $\frg_0=\frg_0'\oplus\frg_0''\oplus\mathfrak z$ and $\mathfrak z$ is the center; a zero factor is allowed.

    \item $L_Y(\Lambda)$ : simple highest weight $Y$-module, where $Y$ is of Lie type $A$, $B$ or $D$.

    \item $W=W_0$ : the Weyl group of $\frg_0$; $W_X$ is the Weyl group of the ordinary Lie algebra of type $X\in\{A,B,C,D,G\}$.

    \item $\Pi_X$ : the simple system of the ordinary Lie algebra of type $X\in\{A,B,C,D,G\}$.

    \item Subscripts $0$ and $1$ denote the even and odd parts, respectively.

    \item $\Delta_{[\mu]}$ : $\{\alpha\in\Delta_0|(\mu,\alpha^{\vee})\in \mathbb{Z}\}$, defined in Section \ref{section-to-GK-on-HWM}.

    \item $\Pi_{[\mu]}$ : the simple system of $\Delta_{[\mu]}$.

    \item $W_{[\mu]}: \{w\in W | w\mu-\mu\in\mathbb{Z}\Delta_0\}$, subgroup of $W$ defined in Section \ref{section-to-GK-on-HWM}.

    \item $W_{[\mu]}^I:\{w\in W_{[\mu]}  | \ell(ws_{\alpha})=\ell(w)+1\text{ for all }\alpha\in I \}$, where $I:=\{\alpha\in\Pi_{[\mu]} | (\mu,\alpha^{\vee})=0\}$ and $\ell$ is the length function on $W_{[\mu]}$.

    \item $\Delta_z$ : the orthogonal subsystem defined in Section \ref{The-general-case}.

    \item $T^1_z:=\{i\leq m | \Lambda_i\in z+\mathbb{Z}\}$ for $z\in\mathbb{C}$ defined in Section \ref{The-general-case}.
    
    \item $T^2_z:=\{j\leq n | \Lambda_{m+j}\in z+\mathbb{Z}\}$ for $z\in\mathbb{C}$ defined in Section \ref{The-general-case}.

    \item $F_Y(p\|q\|s)$ : $F_Y(p)+F_Y(q)+F_Y(s)$, where the function $F_Y$ is defined in Definition \ref{F_X} with $Y\in\{A,B,D\}$.

\end{itemize}

\subsection{Basic classical Lie superalgebras}
From now on, assume that $\frg$ is one of the following basic classical Lie superalgebras: 
\begin{align}\label{classical-type}
    &A(m-1,n-1)\text{ with }m\neq n,\quad B(m,n),\quad C(n+1),
       \quad D(m,n),\notag\\
    &G(3),\quad F(4),\quad D(2,1;\alpha)\text{ with }\alpha\neq0,-1.
\end{align}

Let $\Delta$ be a root system for a basic classical Lie superalgebra $\frg$ with a given Cartan subalgebra $\frh$. Given a positive system $\Delta^+$, denote the simple system by $\Pi$. Similarly we denote by $\Delta^-$ the corresponding set of negative roots.  Set $\Delta^+_i=\Delta^+\cap \Delta_i$ and $\Delta^-_i=\Delta^-\cap \Delta_i$ for $i\in \mathbb{Z}_2=\{0,1\}$. Then we have $$\Delta^+=\Delta^+_0\cup\Delta^+_1.$$

Define $$\frn^+=\bigoplus_{\alpha\in\Delta^+}\frg^{\alpha},\qquad\frn^-=\bigoplus_{\alpha\in\Delta^-}\frg^{\alpha}$$
for root space $\frg^{\alpha}=\{g\in\frg | [h,g]=\alpha(h)g, \forall h\in\frh\}$. Then we obtain a triangular decomposition $$\frg=\frn^-\oplus\frh\oplus\frn^+$$ of $\frg$ and denote by $\frb=\frh\oplus\frn^+$ the {\it standard Borel subalgebra} of $\frg$ (corresponding to $\Delta^+$). Set $\frb=\frb_0\oplus\frb_1$, where $\frb_i=\frb\cap\frg_i$ for $i\in\mathbb{Z}_2$. For the type I formulas below, we use the distinguished Borel subalgebra because of their natural relation to Kac modules.

Given positive system $\Delta^+$, the half graded sum of positive roots $\rho$ is defined by
$$\rho=\rho_0-\rho_1=\frac{1}{2}\sum_{\alpha\in\Delta_0^+}\alpha-\frac{1}{2}\sum_{\beta\in\Delta_1^+}\beta,$$
for the chosen positive roots, which corresponds to the specific choices of positive roots described later in Sections \ref{slmn}--\ref{notation-G3}. Then we have the following table.

	\begin{tabular}{|c|c|}	
		\hline
		$\frg$ & $\rho_0 \:(\text{resp. }\rho)$  \\ 
		\hline 

		$\mathfrak{sl}(m|n)$ &   \makecell[l]{\rule{0pt}{15pt}\vspace{1em}$\rho_0=\frac{1}{2}(m-1,\cdots,1-m|n-1,\cdots,1-n)$ \\ \vspace{0.5em}$\rho=\frac{1}{2}(m-n-1,\cdots,-m-n+1|m+n-1,\cdots,m-n+1)$} \\ \hline

     $\mathfrak{osp}(2|2n)$ & \makecell[l]{\rule{0pt}{15pt}\vspace{1em}$\rho_0=(0|n,n-1,\cdots,1)$ \\  \vspace{0.5em}$\rho=(-n|n,n-1,\cdots,1)$} \\ \hline

		$\mathfrak{osp}(2m+1|2n)$ &   \makecell[l]{\rule{0pt}{15pt}\vspace{1em}$\rho_0=(m-\frac{1}{2},\cdots,\frac{1}{2}|n,\cdots,1)$ \\  \vspace{0.5em}$\rho=(m-\frac{1}{2},\cdots,\frac{1}{2}|n-m-\frac{1}{2},\cdots,-m+\frac{1}{2})$} \\ \hline

        $\mathfrak{osp}(2m|2n)$ &   \makecell[l]{\rule{0pt}{15pt}\vspace{1em}$\rho_0=(m-1,\cdots,0|n,\cdots,1)$ \\  \vspace{0.5em}$\rho=(m-1,\cdots,0|n-m,\cdots,-m+1)$}\\ \hline

         $G(3)$ &   \makecell[l]{\rule{0pt}{15pt}\vspace{1em}$\rho_0=(1|\frac{1}{2},\frac{3}{2},-\frac{3}{2})=(1|2,3)$ \\  \vspace{0.5em}$\rho=(-\frac{5}{2}|\frac{1}{2},\frac{3}{2},-\frac{3}{2})=(-\frac{5}{2}|2,3)$}\\ \hline
\end{tabular}

\subsubsection{Special linear Lie superalgebra}\label{slmn}
Let $\mathbb{C}^{m|n}$ denote the complex superspace of dimension $m|n$ with the standard basis $e_{i}, 1\leq i\leq m+n$, such that $\{e_1, \dots e_m\}$ is a basis for the even subspace $\mathbb{C}^{m|0}$ and $\{e_{m+1}, \dots e_{m+n}\}$ is a basis for the odd subspace $\mathbb{C}^{0|n}$. As a complex superspace, the {\it general linear Lie superalgebra} $\mathfrak{gl}(m|n)$ consists of $(m+n)\times (m+n)$-block matrices with respect to the standard basis, i.e., 
\begin{align}\label{gl-mn}
     \mathfrak{gl}(m|n):= \Big\{ \begin{pmatrix}
 A & B\\ 
 C & D
\end{pmatrix} \Bigg| A\in \mathfrak{M}_{m,m}, B\in \mathfrak{M}_{m, n}, C\in \mathfrak{M}_{n,m}, D\in \mathfrak{M}_{n, n} \Big\}, 
\end{align}

 where $\mathfrak{M}_{p,q}$ denotes the complex space of $(p\times q)$-matrices for any positive integers $p,q$. The even subalgebra $\mathfrak{gl}(m|n)_0\cong\mathfrak{gl}(m) \oplus \mathfrak{gl}(n)$ consists of matrices with $B$ and $C$ being zero blocks, while the odd subspace $\mathfrak{gl}(m|n)_1$ consists of matrices with $A$ and $D$ being zero blocks. The Lie bracket is defined by $[X,Y]= XY-(-1)^{|X||Y|}YX$ for any homogeneous $X,Y\in \mathfrak{gl}(m|n)$.

 For each element $g\in\mathfrak{gl}(m|n)$, define the {\it supertrace} as
 $$\text{str}(g):=\text{tr}(A)-\text{tr}(D),$$
where $\text{tr}(M)$ denotes the trace of square matrix $M$. Set
$$\frsl(m|n):=\{g\in\mathfrak{gl}(m|n) | \text{str}(g)=0\}$$
to be a subalgebra of $\mathfrak{gl}(m|n)$ and call it {\it special linear Lie superalgebra}.

Let $\widetilde{\mathfrak h}$ be the Cartan subalgebra of $\mathfrak{gl}(m|n)$ consisting of diagonal matrices,
and let
\[
\mathfrak h=\{H\in\widetilde{\mathfrak h}\mid \operatorname{str}(H)=0\}
\]
be the Cartan subalgebra of $\mathfrak{sl}(m|n)$.

Let $\epsilon_i\in\widetilde{\mathfrak h}^*$ be defined by $\epsilon_i(E_{jj})=\delta_{ij}$, where $ 1\leq i,j\leq m+n$. For convenience, we write $\delta_\mu = \epsilon_{m+\mu}$ for $1 \le \mu \le n$. We use the same notation for their restrictions to $\mathfrak h$. These restrictions satisfy
\[
\sum_{i=1}^m\epsilon_i
-
\sum_{\mu=1}^n\delta_\mu=0.
\]
Consequently,
\[
\mathfrak h^*
\simeq
\widetilde{\mathfrak h}^*
\big/
\mathbb C
\left(
\sum_{i=1}^m\epsilon_i
-
\sum_{\mu=1}^n\delta_\mu
\right).
\] The set of positive roots relative to $\mathfrak{b}$ is given by $\Delta^+ = \{\epsilon_i - \epsilon_j \mid 1 \le i < j \le m+n\}$.  The sets of even and odd positive roots are denoted by $\Delta_0^+$ and $\Delta_1^+$, respectively, where
\begin{align*}
 \Delta^+_0 &= \{\epsilon_i - \epsilon_j \mid 1 \leq i < j \leq m\} \cup \{\delta_{\mu} - \delta_{\nu} \mid 1 \leq \mu < \nu \leq n\},\\
 \Delta^+_1 &= \{\epsilon_i - \delta_{\mu} \mid 1 \leq i \leq m, 1\leq \mu\leq n\}.
\end{align*}
Equip $\widetilde{\mathfrak h}^*$ with the
symmetric bilinear form determined by
\[
(\epsilon_i,\epsilon_j)=\delta_{ij},
\qquad
(\delta_\mu,\delta_\nu)=-\delta_{\mu\nu},
\qquad
(\epsilon_i,\delta_\mu)=0.
\]
for $1 \le i,j \le m$ and $1 \le \mu,\nu \le n$. Put
\[
\tau=
\sum_{i=1}^m\epsilon_i
-
\sum_{\mu=1}^n\delta_\mu.
\]
Since $m\neq n$, we have
$(\tau,\tau)=m-n\neq0$.
Thus $\mathfrak h^*$ can be identified with
$\tau^\perp$ via the orthogonal projection
\[
v\longmapsto
v-\frac{(v,\tau)}{m-n}\tau.
\]
The restriction of the above bilinear form
to $\tau^\perp$ gives the bilinear form on
$\mathfrak h^*$.

Since $(\tau,\alpha)=0$ for every
$\alpha\in\Delta$, all pairings involving
roots can be computed using the ambient
bilinear form on $\widetilde{\mathfrak h}^*$. We have
\[ \rho=\rho_0-\rho_1=\frac{1}{2}\sum_{i=1}^m (m-n-2i+1)\epsilon_i+ \frac{1}{2}\sum_{\mu=1}^n(m+n-2\mu+1)\delta_{\mu}.  \]
It is sometimes more convenient to use
$$\partial=\sum_{i=1}^m(m-i)\epsilon_i+\sum_{j=1}^n(1-j)\delta_j$$
since the coefficients of $\partial$ are integers. The difference $\rho-\partial$ is orthogonal to all roots. The difference $\rho-\rho_0$ is orthogonal to all even roots.

The Weyl group $W$ of $\frsl(m|n)$, which is by definition the Weyl group $W_0$ of $\frg_0=\frsl(m)\oplus\frsl(n)\oplus\mathbb C$, is isomorphic to $\mathfrak{S}_m\times\mathfrak{S}_n$.

\subsubsection{Lie superalgebra $\mathfrak{osp}(2|2n)$}
Set $$\mathfrak{h}^* = \mathbb{C}\epsilon_1 \oplus \bigoplus_{r=1}^n \mathbb{C}\delta_r, \quad (\epsilon_i, \epsilon_j) = \delta_{ij}, \quad (\delta_r, \delta_t) = -\delta_{rt}, \quad (\epsilon_i, \delta_r) = 0.$$ Define $(\Lambda,\alpha^{\vee})=2(\Lambda,\alpha)/(\alpha,\alpha)$.

The even (odd) positive system is
\begin{align*}
  \Delta_0^+&= \{\delta_k\pm\delta_l \mid 1\leq k<l\leq n\}\cup\{2\delta_k \mid 1\leq k\leq n\},\\
\Delta_1^+&=\{\epsilon_1\pm\delta_k\mid 1\leq k\leq n\}.  
\end{align*}

The Weyl group $W$ of $\mathfrak{osp}(2|2n)$, which is by definition the Weyl group $W_0$ of $\frg_0=\frso(2)\oplus\frsp(2n)$, is isomorphic to $\{e\}\times W_C$.

\subsubsection{Lie superalgebra $\mathfrak{osp}(k|2n)$ for $k\geq3$}
For $k=2m$ or $2m+1$, set $$\mathfrak{h}^* = \bigoplus_{i=1}^m \mathbb{C}\epsilon_i \oplus \bigoplus_{r=1}^n \mathbb{C}\delta_r\quad (\epsilon_i, \epsilon_j) = \delta_{ij}, \quad (\delta_r, \delta_t) = -\delta_{rt}, \quad (\epsilon_i, \delta_r) = 0.$$ 
The Lie superalgebra considered here is of type $B(m,n)$ (resp. $D(m,n)$) if $k$ is odd (resp. even). The even (odd) positive system is
\begin{equation*}
\Delta_{0}^{+} =
\begin{cases}
\{\epsilon_i \pm \epsilon_j,\, \delta_r \pm \delta_t \mid 1 \le i < j \le m,\, 1 \le r \le t \le n\}\backslash \{0\} , & \text{if } k = 2m; \\[6pt]
\{\epsilon_i \pm \epsilon_j,\, \epsilon_l,\, \delta_r \pm \delta_t \mid 1 \le i < j \le m,\, 1 \le l \le m,\, 1 \le r \le t \le n\}\backslash \{0\}, & \text{if } k = 2m+1,
\end{cases}
\end{equation*}
and
\begin{equation*}
\Delta_{1}^{+} =
\begin{cases}
\{\delta_r \pm \epsilon_i \mid 1 \le i \le m,\, 1 \le r \le n\}, & \text{if } k = 2m; \\[6pt]
\{\delta_r,\, \delta_r \pm \epsilon_i \mid 1 \le i \le m,\, 1 \le r \le n\}, & \text{if } k = 2m+1.
\end{cases}
\end{equation*}

Set
\begin{equation*}
\Delta(0) =
\begin{cases}
\{\epsilon_i \pm \epsilon_j,\, \delta_r - \delta_t \mid 1 \le i < j \le m,\, 1 \le r < t \le n\}, & \text{if } k = 2m; \\[6pt]
\{\epsilon_i \pm \epsilon_j,\, \epsilon_l,\, \delta_r - \delta_t \mid 1 \le i < j \le m,\, 1 \le l \le m,\, 1 \le r < t \le n\}, & \text{if } k = 2m+1,
\end{cases}
\end{equation*}
and
\begin{equation*}
\Delta(2) = \{\delta_r + \delta_t \mid 1 \le r \le t \le n\}.
\end{equation*}

Then $\mathfrak{g}$ admits a $\mathbb{Z}_2$-consistent $\mathbb{Z}$-grading $\mathfrak{g} = \bigoplus_{i=-2}^2 \mathfrak{g}(i)$ such that $\mathfrak{g}_{0} = \mathfrak{g}(-2) \oplus \mathfrak{g}(0) \oplus \mathfrak{g}(2)$ and $\mathfrak{g}_{1} = \mathfrak{g}(-1) \oplus \mathfrak{g}(1)$.
Here $\mathfrak{g}(0) \simeq \mathfrak{o}_k \oplus \mathfrak{gl}_n$ is spanned by $\mathfrak{h}$ and $\mathfrak{g}^{\pm\alpha}$ ($\alpha \in \Delta(0)$); $\mathfrak{g}(\pm 1)$ is spanned by $\mathfrak{g}^{\pm\alpha}$ ($\alpha \in \Delta_1^+$); and $\mathfrak{g}(\pm 2)$ is spanned by $\mathfrak{g}^{\pm\alpha}$ ($\alpha \in \Delta(2)$). There is a subalgebra $\mathfrak{u} = \mathfrak{g}(1) \oplus \mathfrak{g}(2)$ with $\mathfrak{u}_{0} = \mathfrak{g}(2)$ and $\mathfrak{u}_{1} = \mathfrak{g}(1)$.

\subsubsection{Lie superalgebra $G(3)$}\label{notation-G3}

The Lie superalgebra $\frg = G(3)$ is a 31-dimensional exceptional Lie superalgebra of
defect 1. We have $\frg_0 = G_2\oplus A_1$, where $G_2$ is the exceptional Lie algebra, and an
irreducible $\frg_0$-module $\frg_1$ that is isomorphic to ${\bf \underline{7}}\otimes {\bf \underline{2}}$, where ${\bf \underline{7}}$ is the 7-dimensional representation of $G_2$ and ${\bf \underline{2}}$ is the 2-dimensional representation of $\mathfrak{sl}_2$. $\frg_0$ has dimension 17 and rank 3, $\frg_1$ has dimension 14.

 To describe the roots for $G_2$ and $\frg$, we adopt notations in \cite{ML,CW22}.  We introduce $\epsilon_1,\epsilon_2,\epsilon_3$ which satisfy the linear relation
 $$\epsilon_1+\epsilon_2+\epsilon_3=0.$$
Bilinear form $(-,-)$ on $$X:=\mathbb{Z}\delta\oplus\mathbb{Z}\epsilon_1\oplus\mathbb{Z}\epsilon_2$$ is given by
$$(\delta,\epsilon_i)=(\epsilon_i,\delta)=0,\qquad(\epsilon_i,\epsilon_i)=-(\delta,\delta)=-2(\epsilon_i,\epsilon_j)=2\quad \text{for }1\leq i\neq j\leq 3.$$
We choose the simple system $\Pi=\{\alpha_1,\alpha_2,\alpha_3\}$ for $G(3)$, where
\[
\alpha_1=\epsilon_2-\epsilon_1, \quad 
\alpha_2=\epsilon_1, \quad 
\alpha_3=\epsilon_3+\delta.
\]
The Dynkin diagram associated to $\Pi$ is depicted as follows:
  \begin{center}
$\dynkin[labels={\alpha_3,\alpha_2,\alpha_1},extended,affine mark=t,reverse arrows]G2$
  \end{center}
The root system of $G(3)$ is a union of even and odd roots: 
\[
\Delta=\Delta_{0}\cup \Delta_{1}.
\]
The positive roots are $\Delta^+=\Delta_{0}^+\cup\Delta_{1}^+$, where
\[
\Delta_{0}^+=\{2\delta, \epsilon_1,\epsilon_2,-\epsilon_3,\epsilon_2-\epsilon_1,\epsilon_1-\epsilon_3,\epsilon_2-\epsilon_3\}, 
\quad
\Delta_{1}^+=\{\delta, \delta\pm \epsilon_i \mid 1\le i\le 3\}.
\]
We have the following subsets of positive roots:
\begin{align*}
    \Delta_{1,\otimes}^+&:=\{\alpha\in\Delta_1^+ | (\alpha,\alpha)=0\}=\{\delta\pm \epsilon_i \mid 1\le i\le 3\}, \\
\Delta_{1,\bullet}^+&:=\Delta_1^+\setminus\Delta_{1,\otimes}^+ = \{\delta\}, \\ 
\Delta_{0,\circ}^+&:=\{\alpha\in\Delta_0^+ | \frac{1}{2}\alpha\notin\Delta_{1,\bullet}^+\}=\{\epsilon_1,\epsilon_2,-\epsilon_3,\epsilon_2-\epsilon_1,\epsilon_1-\epsilon_3,\epsilon_2-\epsilon_3\}.
\end{align*}

Up to $W$ equivalence, there are four systems of simple roots for $G(3)$ given by
\begin{align*}
&\Pi_1 = \{ \epsilon_3 + \delta,\; 
        \epsilon_1,\;
         \epsilon_2 - \epsilon_1\}, \\
&\Pi_2 = \{-\epsilon_3 - \delta,\;
        -\epsilon_2 + \delta,\;
        \epsilon_2 - \epsilon_1\},\\
&\Pi_3 = \{\epsilon_1,\;
         \epsilon_2 - \delta,\;
         -\epsilon_1 + \delta\}, \\
&\Pi_4 = \{\delta,\;
          \epsilon_1 - \delta,\;
          \epsilon_2 - \epsilon_1\}.
\end{align*}

The latter three simple systems can be obtained from the standard simple system $\Pi_1$ by
applying odd reflections associated with isotropic simple roots (see \cite{PS94}).

One computes that the Weyl vector for $G(3)$ is
\[
\rho=-\frac{5}{2}\delta+\frac{1}{2}\epsilon_1+\frac{3}{2}\epsilon_2-\frac{3}{2}\epsilon_3=-\frac{5}{2}\delta+2\epsilon_1+3\epsilon_2.
\]

Note that $\{\alpha_1,\alpha_2\}$ forms a simple system of $G_2$. Denote by $\omega_1$ and $\omega_2$ (resp. $\omega'_1$) the corresponding fundamental weights of $G_2$ (resp. $A_1$). We have
\[
\omega'_1=\delta,\qquad  \omega_1=\epsilon_1+2\epsilon_2, \qquad \omega_2=\epsilon_1+\epsilon_2.
\]

Therefore, we can identify $X$ with the weight lattice of $\mathfrak{g}$, and we have
\[
X=\mathbb{Z}\omega'_1\oplus X_2,
\]
where
\[
X_2=\mathbb{Z}\omega_1\oplus \mathbb{Z}\omega_2
\]
is the weight lattice of $G_2$.

Write
\[
s_0=s_{2\delta},\qquad  s_1=s_{\alpha_1},\qquad s_2=s_{\alpha_2}.
\]

The Weyl group $W$ of $\mathfrak{g}$ is
\[
W=\mathfrak{S}_2\times W_G,
\]
where $W_G=\langle s_1,s_2\rangle$ denotes the Weyl group of $G_2$ and $\mathfrak{S}_2=\langle s_0\rangle$ is the Weyl group of $\mathfrak{sl}_2$.

For a basic classical Lie superalgebra $\frg$, if the representation of $\frg_0$ on $\frg_1$ is irreducible (resp. the direct sum of two irreducible representations of $\frg_0$), $\frg$ is said to be of {\it type II} (resp. {\it type I}). Clearly $A(m-1,n-1)$ and $C(n+1)$ are of type I, while $B(m,n)$, $D(m,n)$, $G(3)$, $F(4)$ and $D(2,1;\alpha)$ are of type II. The roots for the latter two exceptional types are specified in Subsections \ref{sec:F4-GK-algorithm} and \ref{sec:D21-GK-algorithm}.

Fix a triangular decomposition $$\frg=\frn^-\oplus\frh\oplus\frn^+$$ of $\frg$ and set $\frb=\frh\oplus\frn^+$. Similarly let $\dot{\frg}=\frg_0$ be a reductive Lie algebra, and fix a triangular decomposition of $\dot{\frg}$ as in \cite[A.1.6]{Mu12}. The Verma modules of $\dot{\frg}$ and $\frg$ with highest weight $\lambda'\in\mathfrak{h}^*$ are defined by 
$$M(\lambda')=U(\dot{\frg})\otimes_{U(\frb_0)}F_{\lambda'},\qquad\widetilde{M}(\lambda')=U(\frg)\otimes_{U(\frb)}V_{\lambda'},$$
where $F_{\lambda'}$ is the one-dimensional $\frb_0$-module on which $\frn_0^+$ acts trivially and $h\in\frh$ acts by $\lambda'(h)$, and $V_{\lambda'}$ is the one-dimensional $\frb$-module satisfying
\[
\frn^+V_{\lambda'}=0,\qquad
hv=\lambda'(h)v\quad(h\in\frh,\ v\in V_{\lambda'}).
\]
The module $M(\lambda')$ (resp. $\widetilde{M}(\lambda')$) has a unique simple (resp. graded-simple) quotient which we denote by $L(\lambda')$ (resp. $\widetilde{L}(\lambda')$). The module $L(\lambda')$ (resp. $\widetilde{L}(\lambda')$) with $\lambda'\in\frh^*$ is called the simple highest weight $\dot{\frg}$ (resp. $\frg$)-module.


For a simple highest weight $\frg$-module $\widetilde{L}(\Lambda)$, to calculate the GK dimension of $\widetilde{L}(\Lambda)$, denote
\[\Lambda=(\Lambda_1,\Lambda_2,\cdots, \Lambda_{m+n})=(\lambda_1,\cdots, \lambda_m \mid  \mu_1, \dots, \mu_n):=\sum_{i=1}^m \lambda_i\epsilon_i + \sum_{j=1}^n\mu_{j} \delta_{j}\in \mathfrak{h}^*, \]
Here we abbreviate $(\lambda_1,\dots, \lambda_m| \mu_1, \dots, \mu_n)$ as $(\lambda|\mu)$. Also denote 
$$\Lambda^{\rho_0}:=\Lambda+\rho_0=((\lambda+\rho_0)_1,\cdots,(\lambda+\rho_0)_m | (\mu+\rho_0)_1,\cdots,(\mu+\rho_0)_n)=(\lambda^{\rho_0}|\mu^{\rho_0}).$$


Let $\Pi_0$ be the simple system of the fixed positive even roots. We call $\Lambda'$ integral for $\frg_0$ if $(\Lambda',\alpha^\vee)\in\mathbb Z$ for every $\alpha\in\Delta_0$, and dominant integral for $\frg_0$ if $(\Lambda',\alpha^\vee)\in\mathbb Z_{\geq0}$ for $\alpha\in\Pi_0$. For more details, see \cite{CM18,Su}.


\subsection{Induced modules}

Let $\frg$ be a basic classical Lie superalgebra in \eqref{classical-type} and let $\frp$ be a subalgebra of $\frg$. Denote by $U(\frg)$ and $U(\frp)$ the corresponding universal enveloping algebras, $N$ a finitely generated $\frp$-module. The $\frg$-module $U(\frg)\otimes_{U(\frp)}N$ is called the {\it induced module} and denoted by $\text{Ind}^{\frg}_{\frp}N$.

Following \cite{Kac78}, fix the grading below and a compatible Borel subalgebra $\frb\subseteq\frp=\frg(0)\oplus\bigoplus_{i>0}\frg(i)$. The Levi subalgebra $\frg(0)$ inherits its positive roots from $\frb$. Let $L_0(\Lambda)$ be its simple highest weight module and extend it to $\frp$ by letting $\bigoplus_{i>0}\frg(i)$ act trivially.

For the Lie superalgebra of type I,  $\frg$ admits a $\mathbb{Z}$-gradation
$$\frg=\frg(-1)\oplus\frg(0)\oplus\frg(1),$$
In the special linear case, $\frg(-1)$ (resp. $\frg(1)$) is spanned by all $E_{ij}$ such that $1\leq j\leq m<i\leq m+n$ (resp. $1\leq i\leq m<j\leq m+n$). In either type I case, $\frg(0)=\frg_0$. Define the {\it Kac module} over $\frg$ by:
$$\mathcal{K}(\Lambda):=\text{Ind}^{\frg}_{\frp}L_0(\Lambda)\cong\wedge(\frg(-1))\otimes L_0(\Lambda).$$

For the Lie superalgebra of type II, $\frg$ admits a $\mathbb{Z}_2$-consistent $\mathbb{Z}$-grading $\frg=\oplus_{i=-2}^2\frg(i)$ such that $\mathfrak{g}_{0} = \mathfrak{g}(-2) \oplus \mathfrak{g}(0) \oplus \mathfrak{g}(2)$ 
and $\mathfrak{g}_{1} = \mathfrak{g}(-1) \oplus \mathfrak{g}(1)$.
Here $\mathfrak{g}(0)$ is the Levi subalgebra spanned by $\mathfrak h$ and the root spaces of degree zero; in the orthosymplectic case, $\mathfrak{g}(0)\simeq\mathfrak{o}_k\oplus\mathfrak{gl}_n$, while for $G(3)$ it is $G_2\oplus\mathfrak{gl}_1$. For $F(4)$ and $D(2,1;\alpha)$, the distinguished Levi subalgebras are $\frso(7)\oplus\mathfrak{gl}_1$ and $\frsl(2)\oplus\frsl(2)\oplus\mathfrak{gl}_1$, respectively; their positive degree-two root spaces correspond to $\delta$ and $2\epsilon_1$. The components $\mathfrak g(\pm1)$ and $\mathfrak g(\pm2)$ are spanned by the root spaces of degrees $\pm1$ and $\pm2$, respectively.
There is a subalgebra $\mathfrak{u} = \mathfrak{g}(1) \oplus \mathfrak{g}(2)$ with 
$\mathfrak{u}_{0} = \mathfrak{g}(2)$ and $\mathfrak{u}_{1} = \mathfrak{g}(1)$. 

For type II, let $\alpha$ be the simple root of $\frg_0$ omitted from the Levi subalgebra, so that $\alpha=2\delta_n$ in the orthosymplectic case. When $k=(\Lambda,\alpha^\vee)\in\mathbb Z_{\geq0}$, the induced module $\text{Ind}^{\frg}_{\frg(0)+\fru}L_0(\Lambda)$ contains the submodule
$$\mathcal{M}(\Lambda)=U(\frg)(\frg^{-\alpha})^{k+1}v_\Lambda,$$
where $v_\Lambda$ is the highest weight vector of $L_0(\Lambda)$. If $\mathcal M(\Lambda)$ is proper, the corresponding Kac module is the quotient below. Thus we set
$$\mathcal{K}(\Lambda):=\left\{\begin{array}{ll}
	\text{Ind}^{\frg}_{\frp}L_0(\Lambda), & \text { if } \frg  \text { is of type I, }\vspace{0.5em}  \\
    
    \text{Ind}^{\frg}_{\frp}L_0(\Lambda)/\mathcal{M}(\Lambda), & \text { if } \frg  \text { is of type II and }\mathcal M(\Lambda)\text{ is proper. }
\end{array}\right.$$
Whenever $\mathcal K(\Lambda)$ is defined as above, it has a unique maximal proper submodule $\mathcal I(\Lambda)$, and its unique simple factor module $\mathcal K(\Lambda)/\mathcal I(\Lambda)$ is isomorphic to $\widetilde L(\Lambda)$ (whether or not $L_0(\Lambda)$ is finite-dimensional in type I; see \cite{CM21}). For every type II highest weight, the module $\mathcal{N}(\Lambda)=\text{Ind}^{\frg}_{\frp}L_0(\Lambda)\cong U(\frg(-2)+\frg(-1))\otimes L_0(\Lambda) $ is defined without these restrictions; see \cite{CL13,SZ12}. It is called a {\it generalized Verma module} and has unique irreducible quotient $\widetilde{L}(\Lambda)$.

Denote by $L_{\frg_0}(\Lambda)$ the simple highest weight $\frg_0$-module, then we have the following lemma.

\begin{Lem}\label{GK-dim-submodule}
 $\text{GK}\dim L_0(\Lambda)\leq\text{GK}\dim L_{\frg_0}(\Lambda)\leq \text{GK}\dim\widetilde{L}(\Lambda)$. 

\end{Lem}
\begin{proof}
Let $v$ be a highest weight vector of $\widetilde L(\Lambda)$ and put $W=U(\frg_0)v$. This is a quotient of the Verma module $M(\Lambda)$ and has simple quotient $L_{\frg_0}(\Lambda)$. Since $W\subseteq\widetilde L(\Lambda)$, this proves $\text{GK}\dim L_{\frg_0}(\Lambda)\leq\text{GK}\dim\widetilde L(\Lambda)$.

Let $v_\Lambda$ be a highest weight vector of $L_{\mathfrak g_{0}}(\Lambda)$ and put $N=U(\mathfrak g(0))v_\Lambda$. Since the simple roots of the Levi subalgebra $\mathfrak g(0)$ form a subset of those of $\mathfrak g_{0}$, $N$ is a highest weight module with simple quotient $L_0(\Lambda)$. Moreover, $U_n(\mathfrak g(0))v_\Lambda\subseteq U_n(\mathfrak g_{0})v_\Lambda$ for all $n\geq0$. Consequently,
$\operatorname{GKdim}L_0(\Lambda)\leq\operatorname{GKdim}_{U(\mathfrak g(0))}N\leq\operatorname{GKdim}L_{\mathfrak g_{0}}(\Lambda)$,
as required.

For later use, we give a direct proof as follows for type I. Assume that $\frg$ is of type I. Let $v_\Lambda$ and $\widetilde v_\Lambda$ be highest weight vectors of $M(\Lambda)$ and $\widetilde M(\Lambda)$, respectively. Extend $M(\Lambda)$ to $\frp$ by letting $\frg(1)$ act trivially. Define a map of $U(\frg)$-modules 
\begin{align*}
    \phi:U(\frg)\otimes_{U(\frp)} M(\Lambda)&\longrightarrow \widetilde{M}(\Lambda)\\
    a\otimes bv_{\Lambda}&\longmapsto ab\widetilde v_\Lambda.
\end{align*}
Then $\phi$ is well-defined and we have $U(\frg)=U(\frg(-1))\otimes U(\frp)$, so since $\frg(-1)\oplus\frn_0^-=\frn^-$ and $M(\Lambda)=U(\frn_0^-)v_{\Lambda}$ we obtain
\begin{align}\label{kac-module-iso}
    U(\frg)\otimes_{U(\frp)} M(\Lambda)&=\notag U(\frg(-1))\otimes U(\frp)\otimes_{U(\frp)} M(\Lambda)\\&=\notag U(\frg(-1))\otimes U(\frn_0^-)v_{\Lambda}\\&=U(\frn^-)\otimes V_{\Lambda}=\widetilde{M}(\Lambda).
\end{align}
This implies that $\phi$ is an isomorphism.

Suppose that $L_{\frg_0}(\tau)$ is a simple $U(\frp)$-submodule of $\widetilde{L}(\Lambda)$. Then there is a nonzero homomorphism of $U(\frg)$-modules from $ U(\frg)\otimes_{U(\frp)} M(\tau)$ to $\widetilde{L}(\Lambda)$. Since $\widetilde{L}(\Lambda)$ is simple, $\widetilde{L}(\Lambda)$ is a factor module of $ U(\frg)\otimes_{U(\frp)} M(\tau)$, and (\ref{kac-module-iso}) implies that $\widetilde{L}(\Lambda)$ is also a factor module of $U(\frg)\otimes_{U(\frp)} M(\tau)\cong \widetilde{M}(\tau)$. However $\widetilde{M}(\tau)$ has a unique simple quotient $\widetilde{L}(\tau)$, therefore $\widetilde{L}(\tau)=\widetilde{L}(\Lambda)$ and it follows that $\tau=\Lambda$. In this case $L_0(\Lambda)= L_{\frg_0}(\Lambda)$.

\end{proof}

\begin{Rem}
    For a factor module $L$ of module $M$, one can also easily check that $\text{GK}\dim L\leq \text{GK}\dim M$ since $M\twoheadrightarrow L$ is surjective.
\end{Rem}

\subsection{GK dimensions and associated varieties}\label{section-GK-dimension}
\begin{Def}
  Let $A$ be an algebra generated by a finite-dimensional subspace $K$. Let $\{K_n\}$ be the linear span of all products of length at most $n$ in $K$. Then $\text{GK}\dim(A)$ is defined by: $$\text{GK}\dim(A)=\varlimsup\limits_{n\rightarrow \infty}\frac{\log\dim(K_n)}{\log n}.$$
\end{Def}
Moreover, the definition of the GK dimension can be extended to a left $A$-module, so we have the following definition.

\begin{Def}
   Let $A$ be an algebra generated by a finite-dimensional subspace $K$, $M$ the left $A$-module generated by a finite-dimensional subspace $M_0$. Then $\text{GK}\dim(M)$ is defined by: $$\text{GK}\dim(M)=\varlimsup\limits_{n\rightarrow \infty}\frac{\log\dim(K_nM_0)}{\log n}.$$  
\end{Def}

It is well known that the above two definitions are independent of the choice of $K$ and $M_0$. 

Let $N=\bigoplus_{m\geq 0}N_m$ be a nonzero finitely generated graded $S(\frg_0)$-module with $N(n)=\bigoplus_{m=0}^nN_m$. For $n\gg 0$, we have
\begin{align}\label{hilbert-poly}
  \dim N(n)=a_d\binom{n}{d}+a_{d-1}\binom{n}{d-1}+\cdots+a_0
\end{align}
for suitable constants $a_0,a_1,\cdots,a_d$ with $a_d>0$. Then we set $d(N)=d$ and $e(N)=a_d$. Equip $U(\frg)$ with a standard filtration and denote associated graded ring by $\text{gr}U(\frg)$. Let $M$ be a nonzero finitely generated $U(\frg)$-module and equip $M$ with a good filtration $\{M_n\}_{n\geq 0}$. Given that $N=\text{gr}M$ is finitely generated over $\text{gr}U(\frg)$ and thus over $S(\frg_0)$, then we set $d(M) = d(N)$ and
$e(M) = e(N)$. It is not hard to show that $d(M)$ and $e(M)$ are independent of the good filtration
and that $d(M)$ is the GK dimension of $M$ calculated either as a $U(\frg)$-module or as
a $U(\frg_0)$-module. See \cite[Chapter 7]{MR2000} for more details. The function in \eqref{hilbert-poly} is called {\it Hilbert polynomial}.   Then we have the following definition.
\begin{Def}\label{GK-dimension}
Let $M$ be a finitely generated $U(\frg)$-module. Fix a finite-dimensional generating subspace $M'$ of $M$ and let $U_n(\frg_0)$ be the standard filtration of $U(\frg_0)$. Then
$$\text{GK}\dim(M):=\varlimsup\limits_{n\rightarrow \infty}\frac{\log\dim(U_n(\frg_0)M')}{\log n}.$$
\end{Def}

For a Lie subalgebra $\frk\subseteq\frg_0$ and a finitely generated $U(\frk)$-module $M$, its associated variety is
\[
\AV_{\frk}M:=\operatorname{Supp}_{S(\frk)}\operatorname{gr}M\subseteq\frk^*,
\]
where $\operatorname{gr}M$ is taken with respect to a good filtration. This variety is independent of the good filtration and has dimension $\GKD_{U(\frk)}M$. For a finitely generated $\frg$-module, $\AV_{\frg_0}M$ refers to the associated variety of its restriction to $\frg_0$. It also follows from the Hilbert polynomial that a nonzero finitely generated $\mathfrak{g}$-module has GK dimension zero if and only if it is finite-dimensional.

\begin{Cor}[{\cite[Lemma 4.1]{BB}}]\label{cor:finite-tensor}
Let $\dot{\frg}$ be a finite-dimensional Lie algebra, $M$ a finitely generated $U(\dot{\frg})$-module, and $F$ a nonzero finite-dimensional $\dot{\frg}$-module. Then
\[
\AV_{\dot{\frg}}(M\otimes F)=\AV_{\dot{\frg}}M,
\qquad \GKD(M\otimes F)=\GKD M.
\]
\end{Cor}

The next proposition gives the growth of induced modules.
\begin{Prop}[{\cite{Mu06}}]\label{GK-dim-induced-mod}
 With the notations above, suppose $M=\text{Ind}^{\frg}_{\frp}N$, we have
    $$\text{GK}\dim(M)=\text{GK}\dim(N)+\dim \frg_0-\dim \frp_0.$$
\end{Prop}

The following lemma is a graded analogue of the independence of associated varieties of cyclic submodules generated by nonzero vectors; see \cite[Proposition 4.19]{Kit26}.
\begin{Lem}\label{lem:cyclic-support}
Let $\frg$ be a finite-dimensional Lie superalgebra, $V$ a simple graded $\frg$-module, and $\frk\subseteq\frg_0$ a Lie subalgebra. For any two nonzero homogeneous vectors $a,b\in V$,
\[
\AV_{\frk}U(\frk)a=\AV_{\frk}U(\frk)b.
\]
\end{Lem}
\begin{proof}
By simplicity, write $a=tb$ with $t\in U_d(\frg)$ for some $d$. The space $F=U_d(\frg)$ is finite-dimensional and satisfies $[\frk,F]\subseteq F$. Give it the action $x\cdot f=[x,f]$. Multiplication defines a $\frk$-module map
\[
F\otimes U(\frk)b\longrightarrow V,
\qquad f\otimes u b\longmapsto fub.
\]
Its image $I$ contains $U(\frk)a$. The map is a $U(\frk)$-module homomorphism because
\[
[x,f]ub+f(xub)=x(fub),\qquad x\in\frk.
\]
The tensor product is finitely generated over $U(\frk)$, and so is $I$. Since $U(\frk)$ is Noetherian, its submodule $U(\frk)a$ admits the induced good filtration. Support inclusion for a submodule and a quotient, followed by Corollary \ref{cor:finite-tensor}, gives
\[
\AV_{\frk}U(\frk)a\subseteq\AV_{\frk}I
\subseteq\AV_{\frk}(F\otimes U(\frk)b)
=\AV_{\frk}U(\frk)b.
\]
Interchanging $a$ and $b$ gives equality.
\end{proof}

\begin{Prop}\label{lem:even-submodule-support}
Under the preceding hypotheses, let $S\subseteq V$ be a nonzero finitely generated $\frg_0$-submodule contained in one parity. Then
\[
\AV_{\frg_0}V=\AV_{\frg_0}S.
\]
\end{Prop}
\begin{proof}
By PBW Theorem and the graded simplicity of $V$, the restriction of $V$ to $\mathfrak g_0$ is finitely generated. Thus there exist nonzero homogeneous vectors $v_1,\ldots,v_r\in V$ such that $V=\sum_{i=1}^r U(\mathfrak g_0)v_i$.
Choose a nonzero vector $b\in S$, which is homogeneous by assumption. Lemma \ref{lem:cyclic-support} gives
\[
\operatorname{AV}_{\mathfrak g_0}U(\mathfrak g_0)v_i=\operatorname{AV}_{\mathfrak g_0}U(\mathfrak g_0)b
\]
for every $i$. Since $U(\mathfrak g_0)b\subseteq S\subseteq V$, we have
\[
\begin{aligned}
\operatorname{AV}_{\mathfrak g_0}V
\subseteq
\bigcup_{i=1}^r
\operatorname{AV}_{\mathfrak g_0}U(\mathfrak g_0)v_i=\operatorname{AV}_{\mathfrak g_0}U(\mathfrak g_0)b\subseteq\operatorname{AV}_{\mathfrak g_0}S
\subseteq\operatorname{AV}_{\mathfrak g_0}V.
\end{aligned}
\]
\end{proof}

\subsection{Lusztig's {\bf a}-function} \label{a_function}
Let $(W,S)$ be a finite Weyl group with Coxeter matrix $\left(m_{s t}\right)_{s, t \in S}$ and length function $\ell$. Throughout this subsection the Hecke algebra has equal parameters. Following \cite{KL,lusztig1985cellsI}, we define the Hecke algebra $\bm{\mathcal{H}}$ of $(W, S)$ as follows. Let $q^{\frac{1}{2}}$ be an indeterminate and $\mathbb{Z}\left[q^{ \pm \frac{1}{2}}\right]$ be the ring of Laurent polynomials in $q^{\frac{1}{2}}$. Let $\bm{\mathcal{H}}$ be a free $\mathbb{Z}\left[q^{ \pm \frac{1}{2}}\right]$-module with formal basis $\left\{\widetilde{T}_w\right\}_{w \in W}$. The multiplication on $\bm{\mathcal{H}}$ is defined by

$$
\widetilde{T}_s \widetilde{T}_w:=\left\{\begin{array}{ll}
	\widetilde{T}_{s w}, & \text { if } \ell(s w)>\ell(w) \\
	\left(q^{\frac{1}{2}}-q^{-\frac{1}{2}}\right) \widetilde{T}_w+\widetilde{T}_{s w}, & \text { if } \ell(s w)<\ell(w)
\end{array}, \quad \forall s \in S, w \in W.\right.
$$
Then $\bm{\mathcal{H}}$ is an associative $\mathbb{Z}\left[q^{ \pm \frac{1}{2}}\right]$-algebra with unity $\widetilde{T}_e$. Here $\widetilde{T}_w=q^{-\frac{\ell(w)}{2}} T_w$ where $T_w$ is the standard basis element in \cite{KL}.

Besides the normalized standard basis $\left\{\widetilde{T}_w \mid w \in W\right\}$, we have the basis $\left\{C_w \mid w \in W\right\}$ in \cite{KL}, such that

$$
\begin{aligned}
	C_w & =\sum_{y \in W}(-1)^{\ell(w)+\ell(y)} q^{\frac{1}{2}(\ell(w)-\ell(y))} P_{y, w}\left(q^{-1}\right) \widetilde{T}_y \\
	& =\sum_{y \in W}(-1)^{\ell(w)+\ell(y)} q^{-\frac{1}{2}(\ell(w)-\ell(y))} P_{y, w}(q) \widetilde{T}_{y^{-1}}^{-1},
\end{aligned}
$$
where $P_{y, w} \in \mathbb{Z}[q], P_{y, w}=0$ unless $y \leq w, P_{w, w}=1$, and $\operatorname{deg}_q P_{y, w} \leq \frac{1}{2}(\ell(w)-\ell(y)-1)$ if $y<w$. Here $\leq$ is the Bruhat order on $W$, and $y<w$ indicates $y \leq w$ and $y \neq w$.

For any $x, y, w \in W$, let $h_{x, y, w} \in \mathbb{Z}\left[q^{ \pm \frac{1}{2}}\right]$ be such that $C_x C_y=\sum_{w \in W} h_{x, y, w} C_w$. Following G. Lusztig \cite{lusztig1985cellsI}, for each $w \in W$ we define $${\bf a}(w):=\min \left\{i \in \mathbb Z_{\geq0} \left\lvert\, q^{\frac{i}{2}} h_{x, y, w} \in \mathbb{Z}\left[q^{\frac{1}{2}}\right]\right., \forall x, y \in W\right\}.$$ Then ${\bf a}(w)$ is a well-defined nonnegative integer. We say the function ${\bf a}$ is bounded if there is $N \in \mathbb{N}$ such that ${\bf a}(w) \leq N$ for any $w \in W$.

For $x, y \in W$, we say $y \underset{\text { LR }}{\leqslant} x$ if there exist $H_1, H_2 \in \bm{\mathcal{H}}$, such that $C_y$ has nonzero coefficient in the expression of $H_1 C_x H_2$ with respect to the basis $\left\{C_w\right\}_{w \in W}$. We say $x \underset{\mathrm{LR}}{\sim} y$ if $x \underset{\mathrm{LR}}{\leqslant} y$ and $y \underset{\mathrm{LR}}{\leqslant} x$. It turns out that $\underset{\mathrm{LR}}{\leqslant}$ is a preorder on $W$, and $\underset{\mathrm{LR}}{\sim}$ is an equivalence relation on $W$. The equivalence classes are called two-sided cells. The set of two-sided cells forms a partially ordered set with respect to $\underset{\mathrm{LR}}{\leqslant}$.

We have the following standard properties of Kazhdan-Lusztig basis $\left\{C_w\right\}_{w \in W}$, two-sided cells and Lusztig's ${\bf a}$-function.

\begin{Prop}[{\cite{KL,lusztig1985cellsI,lusztig2003hecke}}]\label{a_function_property}
	Let $s \in S$ and $w, x, y \in W$. For distinct $r,t\in S$, let $w_{rt}$ denote the longest element of the parabolic subgroup generated by $r$ and $t$.
	
	\begin{enumerate}
		
		\item ${\bf a}(w)=0$ if and only if $w=e$. We have ${\bf a}(s)=1$ for any $s \in S$.
		
		\item ${\bf a}(w)={\bf a}(w^{-1})$.

		\item If $m_{r t}<\infty$ for some $r, t \in S$ with $r \neq t$, then ${\bf a}\left(w_{r t}\right)=m_{r t}$.
		
		\item If $y \underset{\mathrm{LR}}{\leqslant} x$, then ${\bf a}(y) \geq {\bf a}(x)$. If $y \underset{\mathrm{LR}}{\sim} x$, then ${\bf a}(y)={\bf a}(x)$.
		
		\item Suppose ${\bf a}$ is bounded on $W$. If ${\bf a}(y)={\bf a}(x)$ and $y \underset{\mathrm{LR}}{\leqslant} x$, then $y \underset{\mathrm{LR}}{\sim} x$.
		
		\item If $W$ is a direct product of Coxeter subgroups $W_1$ and $W_2$, then
		$${\bf a}(w)={\bf a}(w_1)+{\bf a}(w_2)$$
		for $w=(w_1,w_2)\in W_1\times W_2=W.$ \label{a-7}
	\end{enumerate}  
\end{Prop}

	\subsection{PyCox}\label{PyCox}  For a Lie algebra, if simple root $\alpha_i\in\Pi$, we simply write the simple reflection $s_{\alpha_i}$ by $s_i$. Here we adopt Geck's notations in \cite{Geck}, which can be used to compute the value of ${\bf a}(w)$ for $w\in W$.  In type $A_n$ or $G_2$, we use $[i_1-1,i_2-1,\cdots,i_k-1]$ to represent $w=s_{i_1}s_{i_2}\cdots s_{i_k}$. In type $B_n$, $C_n$ or $D_n$, we use $[n-i_1,n-i_2,\cdots,n-i_k]$ to represent $w=s_{i_1}s_{i_2}\cdots s_{i_k}$.    In PyCox the function \texttt{klcellrepelm} will return the ${\bf a}$ value of $w\in W$ and the character of the left cell to which $w$ belongs (note that the notation for characters in PyCox differs from  \cite{Ca85}).
	
	\begin{Exam}\label{exam-G3-a-fun}
		In the case of $G_2$, we use  $[1,0,1,0,1,0]$ to denote $w=s_2s_1s_2s_1s_2s_1$. Then PyCox gives:
        \begin{verbatim}
    >>> W = coxeter("G", 2)
    >>> print(klcellrepelm(W, [1,0,1,0,1,0]))
    {'size': 1, 'character': ['phi_{1,6}', 1], 'a': 6,
    'special': 'phi_{1,6}', 'index': 1, 'elms': False, 'distinv': False}
\end{verbatim}
		Thus we have ${\bf a}(w)=6$ and the corresponding special character is $\phi_{1,6}$.
	\end{Exam}
    

\section{GK dimensions of highest weight modules in type I Lie superalgebras}\label{section-to-GK-on-HWM}

For $\mu\in\frh^*$, define
	$$\Delta_{[\mu]}:=\{\alpha\in\Delta_0 | (\mu,\alpha^{\vee})\in\mathbb{Z}\},$$
	where $\alpha^{\vee}$ is the coroot of the root $\alpha\in\Delta_0$. Denote
	$$W_{[\mu]}:=\{w\in W | w\mu-\mu\in\mathbb{Z}\Delta_0\}.$$
	Then $\Delta_{[\mu]}$ is a root system and let $\Pi_{[\mu]}$ be the simple system of $\Delta_{[\mu]}$. Set $I:=\{\alpha\in\Pi_{[\mu]} | (\mu,\alpha^{\vee})=0\}$. Let $\ell$ be the length function on $W_{[\mu]}$ and 
	$$W_{[\mu]}^I:=\{w\in W_{[\mu]}  | \ell(ws_{\alpha})=\ell(w)+1\text{ for all }\alpha\in I \}.$$
When $\mu$ is integral, set $W^I:=W_{[\mu]}^I$.

In \cite{BX1}, there is a formula connecting Lusztig's {\bf a}-function and the GK dimension.
\begin{Lem}{\cite[Prop. 3.8]{BX1}}\label{GKdim-uniform1}
    Let $L(\lambda)$ be a simple highest weight module for the ordinary Lie algebra $\dot{\frg}$ with positive root system $\Phi^+$. Put $\lambda'=\lambda+\rho_0$, where $\rho_0=\frac12\sum_{\alpha\in\Phi^+}\alpha$. Suppose that $\lambda'=w_{\lambda'}\delta$ with a weight $\delta$ anti-dominant for $\Delta_{[\lambda']}$ and that $w_{\lambda'}\in W_{[\lambda']}^I$ has minimal length, where $I=\{\alpha\in\Pi_{[\lambda']}:(\delta,\alpha^\vee)=0\}$. For an arbitrary complex shifted parameter, anti-dominance is considered relative to the positive integral roots, requiring their integer coroot pairings to be nonpositive. Then
    $$\text{GK}\dim L(\lambda)=|\Phi^+|-{\bf a}_{[\lambda']}(w_{\lambda'}),$$
where ${\bf a}_{[\lambda']}$ is the {\bf a}-function on $W_{[\lambda']}$.
\end{Lem}

For the distinguished Borel subalgebra in type I, the Kac module reduces GK dimension to the simple highest weight modules of the even factors due to the following proposition. 

\begin{Prop}\label{semisimple-GK-dim}
Let $\frg$ be a type I basic classical Lie superalgebra within the scope of \eqref{classical-type}, with the distinguished Borel  subalgebra. Write
\[
\frg_0=\frg_0'\oplus\frg_0''\oplus\mathfrak z,
\]
where $\mathfrak z$ is the center and $\frg_0',\frg_0''$ are the semisimple factors, allowing a zero factor. For the highest weight $\Lambda=(\lambda|\mu)$, let $L(\lambda)$ and $L(\mu)$ be the simple highest weight modules of these factors, and let $\mathbb C_\chi$ be the one-dimensional central character determined by $\Lambda$. Thus
\[
L_{\frg_0}(\Lambda)\cong L(\lambda)\otimes L(\mu)\otimes\mathbb C_\chi.
\]
Then
\[
\GKD\bigl(L(\lambda)\otimes L(\mu)\bigr)
=\GKD L(\lambda)+\GKD L(\mu)
=\GKD\widetilde L(\Lambda).
\]
\end{Prop}
\begin{proof}
{\bf Step 1.}  First we prove that $\text{GK}\dim (L(\lambda) \otimes L(\mu))=\text{GK}\dim  L(\lambda)+\text{GK}\dim  L(\mu)$.


Let $ U(\mathfrak{g}_0') $  (resp. $ U(\mathfrak{g}_0'') $) be the universal enveloping algebra of $\mathfrak{g}_0'$ (resp. $\mathfrak{g}_0''$). Fix a finite-dimensional generating subspace $ M_{0}' \subseteq L(\lambda) $. Consider the standard filtrations on $ U(\mathfrak{g}_0') $
\[
U_0(\mathfrak{g}_0') \subseteq U_1(\mathfrak{g}_0') \subseteq U_2(\mathfrak{g}_0') \subseteq \cdots \subseteq U(\mathfrak{g}_0'),
\]
where $ U_n(\mathfrak{g}_0') $ is the subspace of $ U(\mathfrak{g}_0') $ spanned by products of at most $ n $ generators. Define filtrations on $ L(\lambda) $ by
\[
L(\lambda)_n = U_n(\mathfrak{g}_0') \cdot M_{0}'.
\]
Since $ L(\lambda) $ is finitely generated, each $ L(\lambda)_n $ is finite-dimensional, and $ \bigcup_{n=0}^\infty  L(\lambda)_n =  L(\lambda) $. Choose a finite-dimensional generating subspace $M_0''\subseteq L(\mu)$ and set
$$L(\mu)_n = U_n(\mathfrak{g}_0'') \cdot M_{0}'',$$
Each $ L(\mu)_n $ is finite-dimensional, and $ \bigcup_{n=0}^\infty  L(\mu)_n =  L(\mu) $.

Consider the  $ U(\mathfrak{g}_0' \oplus \mathfrak{g}_0'') $-module $ L(\lambda) \otimes L(\mu) $. Let $ M_0 = M_{0}' \otimes M_{0}'' $ be a finite-dimensional generating subspace of $ L(\lambda) \otimes L(\mu) $. The filtration $F_n$ on $  L(\lambda) \otimes L(\mu) $ is
\begin{align*}
    F_n&=( L(\lambda) \otimes L(\mu))_n \\&= U_n(\mathfrak{g}_0' \oplus \mathfrak{g}_0'') \cdot M_0\\&=\sum\limits_{i+j\leq n}(U_i(\mathfrak{g}_0') \otimes U_j(\mathfrak{g}_0'')) \cdot (M_{0}' \otimes M_{0}'') = \sum\limits_{i+j\leq n} L(\lambda)_i \otimes L(\mu)_j.
\end{align*}
Then we have
\begin{align}\label{fil}
    L(\lambda)_{\lfloor n/2 \rfloor} \otimes L(\mu)_{\lfloor n/2 \rfloor} \subseteq  F_n \subseteq L(\lambda)_n \otimes L(\mu)_n,
\end{align}
where $ \lfloor n/2 \rfloor $ is the largest integer $ a $ such that $ a\leq n/2$. Then (\ref{fil}) implies that
\begin{align}\label{fil-dim}
     \dim L(\lambda)_{\lfloor n/2 \rfloor} \cdot \dim L(\mu)_{\lfloor n/2 \rfloor}\leq \dim(F_n) \leq \dim L(\lambda)_n \cdot \dim L(\mu)_n.
\end{align}
If we take the natural logarithm of each part of the inequality (\ref{fil-dim}), we obtain
$$\frac{\log \dim L(\lambda)_{\lfloor n/2 \rfloor}}{\log n} + \frac{\log \dim L(\mu)_{\lfloor n/2 \rfloor}}{\log n} \le \frac{\log \dim(F_n)}{\log n} \le \frac{\log \dim L(\lambda)_n}{\log n} + \frac{\log \dim L(\mu)_n}{\log n}.$$
By \eqref{hilbert-poly}, both cumulative dimension functions are eventually polynomials with positive leading coefficients. Hence the ratios $\log\dim L(\lambda)_n/\log n$ and $\log\dim L(\mu)_n/\log n$ have limits, equal to the corresponding GK dimensions. Thus the sums of upper limits in the following calculation are justified. We have
\begin{align}\label{GK-tensor1}
    \varlimsup_{n \to \infty}\left(\frac{\log \dim L(\lambda)_{\lfloor n/2 \rfloor}}{\log n} + \frac{\log \dim L(\mu)_{\lfloor n/2 \rfloor}}{\log n}\right) &\le\notag \varlimsup_{n \to \infty}\frac{\log \dim(F_n)}{\log n} \\&=\notag\text{GK}\dim (L(\lambda) \otimes L(\mu))\\&\le\notag \varlimsup_{n \to \infty}\left(\frac{\log \dim L(\lambda)_n}{\log n} + \frac{\log \dim L(\mu)_n}{\log n}\right)\\&=\notag \varlimsup_{n \to \infty} \frac{\log\dim L(\lambda)_n}{\log n} + \varlimsup_{n \to \infty} \frac{\log\dim L(\mu)_n}{\log n}\\&=\text{GK}\dim  L(\lambda)+\text{GK}\dim  L(\mu).
\end{align}
On the other hand, let $k = \lfloor n/2 \rfloor$, we have $(n - 1)/2 \leq k \le n/2$, which implies that $\varlimsup\limits_{n \to \infty}\frac{\log k}{\log n}=1$. As $n \to \infty$, clearly $k \to \infty$. Then we have
\begin{align}\label{limit-GK}
    \varlimsup_{n \to \infty}\frac{\log \dim L(\lambda)_{k}}{\log n}&=\notag\varlimsup_{k \to \infty}\frac{\log \dim L(\lambda)_{k}}{\log k}\cdot \frac{\log k}{\log n}\\&=\notag\text{GK}\dim L(\lambda)\cdot \varlimsup_{n \to \infty}\frac{\log k}{\log n}\\&=\text{GK}\dim L(\lambda).
\end{align}
Combining (\ref{GK-tensor1}) and (\ref{limit-GK}), we can obtain 
\begin{align}\label{GK-tensor}
    \text{GK}\dim (L(\lambda) \otimes L(\mu))=\text{GK}\dim  L(\lambda)+\text{GK}\dim  L(\mu).
\end{align}

{\bf Step 2.} Next we prove that 
\begin{align}\label{prop-GK-tensor}
    \text{GK}\dim \widetilde{L}(\lambda|\mu)=\text{GK}\dim L(\lambda)+\text{GK}\dim L(\mu).
\end{align}

When $\frg$ is of type I, recall that Kac module is defined by
$$\mathcal{K}(\lambda|\mu):=U(\frg)\otimes_{U(\frg(0)+\frg(1))}L_0(\lambda|\mu)\cong \wedge(\frg(-1))\otimes L_0(\lambda|\mu),$$
where $L_0(\lambda|\mu)$ is a simple $\frg_0$-module of highest weight $(\lambda|\mu)$. Hence 
\begin{align*}
    \text{GK}\dim \widetilde{L}(\lambda|\mu)&\leq\text{GK}\dim \big(\wedge(\frg(-1))\otimes L_0(\lambda|\mu)\big)\\&=\text{GK}\dim L_0(\lambda|\mu)=\text{GK}\dim (L(\lambda) \otimes L(\mu)) \overset{(\ref{GK-tensor})}{=}\text{GK}\dim L(\lambda)+\text{GK}\dim L(\mu)
\end{align*}
since $\wedge(\frg(-1))$ is finite-dimensional and $L_0(\lambda|\mu)\cong L(\lambda)\otimes L(\mu)\otimes\mathbb C_\chi$, with the one-dimensional central factor contributing zero to GK dimension. In this case the equation (\ref{prop-GK-tensor}) holds by Lemma \ref{GK-dim-submodule}.

\end{proof}

\begin{Rem}
    Proposition \ref{semisimple-GK-dim} will hold for all type I Lie superalgebras, not only for basic classical ones. 
\end{Rem}

\begin{Cor}
   $\text{GK}\dim \widetilde{M}(\Lambda)= \text{GK}\dim M(\Lambda)=|\Delta_0^+|$.
\end{Cor}
\begin{proof}
  Let $\frg$ be of type I, by the process of Lemma \ref{GK-dim-submodule} we have 
    $$U(\frg)\otimes_{U(\frp)} M(\Lambda)\cong\widetilde{M}(\Lambda).$$
    Then 
    $$\text{GK}\dim \widetilde{M}(\Lambda)\overset{\text{Prop. \ref{GK-dim-induced-mod}}}{=}\text{GK}\dim M(\Lambda)+\dim \frg_0-\dim\frp_0=\text{GK}\dim M(\Lambda)=|\Delta_0^+|.$$
The final equality follows from \cite{EH04}.

Let $\mathfrak{g}$ be of type II, we have
\begin{equation*}
\widetilde{M}(\Lambda) = U(\mathfrak{g}) \otimes_{U(\mathfrak{b})} V_{\Lambda} \cong U(\mathfrak{n}^{-}) \otimes V_{\Lambda} = S(\mathfrak{n}_{0}^{-}) \otimes \wedge(\mathfrak{n}_{1}^{-}) \otimes V_{\Lambda}.
\end{equation*}
Then
\begin{align*}
\mathrm{GK\,dim}\, \widetilde{M}(\Lambda) &= \mathrm{GK\,dim}(S(\mathfrak{n}_{0}^{-}) \otimes \wedge(\mathfrak{n}_{1}^{-}) \otimes V_{\Lambda}) \\
&= \mathrm{GK\,dim}\, S(\mathfrak{n}_{0}^{-}) = \dim \mathfrak{n}_{0}^{-} = |\Delta_{0}^{+}|.
\end{align*}
\end{proof}

\begin{Prop}\label{GKdim_uniform}
     Let $\frg$ be of type I, with the distinguished Borel subalgebra, and $\widetilde{L}(\lambda|\mu)$ be a simple highest weight $\frg$-module, then GK dimension of $\widetilde{L}(\lambda|\mu)$ is given by
  $$\text{GK}\dim \widetilde{L}(\lambda|\mu)=|\Delta_0^+|-{\bf a}_{[\Lambda]}(w)\text{ for }w=(u,v)\in W,$$ 
  where $w = w_{\Lambda}$ denotes the minimal length element specified in Lemma \ref{GKdim-uniform1}, applied to the shifted weight $\Lambda^{\rho_0}=\Lambda+\rho_0$. Here $w_{\Lambda}$ abbreviates $w_{\Lambda^{\rho_0}}$, and $W_{[\Lambda]}=W_{[\Lambda^{\rho_0}]}$ since $\rho_0$ is integral for the even root system.
\end{Prop}
\begin{proof}
    This is a consequence of  Proposition \ref{a_function_property},   Lemma \ref{GKdim-uniform1} and Proposition \ref{semisimple-GK-dim}.
\end{proof}

Section \ref{section-a-fun} gives the insertion formulas used to evaluate the $\mathbf a$-function in Proposition \ref{GKdim_uniform}.

\section{Formulas for Lusztig's {\bf a}-functions}\label{section-a-fun}
In this section we provide explicit combinatorial expressions for Lusztig’s {\bf a}-functions, which are the key to the GK dimension formula. 
\subsection{Computation of ${\bf a}_{[\Lambda]}(w_{\Lambda})$ for $\mathfrak{sl}(m|n)$ and $\mathfrak{osp}(k|2n)$}


\begin{Def}\label{F_X}
    For $ x \in \mathrm{Seq}_n (\Gamma)$, define
    \begin{align*}
&F_A(x):=\sum_{i\geq 1}(i-1){\bf p}_i,\\
        &F_B(x):=\sum_{i\geq 1}(i-1){\bf p}_i^{\odd},\\
        &F_D(x):=\sum_{i\geq 1}(i-1){\bf p}_i^{\ev}.
    \end{align*}
\end{Def}

\begin{Lem}\label{F-A}
    For $ x \in \mathrm{Seq}_n (\Gamma)$, put $Y=Y(x)$. Then $F_A(x)=A(Y)$, where $A(Y)$ is defined in Section \ref{A(Y)}. In particular, for $w\in \mathfrak S_n$, ${\bf a}(w)=A(Y(w))$.
\end{Lem}
\begin{proof}
    $\sum\limits_{i\geq 1}(i-1){\bf p}_i=\sum\limits_{i\geq 1}\sum\limits_{j= 1}^{{\bf p}_i}(i-1)=\sum\limits_{j\geq 1}\sum\limits_{i= 1}^{{\bf q}_j}(i-1)=\sum\limits_{j\geq 1}\dfrac{{\bf q}_j({\bf q}_j-1)}{2}=A(Y)$.  The last assertion follows from \cite[Proposition 4.4]{BX1}.
\end{proof}

\begin{Lem}{\cite[Propositions 3.3 and 3.6]{BXX}}\label{F-B-D}
  For $w\in W$, we have
    \[
{\bf a}(w) = 
\begin{cases} 
F_B({}^-w) & \text{ if } w \in W_B, \\
F_D({}^-w) & \text{ if }  w \in W_D.
\end{cases}
\]
\end{Lem}

\begin{Thm}\label{a-function-gl(m|n)-osp}
For $w=(u,v)\in W_0$, we have
\[
{\bf a}(w) = 
\begin{cases} 
\sum_{i\geq 1}(i-1){\bf p}(u)_i+\sum_{j\geq 1}(j-1){\bf p}(v)_j & \text{ if } \frg=\frsl(m|n), \\
\sum_{i\geq 1}(i-1){\bf p}({}^-u)_i^{\odd}+\sum_{j\geq 1}(j-1){\bf p}({}^-v)_j^{\odd} & \text{ if } \frg=\mathfrak{osp}(2m+1|2n), \\
\sum_{i\geq 1}(i-1){\bf p}({}^-u)_i^{\ev}+\sum_{j\geq 1}(j-1){\bf p}({}^-v)_j^{\odd} & \text{ if } \frg=\mathfrak{osp}(2m|2n)\text{ for }m\geq2, \\
\sum_{j\geq 1}(j-1){\bf p}({}^-v)_j^{\odd} & \text{ if } \frg=\mathfrak{osp}(2|2n).
\end{cases}
\]
\end{Thm}
\begin{proof}

For $\frg=\frsl(m|n)$, Proposition \ref{a_function_property} (6) gives $\mathbf a(w)=\mathbf a(u)+\mathbf a(v)$.  By Lemma \ref{F-A} we have $\mathbf a(u)=F_A(u)=\sum_{i\geq1}(i-1)\mathbf p(u)_i$, and similarly for $\mathbf a(v)$.

  For $\frg=\mathfrak{osp}(2m+1|2n)$ or $\mathfrak{osp}(2m|2n)$, it follows from Proposition \ref{a_function_property} and Lemma \ref{F-B-D}.  

  For $\frg=\mathfrak{osp}(2|2n)$, it follows from ${\bf a}(e)=0$ and Proposition \ref{a_function_property} and Lemma \ref{F-B-D}. 
\end{proof}

\begin{Rem}
    Also we can compute the value of {\bf a}$(w)$ by PyCox if $w$ is given in one-line notation $(w(1),w(2),\cdots)$. First we need to convert one-line notation to the product of simple reflections. We achieve this conversion through a simple program. The program takes Coxeter type and $w$ as  the input and returns an expression as a product of simple reflections. The program is implemented in Python and available at
    \begin{center}
\textcolor{magenta}{https://github.com/JingJiang-web/convert-oneline-notation-to-simple-reflection}.
    \end{center}
\end{Rem}

\begin{Exam}
     Let $\frg=\mathfrak{osp}(11|8)$ and consider $w=(\underbrace{3,4,-1,5,2}_u  | \underbrace{3,-1, -4, 2}_v)$ from Example \ref{a-w}, then by the program we have $w=[\underbrace{[2,3,4,0,1,2]}_u|\underbrace{[2,3,2,1,0,1,2,0,1]}_v]$ (Geck's notation, see Section \ref{PyCox} for more details), therefore by PyCox we have {\bf a}$(u)=2$, similarly {\bf a}$(v)=5$, hence {\bf a}$(w)={\bf a}(u)+{\bf a}(v) =7$.
\end{Exam}

\subsection{Computation of $\mathbf a_{[\Lambda]}(w_{\Lambda})$ for $G(3)$}

For the $G_2$ factor, we use the positive index reduction algorithm of \cite{BGXW} and evaluate the resulting $\mathbf a$-function with PyCox. The input is the ordinary shifted weight, as in Lemma \ref{GKdim-uniform1}; throughout this subsection, $w_\Lambda$ abbreviates $w_{\Lambda^{\rho_0}}$.

\subsubsection{The integral case}
For a finite-dimensional complex simple Lie algebra $\dot{\mathfrak{g}}$ with a fixed Cartan subalgebra $\mathfrak{h}$, use the invariant bilinear form to identify $\mathfrak h$ with $\mathfrak h^*$ and write $(\lambda,\alpha^\vee)$ for the coroot pairing. Let $\Pi=\{\alpha_i \mid  1\leq i\leq n\}$ be the set of simple roots of $\dot{\mathfrak{g}}$ with corresponding fundamental weights $\{\omega_i\mid 1\leq i\leq n\}$. 
Recall that the Cartan matrix is defined by
$A=(A_{ij})_{n\times n}$ with $A_{ij}=(\alpha_i, \alpha_j^\vee )$.

\begin{Lem}{\cite[Lemma 3.1]{BGXW}}\label{rho}
   The ordinary Weyl vector is $\rho'=\sum\limits_{k=1}^n\omega_k$, and 
   $\alpha_j=\sum\limits_{k=1}^n A_{jk}\omega_k$. For any $\lambda\in \mathfrak{h}^*$, we write $\lambda=\sum\limits_{1\leq i\leq n}k_i\omega_i$, where $k_i=( \lambda, \alpha_i^{\vee})$. Then $\lambda$ is an integral weight if and only if $k_i\in \mathbb{Z}$ for all $1\leq i\leq n$. Also, $\lambda$ is an integral and anti-dominant weight if and only if   $k_i\in \mathbb{Z}_{\leq 0}$ for  all $1\leq i\leq n$.
\end{Lem}

To find $w_\lambda$ for an integral parameter $\lambda$, proceed as follows. When $\lambda$ is anti-dominant, $w_{\lambda}=\mathrm{id}$. Suppose $\lambda$ is not anti-dominant, then there are some $k_i\in \mathbb{Z}_{>0}$. Suppose the largest index is $i_{\lambda}={i_1}$ such that $k_{i_1}\in \mathbb{Z}_{>0}$. Then we have $$s_{i_1}\lambda=\lambda-k_{i_1}\alpha_{i_1}=\lambda-k_{i_1}\sum\limits_{j=1}^n A_{i_1j}\omega_j=\lambda-2k_{i_1}\omega_{i_1}-k_{i_1}\sum\limits_{j=1,j\neq i_1}^n A_{i_1j}\omega_j.$$
 Then the coefficient of $\omega_{i_1}$ in $s_{i_1}\lambda$ becomes $-k_{i_1}$. We continue this process until all the coefficients of $\omega_{i}$ are in $\mathbb{Z}_{\leq 0}$. This process will stop after finite steps. If the successive reflections are $s_{i_1},\ldots,s_{i_k}$, then $s_{i_k}\cdots s_{i_1}\lambda$ is anti-dominant and $w_\lambda=s_{i_1}\cdots s_{i_k}$.

This is the \emph{positive index reduction} algorithm.

\begin{Lem}{\cite[Lemma 3.2]{BGXW}}\label{find-w-lambda}
    For any integral weight $\lambda$ (regular or singular),  we can get an anti-dominant weight $\mu$ by applying the positive index reduction algorithm. We multiply all the $s_{i_j}, 1\leq j \leq k$ that appeared in this process and get a $w_{\lambda}:=s_{i_1} s_{i_2} \cdots s_{i_k}$. Then this $w_{\lambda}$ has the minimal length in $W$ such that $w_{\lambda}^{-1}\lambda$ is anti-dominant.
\end{Lem}

    From \cite[Lemma 3.1]{Dong23},  we can choose any positive index (whose corresponding coefficient is positive) during the positive index reduction algorithm. 
The following web page in \cite{BGXW} is convenient to apply the algorithm:
    \begin{center}
\textcolor{magenta}{http://test.slashblade.top:5000/lie/antidominant}.
    \end{center}
We write $\lambda=\sum\limits_{1\leq i\leq n}k_i\omega_i:=[k_1,\cdots,k_{n}]$, where $k_i=( \lambda, \alpha_i^{\vee})$. Here  the order of simple roots is the same as that in \cite{Geck}. In the above web page, we only need to input $``k_1,\cdots,k_{n}"$, and the corresponding $w_{\lambda}$ will appear as the output.

\begin{Prop}
  Use the notation of Subsection \ref{notation-G3} and write the ordinary shifted integral weight as $\Lambda^{\rho_0}=b'_1\omega'_1+b_1\omega_1+b_2\omega_2$. Then $\mathbf a(w_\Lambda)=\mathbf a(w_{b'_1})+\mathbf a(w_{(b_1,b_2)})$.
\end{Prop}
\begin{proof}
    This is a direct corollary of Proposition \ref{a_function_property} \eqref{a-7}.
\end{proof}

 \begin{Exam}\label{exam-G(3)}
     Let $\Lambda^{\rho_0}:=\Lambda+\rho_0=(1|\frac{1}{2},\frac{3}{2},-\frac{3}{2})$ be an integral weight of type $G(3)$, then we can write $\Lambda^{\rho_0}=\omega'_1+\omega_1+\omega_2:=[1|1,1]$. By the above web page we have $w_{\Lambda}=[[0]|\underbrace{[1,0,1,0,1,0]}_{w_{\mu^{\rho_0}}}]$. By using PyCox, we have {\bf a}$(w_{\Lambda})=1+6=7$.
 \end{Exam}

\subsubsection{The non-integral case}
Write $\Lambda=(\lambda\mid\mu)$ and decompose $\Phi_{[\mu^{\rho_0}]}$ into its orthogonal irreducible components.
Choose an isomorphism $\phi$ to standard root systems that sends their positive simple roots to positive simple roots. The transported weight $\phi(\mu^{\rho_0})$ is defined by retaining the simple coroot pairings of $\mu^{\rho_0}$ on this subsystem; the complementary Cartan directions are discarded. It is integral for the target root system. By \cite[Theorem 3.6 and Corollary 3.7]{BGXW},
\[
\mathbf a(w_{\mu^{\rho_0}})
=\mathbf a(w_{\phi(\mu^{\rho_0})}).
\]
For a classical component use RS insertion; for $G_2$ use positive index reduction and PyCox.

\begin{Prop}
 Keep the notations as in section \ref{notation-G3}. Let $\Lambda = (\Lambda_1 | \underbrace{\Lambda_2, \Lambda_3, \Lambda_4}_{\mu})$ be a non-integral weight of $G(3)$, then $\mathbf{a}(w_{\Lambda}) = \mathbf{a}(w_{\Lambda_1^{\rho_0}}) + \mathbf{a}(w_{\phi(\mu^{\rho_0})})$.
\end{Prop}
\begin{proof}
This is a direct corollary of Proposition \ref{a_function_property} \eqref{a-7} and \cite[Theorem 3.6]{BGXW}.
\end{proof}

\begin{Exam}
Let $\mathfrak g=G(3)$ and let $\Lambda^{\rho_0}=\left(1\mid\frac32,2,0\right)$. Then ${\bf a}(w_{\lambda^{\rho_0}})=1$.  One can easily check that 
\[
\Phi_{[\mu^{\rho_0}]}
=\{\pm(\epsilon_2-\epsilon_3)\}\sqcup\{\pm\epsilon_1\}
\cong A_1\times\widetilde A_1.
\]
Let $\alpha_1=(1,-1)$ and $\alpha'_1=(1,-1)$ be the standard simple roots of the two $A_1$ factors, and choose
$\phi(\epsilon_2-\epsilon_3)=\alpha_1$ and $\phi(\epsilon_1)=\alpha'_1$. Then
\[
\phi(\mu^{\rho_0})\big|_{A_1}=(1,-1),
\qquad
\phi(\mu^{\rho_0})\big|_{\widetilde A_1}
=\left(\frac12,-\frac12\right).
\]
By using RS algorithm, we have ${\bf a}(w_{\mu^{\rho_0}})=2$. Hence
\[
{\bf a}(w_{\Lambda^{\rho_0}})
={\bf a}(w_{\lambda^{\rho_0}})
+{\bf a}(w_{\mu^{\rho_0}})
=1+2=3.
\]
\end{Exam}

\section{GK dimensions of highest weight modules $L_{\frg_0}(\Lambda)$}\label{section-GK-formula}
\subsection{GK dimensions of $L_{\frg_0}(\Lambda)$ for basic classical Lie superalgebras excluding exceptional types}

\begin{Thm}\label{GK-dim-gl(m|n)}
Let $\frg=\frsl(m|n)$ with $m\neq n$. For  $(\lambda\mid\mu)\in\frh^*$, we have
 \begin{align*}
   \text{GK}\dim L_{\frg_0}(\Lambda)&=|\Delta_0^+|-A(P(\lambda|\mu)) \\&=\frac{m(m-1)}{2}-A(P(\lambda))+\frac{n(n-1)}{2}-A(P(\mu)).
 \end{align*}  
\end{Thm}
\begin{proof}
By Proposition \ref{semisimple-GK-dim} we have $\mathrm{GK\,dim}\, L_{\mathfrak{g}_0}(\Lambda) = \mathrm{GK\,dim}\, L(\lambda) + \mathrm{GK\,dim}\, L(\mu)$, where $L(\lambda)$ (resp. $L(\mu)$) is an irreducible highest weight $\mathfrak{sl}(m)$ (resp. $\mathfrak{sl}(n)$)-module. Then the theorem follows from Lemma \ref{F-A}, Theorem \ref{a-function-gl(m|n)-osp} and \cite[Theorem 4.6]{BX1}.

\end{proof}

Next we consider the orthosymplectic Lie superalgebras.
\subsection{The integral case}
In this subsection, the anti-dominant weight $\delta$ is always integral.

\begin{Lem}{\cite[Lemmas 5.2 and 5.3]{BXX}}\label{lambda-w}
\begin{enumerate}
    \item  Suppose $\Phi=B_n$ or $C_n$. Let $\lambda=w\delta$ for $w\in W^I$, then ${\bf p}(\lambda^-)={\bf p}({}^-w)$.

    \item Suppose $\Phi=D_n$. Let $\lambda=w\delta$ for $w\in W^I$, then ${\bf p}(\lambda^-)={\bf p}({}^-w)$ or ${\bf p}({}^-wt)$, where $t=(-1,2,\cdots,n)=(-1,1)\in W_B$.
\end{enumerate}
\end{Lem}

\begin{Prop}\label{GK-dim-L_g0}
    Let $\Lambda=(\lambda_1,\lambda_2,\cdots,\lambda_m|\mu_1,\mu_2,\cdots,\mu_n)\in\frh^*$ be integral, then
$$
\resizebox{\linewidth}{!}{$
\text{GK}\dim L_{\frg_0}(\Lambda) = 
\begin{cases} 
\frac{m(m-1)}{2}+\frac{n(n-1)}{2}-( \sum_{i\geq 1}(i-1){\bf p}(\lambda^{\rho_0})_i+\sum_{j\geq 1}(j-1){\bf p}(\mu^{\rho_0})_j) & \text{ if } \frg=\frsl(m|n), \\
m^2+n^2-(\sum_{i\geq 1}(i-1){\bf p}((\lambda^{\rho_0})^-)_i^{\odd}+\sum_{j\geq 1}(j-1){\bf p}((\mu^{\rho_0})^-)_j^{\odd}) & \text{ if } \frg=\mathfrak{osp}(2m+1|2n), \\
m^2-m+n^2-(\sum_{i\geq 1}(i-1){\bf p}((\lambda^{\rho_0})^-)_i^{\ev}+\sum_{j\geq 1}(j-1){\bf p}((\mu^{\rho_0})^-)_j^{\odd}) & \text{ if } \frg=\mathfrak{osp}(2m|2n)\text{ for }m\geq2, \\
n^2-\sum_{i\geq 1}(i-1){\bf p}((\mu^{\rho_0})^-)_i^{\odd} & \text{ if } \frg=\mathfrak{osp}(2|2n).
\end{cases}
$}
$$
\end{Prop}
\begin{proof}
    
  For $\frg=\mathfrak{osp}(2m+1|2n)$,  this is a consequence of Theorem \ref{a-function-gl(m|n)-osp}, Lemma \ref{lambda-w} and the fact that $|\Delta_0^+|=m^2+n^2$.

  For $\frg=\mathfrak{osp}(2m|2n)$, let $u\in W_D$ and $v\in W_C$ be the minimal length elements in Lemma \ref{GKdim-uniform1}, corresponding to the shifted parameters $\lambda^{\rho_0}$ and $\mu^{\rho_0}$, respectively. Applying Lemma \ref{lambda-w} separately to these two factors gives ${\bf p}((\lambda^{\rho_0})^-)={\bf p}({}^-u)$ or ${\bf p}({}^-ut)$, and ${\bf p}((\mu^{\rho_0})^-)={\bf p}({}^-v)$. By \cite[Thm. 4.10]{BXX}, ${\bf p}({}^-u)^{\ev}={\bf p}({}^-ut)^{\ev}$. Hence Theorem \ref{a-function-gl(m|n)-osp} gives ${\bf a}(u)+{\bf a}(v)=F_D((\lambda^{\rho_0})^-)+F_B((\mu^{\rho_0})^-)$. Lemma \ref{GKdim-uniform1} and $|\Delta_0^+|=m^2-m+n^2$ now give the displayed formula.

  For $\frg=\mathfrak{osp}(2|2n)$, it follows from Theorem \ref{a-function-gl(m|n)-osp} and the fact that $|\Delta_0^+|=n^2$.
\end{proof}


    

\subsection{The general case}\label{The-general-case}\label{sec:BD-coordinate-blocks}
In this subsection we consider the GK dimension of $L_{\frg_0}(\Lambda)$ with $\Lambda$ not necessarily integral. 

For orthosymplectic Lie superalgebras, let $\Delta_{[\Lambda]}$ be an integral root subsystem of $\Delta_0$, i.e. $\Delta_{[\Lambda]}=\{\alpha\in\Delta_0 | (\Lambda,\alpha^{\vee})\in\mathbb{Z}\}$. 

Let $T^1_z:=\{i\leq m | \Lambda_i\in z+\mathbb{Z}\}$ for $z\in\mathbb{C}$ and $T^2_z:=\{j\leq n | \Lambda_{m+j}\in z+\mathbb{Z}\}$ for $z\in\mathbb{C}$. Thus the indices in $T^2_z$ label the $\delta$-coordinates. Define $\Delta_z$ as follows.

If $\frg=\mathfrak{osp}(2m+1|2n)$ and $z\in\mathbb{Z}$, set
\begin{align*}
\Delta_z:={}&\{\pm(\epsilon_i\pm\epsilon_j):i,j\in T^1_z,\ i<j\}\cup\{\pm\epsilon_i:i\in T^1_z\}\\
&\cup\{\pm(\delta_i\pm\delta_j):i,j\in T^2_z,\ i<j\}\cup\{\pm2\delta_i:i\in T^2_z\}.
\end{align*}

If $\frg=\mathfrak{osp}(2m+1|2n)$ and $z\in\frac12+\mathbb{Z}$, set
\begin{align*}
\Delta_z:={}&\{\pm(\epsilon_i\pm\epsilon_j):i,j\in T^1_z,\ i<j\}\cup\{\pm\epsilon_i:i\in T^1_z\}\\
&\cup\{\pm(\delta_i\pm\delta_j):i,j\in T^2_z,\ i<j\}.
\end{align*}

If $\frg=\mathfrak{osp}(2m|2n)$ and $z\in\mathbb{Z}$, set
\begin{align*}
\Delta_z:={}&\{\pm(\epsilon_i\pm\epsilon_j):i,j\in T^1_z,\ i<j\}\\
&\cup\{\pm(\delta_i\pm\delta_j):i,j\in T^2_z,\ i<j\}\cup\{\pm2\delta_i:i\in T^2_z\}.
\end{align*}

If $\frg=\mathfrak{osp}(2m|2n)$ and $z\in\frac12+\mathbb{Z}$, set
\begin{align*}
\Delta_z:={}&\{\pm(\epsilon_i\pm\epsilon_j):i,j\in T^1_z,\ i<j\}\\
&\cup\{\pm(\delta_i\pm\delta_j):i,j\in T^2_z,\ i<j\}.
\end{align*}

In general, if $z\notin\frac{1}{2}\mathbb{Z}$, set
$\Delta_z=\Delta_z^{1}\cup\Delta_z^{2}$, where
\begin{align*}
   \Delta_z^{1}:=\{\epsilon_i-\epsilon_j\mid i,j\in T_z^1,\ i\neq j\}&\cup\{\pm(\epsilon_i+\epsilon_k)
\mid i\in T_z^1,\ k\in T_{-z}^1\}\\&\cup\{\epsilon_k-\epsilon_l\mid k,l\in T_{-z}^1,\ k\neq l\}
\end{align*}
and
\begin{align*}
\Delta_z^{2}:=\{\delta_p-\delta_q\mid p,q\in T_z^2,\ p\neq q\}&\cup\{\pm(\delta_p+\delta_r)\mid p\in T_z^2,\ r\in T_{-z}^2\}\\&\cup\{\delta_r-\delta_s\mid r,s\in T_{-z}^2,\ r\neq s\}.
\end{align*}

For $\mathfrak{g} = \mathfrak{osp}(2|2n)$, we only need to remove $\pm(\epsilon_{i_1} \pm \epsilon_{j_1})$ from $\Delta_z$ in the case of $\mathfrak{g} = \mathfrak{osp}(2m|2n)$. Then $\Delta_{[\Lambda]}$ can be decomposed into some orthogonal subsystems by the following proposition.

For $z,z'\in\mathbb C$, define
\[
z\sim z'
\quad\Longleftrightarrow\quad
z-z'\in\mathbb Z \text{ or } z+z'\in\mathbb Z.
\]
For $\Lambda=(\lambda\mid\mu)\in\frh^*$, let $S_\Lambda$ consist of $0$, $\frac12$, and one representative of each remaining equivalence class containing an entry of $\lambda$ or $\mu$. Empty root subsystems are allowed, with $W(\varnothing)=\{1\}$.

\begin{Prop}\label{phi-lambda-phi-z}
With the notation above, we have
\[
\Delta_{[\Lambda]}=\bigsqcup_{z\in S_\Lambda}\Delta_z,\qquad W(\Delta_{[\Lambda]})
\cong\prod_{z\in S_\Lambda}W(\Delta_z).
\]
Moreover, the root subsystems $\Delta_z$ with $z\in S_\Lambda$ are pairwise orthogonal.
\end{Prop}

\begin{proof}
The definitions give $\Delta_z=\Delta_{z'}$ whenever $z\sim z'$. The definition therefore implies
\[
\Delta_{[\Lambda]}=\bigcup_{z\in S_\Lambda}\Delta_z.
\]

For distinct $z,z'\in S_\Lambda$, the corresponding equivalence classes are different. Hence
\[
(T_z^r\cup T_{-z}^r)\cap(T_{z'}^r\cup T_{-z'}^r)=\varnothing,\qquad r=1,2.
\]
Thus roots belonging to different subsystems involve disjoint sets of orthogonal coordinates. Consequently, the union is disjoint and the subsystems are pairwise orthogonal.

The corresponding Weyl subgroups act on disjoint coordinate subspaces and fix the other coordinate subspaces pointwise. Their product is therefore an internal direct product. Since their generating reflections together generate $W(\Delta_{[\Lambda]})$, the asserted Weyl group decomposition follows.
\end{proof}

For the special linear case, define $[\Lambda]$ to be the set of maximal subsequences whose entries have pairwise integral differences. For the orthosymplectic cases, define $[\Lambda]$ using the equivalence relation $a\sim b$; thus each subsequence consists of all entries in one pair $\{z+\mathbb Z,-z+\mathbb Z\}$, in their original order. In this case, we set $ [\Lambda]_1 $ (resp. $ [\Lambda]_2 $) to be the subset of $ [\Lambda] $ consisting of sequences with  all entries belonging to $ \mathbb{Z} $ (resp. $ \frac12+\mathbb{Z} $) and denote them by $\Lambda_0$ (resp. $\Lambda_{\frac12}$). We set $[\Lambda]_{1,2}=[\Lambda]_1\cup [\Lambda]_2, \quad [\Lambda]_3=[\Lambda]\setminus[\Lambda]_{1,2}$. More specifically, for $\Lambda=(\lambda|\mu)$, we will use notations $\lambda_0$, $\lambda_{\frac12}$, $\lambda_{z}$, $\mu_0$, $\mu_{\frac12}$ and $\mu_{z}$. 

\begin{Exam}
		Let $\Lambda=(7,3.5,4,1|4,2,1.2,3.5,1.5,3,4.2)\in\frh^*$, then
		$$\lambda_0=(7,4,1),\:\lambda_{\frac12}=(3.5),\:\mu_0=(4,2,3),\:\mu_{\frac12}=(3.5,1.5)\text{ and }\mu_z=(1.2,4.2).$$
	\end{Exam}

\begin{Def}
In the orthosymplectic cases, let $x=(\lambda_{i_1},\lambda_{i_2},\ldots,\lambda_{i_r})\in[\lambda]_3$. Choose $z\in\mathbb{C}\setminus\frac12\mathbb{Z}$ such that every entry of $x$ lies in $(z+\mathbb{Z})\cup(-z+\mathbb{Z})$. Let $x'=(\lambda_{j_1},\ldots,\lambda_{j_p})$ consist of the entries in $z+\mathbb Z$, and let $(\lambda_{k_1},\ldots,\lambda_{k_q})$ consist of the remaining entries, which lie in $-z+\mathbb Z$. Both subsequences retain their original order and may be empty. Define
\begin{equation}\label{eq:BD-folded-sequence}
\tilde{x}_z=(\lambda_{j_1},\ldots,\lambda_{j_p},-\lambda_{k_q},\ldots,-\lambda_{k_1}).
\end{equation}
    Similarly define $\tilde y_z$ on $\mu_z$. For any sequence $\xi$, the notation $\widetilde\xi_z$ denotes the same construction, while $\xi_0$ and $\xi_{\frac12}$ denote its integer and half-integer subsequences.
\end{Def}

In the special linear formulas, $x_z$ denotes the maximal subsequence of $x$ in $z+\mathbb Z$, retaining its original order. The notation $\sum_{\lambda,\mu}$ means separate sums over the subsequences of the two arguments: integral-difference subsequences in type $A$, and the coset pairs in $\lambda_z$ and $\mu_z$ in the orthosymplectic cases. This convention also applies to $\lambda_{\Ydown},\mu_{\mid}$ in Section \ref{section-GK-basic}. Since $\|$ adds statistics, entries from different even factors are not inserted together.

\begin{Thm}\label{GKdim-thm}
    Let $\Lambda=(\lambda|\mu)=(\lambda_1,\lambda_2,\cdots,\lambda_{m}|\mu_1,\mu_2,\cdots,\mu_n)\in\frh^*$.
\begin{enumerate}
\item  If $\frg=\mathfrak{sl}(m|n)$, then
$$\text{GK}\dim L_{\frg_0}(\Lambda)=\frac{m(m-1)}{2}+\frac{n(n-1)}{2}-\sum_{\lambda,\mu}F_A((\lambda^{\rho_0})_z\|(\mu^{\rho_0})_z).$$
    \item If $\frg=\mathfrak{osp}(2m+1|2n)$, then
    $$\text{GK}\dim L_{\frg_0}(\Lambda)=m^2+n^2-F_B((\lambda^{\rho_0})_0^-\|(\mu^{\rho_0})_0^-\|(\lambda^{\rho_0})_{\frac12}^-)-  F_D((\mu^{\rho_0})_{\frac12}^-)-\sum_{\lambda,\mu}F_A(\widetilde{(\lambda^{\rho_0})}_z\|\widetilde{(\mu^{\rho_0})}_z).$$
\item If $\frg=\mathfrak{osp}(2m|2n)$, then
\[
\begin{aligned}
\GKD L_{\frg_0}(\Lambda)=m^2-m+n^2-F_B((\mu^{\rho_0})_0^-)-F_D((\lambda^{\rho_0})_{\frac12}^-\|(\mu^{\rho_0})_{\frac12}^-\|(\lambda^{\rho_0})_0^-)-\sum_{\lambda,\mu}F_A(\widetilde{(\lambda^{\rho_0})}_z\|\widetilde{(\mu^{\rho_0})}_z).
\end{aligned}
\]
\item  If $\frg=\mathfrak{osp}(2|2n)$,   then
$$\text{GK}\dim L_{\frg_0}(\Lambda)=n^2-F_B((\mu^{\rho_0})_0^-)-F_D((\mu^{\rho_0})_{\frac12}^-)-\sum_{\mu}F_A(\widetilde{(\mu^{\rho_0})}_z).$$

\end{enumerate}
Here $F_Y(p\|q\|s)=F_Y(p)+F_Y(q)+F_Y(s)$ for function $F_Y$ defined in Definition \ref{F_X} with $Y\in\{A,B,D\}$.  
\end{Thm}
\begin{proof}
Put $\xi=\Lambda+\rho_0$. Since $\rho_0$ is integral for the even root system, $\Delta_{[\xi]}=\Delta_{[\Lambda]}$ and $W_{[\xi]}=W_{[\Lambda]}$. For the element $w$ specified in Lemma \ref{GKdim-uniform1}, we have
\[
\GKD L_{\frg_0}(\Lambda)
=|\Delta_0^+|-\mathbf a_{[\xi]}(w).
\]
In the special linear case, applying \cite[Theorem 4.6]{BX1} separately to $\lambda^{\rho_0}$ and $\mu^{\rho_0}$ gives (1).

In the orthosymplectic cases, Proposition \ref{phi-lambda-phi-z} gives
\[
W_{[\xi]}\cong\prod_{z\in S_\xi}W_z,
\qquad
\mathbf a_{[\xi]}(w)
=\sum_{z\in S_\xi}\mathbf a_{W_z}(w_z),
\]
where $W_z=W(\Delta_z)$ and $w=(w_z)_{z\in S_\xi}$. 

In each ordinary factor, let $x$ denote its coordinates. For $z\in S_\xi\setminus\{0,\frac12\}$, the signed coordinate order in \eqref{eq:BD-folded-sequence} identifies the block with a standard positive system of type $A$, giving the contribution $F_A(\widetilde{x}_z)$. By Lemmas \ref{F-B-D} and \ref{lambda-w}, integer blocks contribute $F_B(x_0^-)$ in types $B,C$ and $F_D(x_0^-)$ in type $D$; half-integer blocks contribute $F_B(x_{\frac12}^-)$ in type $B$ and $F_D(x_{\frac12}^-)$ in types $C,D$. The type $D$ ambiguity does not affect the even-box counts by \cite[Theorem 4.10]{BXX}. 

Adding these contributions separately over the ordinary factors, and using $|\Delta_0^+|=m^2+n^2$, $m(m-1)+n^2$, and $n^2$ in the three orthosymplectic cases, respectively, gives (2)--(4).
\end{proof}

\begin{Exam}
    Let $\frg=\mathfrak{osp}(15|8)$. Suppose that $$\Lambda^{\rho_0}:=\Lambda+\rho_0=(2.2,2,0.2,-4,0,1.2,-1|-1.9,2,2.1,0)\in\frh^*.$$
The subsequences are
\[
\begin{aligned}
(\lambda^{\rho_0})_0&=(2,-4,0,-1),&
\lambda^{\rho_0}_z&=(2.2,0.2,1.2),\\
(\mu^{\rho_0})_0&=(2,0),&
\mu^{\rho_0}_z&=(-1.9,2.1).
\end{aligned}
\]
One can verify that ${\bf p}((\lambda^{\rho_0})_0^-)=[4,2,1^2]$, ${\bf p}(\lambda^{\rho_0}_z)=[2,1]$, ${\bf p}((\mu^{\rho_0})_0^-)=[2,1^2]$ and ${\bf p}(\mu^{\rho_0}_z)=[2]$. Hence, ${\bf p}((\lambda^{\rho_0})_0^-)^{\odd}=[2,1,0,1]$ and ${\bf p}((\mu^{\rho_0})_0^-)^{\odd}=[1,1]$. By Theorem \ref{GKdim-thm} we have
$$\text{GK}\dim L_{\frg_0}(\Lambda)=7^2+4^2-4-1-1=59.$$


\end{Exam}


\subsection{GK dimensions of $L_{\frg_0}(\Lambda)$ for the Lie superalgebra $G(3)$}\label{section-GK-formula-G3}
In this section, we will describe the GK dimensions of highest weight modules $L_{\mathfrak{g}_0}(\Lambda)$ for Lie superalgebra $G(3)$. First recall some definitions and propositions from \textcolor{magenta}{[BGWX25]} used in this section.

Let $\Phi$ be an irreducible root system with simple system $\Pi=\{\alpha_1,\ldots,\alpha_n\}$, and let
\[
\theta=\sum_{i=1}^n h_i\alpha_i
\]
be its highest root. For a subset $S\subseteq\Phi$, denote by $\langle S\rangle_{\mathrm{rt}}$ the root subsystem of $\Phi$ generated by $S$. For $1\leq i,j\leq n$, set
\[
\widetilde{\Phi}(i)=\big\langle(\Pi\cup\{-\theta\})\setminus\{\alpha_i\}\big\rangle_{\mathrm{rt}},
\qquad\Phi(j)=\big\langle\Pi\setminus\{\alpha_j\}\big\rangle_{\mathrm{rt}}.
\]
Let
\[
\mathcal P(\Phi)=\left(\{\widetilde{\Phi}(i):1\leq i\leq n\}\cup\{\Phi(j):1\leq j\leq n\}\right)\setminus\{\Phi\}.
\]
A proper root subsystem $\Psi\subsetneq\Phi$ is called \textit{pseudo-maximal} if it is $W(\Phi)$-conjugate to a
maximal element of $\mathcal P(\Phi)$ with respect to inclusion.

\begin{Lem}[{\cite[Propositions 4.1 and 4.2]{BGXW}}]\label{GK-DIM-G2}
Let $L(\mu)$ be a simple $G_2$-module with  highest weight $\mu$, and put $\mu^{\rho_0}=\mu+\rho'$, where
$\rho'=\omega_1+\omega_2$ is $G_2$ Weyl vector.
\begin{enumerate}
\item If $\mu$ is integral, then
\begin{enumerate}
\item $\GKD L(\mu)=0$ if and only if $\mu^{\rho_0}$ is strictly dominant.
\item $\GKD L(\mu)=5$ if and only if $\mu^{\rho_0}$ is neither strictly
dominant nor anti-dominant.
\item $\GKD L(\mu)=6$ if and only if $\mu^{\rho_0}$ is anti-dominant.
\end{enumerate}
\item If $\mu$ is nonintegral and
$\Phi_{[\mu^{\rho_0}]}^\vee$ is pseudo-maximal, then the possible GK dimensions are $\{3,5,6\}$ for
$\Phi_{[\mu^{\rho_0}]}\cong A_2$, and $\{4,5,6\}$ for $\Phi_{[\mu^{\rho_0}]}\cong A_1\times\widetilde A_1$.
\end{enumerate}
Here the label ‘˜’ as in $A_1 \times \tilde{A}_1 \subseteq G_2$ is attached to the connected component corresponding to short roots and $\Phi_{[\mu]}$ is the root system corresponding to integral Weyl subgroup $W_{[\mu]}$.
\end{Lem}
\begin{Rem}
   $\Phi_{[\mu]}$ in the above lemma and $\Delta_{[\mu]}$ in section 3 are essentially the same, and we should distinguish the two notations because $\Phi_{[\mu]}$ starts from the root system of a Lie algebra and $\Delta_{[\mu]}$ starts from the root system of a Lie superalgebra.
\end{Rem}

\begin{Prop}\label{prop:G3-even-GK-values}
Let $\frg=G(3)$. The possible GK dimensions of simple highest weight modules $L_{\frg_0}(\Lambda)$ are exactly
$\{0,1,3,4,5,6,7\}$.
\end{Prop}
\begin{proof}
This is a direct corollary of Lemma \ref{GK-DIM-G2} and the fact that the GK dimensions of simple highest weight $\mathfrak{sl}_2$-module can only achieve 0 and 1.
\end{proof}

\begin{Exam}
     Let $\frg=G(3)$. Suppose that $$\Lambda^{\rho_0}=(1|\frac{1}{2},\frac{3}{2},-\frac{3}{2})\in\frh^*.$$
By Example \ref{exam-G(3)} we have {\bf a}$(w_\Lambda)=7$, then $\text{GK} \dim L_{\frg_0}(\Lambda)=0$.
     
\end{Exam}

\section{GK dimensions of $\widetilde{L}(\Lambda)$ for basic classical Lie superalgebras}\label{section-GK-basic}
\subsection{GK dimensions of $\widetilde{L}(\Lambda)$ for basic classical  Lie superalgebras of type I}
\begin{Thm}\label{GK-dim-type I}
Let $\frg$ be of type I in \eqref{classical-type}, with the distinguished Borel subalgebra, and let $\Lambda=(\lambda_1,\ldots,\lambda_m\mid\mu_1,\ldots,\mu_n)$ be integral. Then
\[
\resizebox{\linewidth}{!}{$
\text{GK}\dim \widetilde{L}(\lambda|\mu) =
\begin{cases}
\frac{m(m-1)}{2}+\frac{n(n-1)}{2}-( \sum\limits_{i\geq 1}(i-1){\bf p}(\lambda^{\rho_0})_i+\sum\limits_{j\geq 1}(j-1){\bf p}(\mu^{\rho_0})_j) & \text{ if } \frg=\frsl(m|n), \\
n^2-\sum\limits_{i\geq 1}(i-1){\bf p}((\mu^{\rho_0})^-)_i^{\odd} & \text{ if } \frg=\mathfrak{osp}(2|2n).
\end{cases}
$}
\]
    
\end{Thm}
\begin{proof}
    It can be obtained directly from Proposition \ref{semisimple-GK-dim} and Proposition \ref{GK-dim-L_g0}.
\end{proof}

\begin{Thm}\label{GK-dim-type I-nonintegral}
Let $\frg$ be of type I in \eqref{classical-type}, with the distinguished Borel subalgebra, and let $\Lambda=(\lambda_1,\lambda_2,\cdots,\lambda_m\mid\mu_1,\mu_2,\cdots,\mu_n)\in\frh^*$ be arbitrary. Then
\begin{enumerate}
\item For $\frg=\frsl(m|n)$,
\[
\GKD\widetilde L(\lambda\mid\mu)
=\frac{m(m-1)}2+\frac{n(n-1)}2
-\sum_{\lambda,\mu}F_A((\lambda^{\rho_0})_z\|(\mu^{\rho_0})_z).
\]
\item For $\frg=\mathfrak{osp}(2|2n)$,
\[
\begin{aligned}
\GKD\widetilde L(\lambda\mid\mu)=n^2-F_B((\mu^{\rho_0})_0^-)
-F_D((\mu^{\rho_0})_{\frac12}^-)-\sum_{\mu}
F_A(\widetilde{(\mu^{\rho_0})}_z).
\end{aligned}
\]
\end{enumerate}
\end{Thm}
\begin{proof}
    It can be obtained directly from Proposition \ref{semisimple-GK-dim} and Theorem \ref{GKdim-thm}.
\end{proof}

\begin{Rem}
Theorem \ref{GK-dim-type I} and Theorem \ref{GK-dim-type I-nonintegral} will hold in $\frsl(m|n)$ when $m=n$. 
\end{Rem}

\begin{Rem}
The irreducible $\mathfrak{g}$-module is finite-dimensional if and only if the highest weight $\Lambda$ is dominant integral and it satisfies a condition which can be expressed in terms of all of the Borel subalgebras containing $\mathfrak{h}$ (in the ordinary theory they are conjugate, while in the super-theory they are not). Furthermore, one obtains all irreducible finite-dimensional representations of $\mathfrak{g}$ in this manner, see \cite{Kac78,Kac77,Mu12,CW12,GS10,CLW,SZ}. Hence in super-theory (type II case), $\mathrm{GK \ dim}\, L_{\mathfrak{g}_0}(\Lambda) = 0$ does not imply that $\widetilde{L}(\Lambda)$ is finite-dimensional, but in type I case, we have $\widetilde{L}(\Lambda)$ is finite-dimensional if and only if $\mathrm{GK \ dim}\, L_{\frg_0}(\Lambda) = 0$.

\end{Rem}

\begin{Prop}
If $\frg$ is a type I Lie superalgebra, then $\widetilde{L}(\Lambda)$ is finite-dimensional if and only if $\text{GK}\dim L_{\frg_0}(\Lambda)=0$.



\end{Prop}
\begin{proof}

  If $\frg$ is of type I, recall that Kac module is defined by
$$\mathcal{K}(\Lambda):=U(\frg)\otimes_{U(\frg(0)+\frg(1))}L_0(\Lambda)\cong \wedge(\frg(-1))\otimes L_0(\Lambda),$$
where $L_0(\Lambda)$ is a simple $\frg(0)$-module of highest weight $\Lambda$. Then the following are equivalent:
\begin{enumerate}[label=(\alph*)]
    \item $\mathcal{K}(\Lambda)$ is finite-dimensional.

    \item $L_0(\Lambda)$  is finite-dimensional.

    \item $\widetilde{L}(\Lambda)$ is finite-dimensional.

    \item $\text{GK}\dim L_{\frg_0}(\Lambda)=0$.

\end{enumerate}
We prove the equivalences as follows.

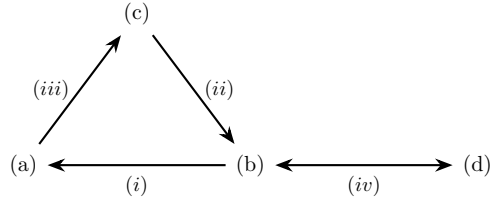
\begin{center}
    \begin{tikzpicture}[>=Stealth]
    \node (A) at (0,0) { \scalebox{0.75}{(a)}};
    \node (B) at (3,0) { \scalebox{0.75}{(b)}}; 
    \node (C) at (1.5,2) { \scalebox{0.75}{(c)}}; 
    \node (D) at (6,0) { \scalebox{0.75}{(d)}};
    \node (A1) at (0.1,0.3) {};
    \node (B1) at (6.2,0.2) {};
    \node (C1) at (3.55,1.75) {};
    \node (D1) at (2.7,1.82) {};
    \draw[<-, thick] (A) -- node[below] { \scalebox{0.7}{$(i)$}} (B);
    \draw[<-, thick] (B) -- node[right] { \scalebox{0.7}{$(ii)$}} (C);
    \draw[<-, thick] (C) -- node[left] {\scalebox{0.7}{$(iii)$}} (A);
    \draw[<->, thick] (B) -- node[below] {\scalebox{0.7}{$(iv)$}} (D);


\end{tikzpicture}
\end{center}

$(i)$ holds since $\wedge(\frg(-1))$ is finite-dimensional.

$(ii)$ holds since \(L_0(\Lambda)\) is an irreducible direct summand of \(\widetilde{L}(\Lambda)\) regarded as a \(\mathfrak{g}_0\)-module.

$(iii)$ holds since the map $\mathcal{K}(\Lambda)\to\widetilde{L}(\Lambda)$ is surjective.

$(iv)$ holds since $L_0(\Lambda)$  is finite-dimensional $\Longleftrightarrow 0=\text{GK}\dim L_0(\Lambda)=\text{GK}\dim L_{\frg_0}(\Lambda)$.

\end{proof}

\subsection{GK dimensions of $\widetilde{L}(\Lambda)$ for $D(m,n)$}
Let $\frg=\mathfrak{osp}(2m|2n)$, $m\geq2$, $n\geq1$. There are two kinds of Dynkin diagrams and corresponding Borel subalgebras for $\mathfrak{osp}(2m|2n):$
\begin{itemize}
    \item Diagrams of $\mid$-shape, i.e., Dynkin diagrams with a long simple root $\pm2\delta_i$.
    \item Diagrams of {\Large $\Ydown$}-shape, i.e., Dynkin diagrams with no long simple root.
\end{itemize}
The {\Large $\Ydown$}-system has ordered sequence $\delta_1\cdots\delta_n\eps_1\cdots\eps_m$ and simple system
\begin{align*}
\Pi_{\Ydown}=\{\delta_j-\delta_{j+1}:j<n\}\cup\{\delta_n-\eps_1\}\cup\{\eps_i-\eps_{i+1}:i<m\}\cup\{\eps_{m-1}+\eps_m\}.
\end{align*}
The $\mid$-system has ordered sequence
$\eps_1\cdots\eps_m\delta_1\cdots\delta_n$ and simple system
\begin{align*}
\Pi_{\mid}=\{\eps_i-\eps_{i+1}:i<m\}\cup\{\eps_m-\delta_1\}\cup\{\delta_j-\delta_{j+1}:j<n\}\cup\{2\delta_n\}.
\end{align*}

The two systems have the same positive even roots. Their relevant inclusions are $\Pi_D\subseteq\Pi_{\Ydown}$ and $\Pi_C\subseteq\Pi_{\mid}$. These allow the corresponding highest weight vectors to generate simple modules for $\mathfrak{so}(2m)$ and $\mathfrak{sp}(2n)$, respectively; see \cite[Lemma 1.2.2(i)]{GG21}. This is why the growth formula below uses $\lambda_{\Ydown}$ for the orthogonal factor and $\mu_{\mid}$ for the symplectic factor. The equality of the two descriptions of $V$ alone does not establish that formula.

For $s\in\{{\Ydown},|\}$, the corresponding $\mathbb{Z}$-grading is
\begin{equation}\label{eq:fivegrading}
 \g=\bigoplus_{k=-2}^{2}\g^{(s)}(k),\qquad\g_0=\g^{(s)}(-2)\oplus\g^{(s)}(0)\oplus\g^{(s)}(2),\qquad\g_1=\g^{(s)}(-1)\oplus\g^{(s)}(1).
\end{equation}

\begin{table}[ht]
\centering\small
\renewcommand{\arraystretch}{1.22}
\begin{tabular}{@{}p{.13\linewidth}p{.32\linewidth}p{.31\linewidth}c@{}}
\toprule
Shape & Positive odd roots $\Delta_1^+$ &
 Levi $\g^{(s)}(0)$ & $\dim\g^{(s)}(2)$\\
\midrule
${\Ydown}$ & $\{\delta_j\pm\eps_i\}$ &
$\so(2m)\oplus\gl(n)$ & $n(n+1)/2$\\
$|$ & $\{\eps_i\pm\delta_j\}$ &
$\frgl(m)\oplus\frsp(2n)$ & $m(m-1)/2$\\
\bottomrule
\end{tabular}
\caption{$\mathbb{Z}$-gradings for type $D(m,n)$}
\label{tab:roots}
\end{table}
Write
\begin{equation}\label{eq:samemodule}
 V\cong\widetilde{L}(\Lambda_{\Ydown})\cong\widetilde{L}(\Lambda_{|}),\qquad
 \Lambda_{\Ydown}=(\lambda_{\Ydown}\mid\mu_{\Ydown}),\quad
 \Lambda_{|}=(\lambda_|\mid\mu_|).
\end{equation}
Denote $\frp_s=\frg^{(s)}(0)\oplus\frg^{(s)}(1)\oplus\frg^{(s)}(2)$, let $L_0^{s}(\Lambda_s)$ be the simple highest weight module for $\frg^{(s)}(0)$.

\begin{Prop}\label{D(m,n)}
Let $\mathfrak g=\mathfrak{osp}(2m|2n)$ and $V$ be a simple highest weight $\frg$-module.
Let $\Pi_{\Ydown}$ and $\Pi_{\mid}$ be the two simple systems considered above. Write the corresponding highest weights as $\Lambda_{\Ydown}=(\lambda_{\Ydown}|\mu_{\Ydown})$ and
$\Lambda_{\mid}=(\lambda_{\mid}|\mu_{\mid})$. Then
\[
\AV_{\frg_0}V
=\AV_{\mathfrak{so}(2m)}L_{\mathfrak{so}(2m)}(\lambda_{\Ydown})
 \times\AV_{\mathfrak{sp}(2n)}L_{\mathfrak{sp}(2n)}(\mu_{\mid}),
\]
and
\[
\GKD V=\GKD L_{\mathfrak{so}(2m)}(\lambda_{\Ydown})
       +\GKD L_{\mathfrak{sp}(2n)}(\mu_{\mid}).
\]
\end{Prop}
\begin{proof}
Let $v_{\Ydown}$ and $v_{\mid}$ be homogeneous highest weight vectors for $\Pi_{\Ydown}$ and $\Pi_{\mid}$, respectively. We first show that $U(\mathfrak{so}(2m))v_{\Ydown}$ is simple. Otherwise, a nonzero proper submodule contains a highest weight vector $u=xv_{\Ydown}$ of lower weight, with $x\in U(\mathfrak n_D^-)$. For $\beta\in\Pi_{\Ydown}\setminus\Pi_D$, the relation $[e_\beta,f_\alpha]=0$ for $\alpha\in\Pi_D$ gives $e_\beta u=0$. Thus $u$ is a $\mathfrak g$-highest weight vector of lower weight, contradicting the simplicity of $V$. Hence $U(\mathfrak{so}(2m))v_{\Ydown}\simeq L_{\mathfrak{so}(2m)}(\lambda_{\Ydown})$. Similarly, $U(\mathfrak{sp}(2n))v_{\mid}\simeq L_{\mathfrak{sp}(2n)}(\mu_{\mid})$.

The highest weight $\frg_0$-module $U(\frg_0)v_{\Ydown}$ is a quotient of a Verma module, so it has finite length and contains a simple submodule $S$. It lies in the parity of $v_{\Ydown}$. Proposition \ref{lem:even-submodule-support} gives $\AV_{\frg_0}V=\AV_{\frg_0}S$.

Since $\mathfrak g_0=\mathfrak{so}(2m)\oplus\mathfrak{sp}(2n)$, we may identify $S$ with $S_D\otimes_{\mathbb C}S_C$, where $S_D$ and $S_C$ act on its corresponding factor. Choose highest weight vectors
$w_D\in S_D$ and $w_C\in S_C$, and set $w=w_D\otimes w_C$. Then
\[
U(\mathfrak{so}(2m))w\simeq S_D,
\qquad
U(\mathfrak{sp}(2n))w\simeq S_C.
\]
Applying Lemma \ref{lem:cyclic-support} to $w,v_{\Large\Ydown}$ and  $w,v_{\mid}$, we obtain
\begin{align*}
\operatorname{AV}_{\mathfrak{so}(2m)}S_D
&=
\operatorname{AV}_{\mathfrak{so}(2m)}
L_{\mathfrak{so}(2m)}(\lambda_{\Large\Ydown}),\\
\operatorname{AV}_{\mathfrak{sp}(2n)}S_C
&=
\operatorname{AV}_{\mathfrak{sp}(2n)}
L_{\mathfrak{sp}(2n)}(\mu_{\mid}).
\end{align*}
Tensor products of good filtrations give
\[
\operatorname{AV}_{\mathfrak g_0}S
=
\operatorname{AV}_{\mathfrak{so}(2m)}S_D
\times
\operatorname{AV}_{\mathfrak{sp}(2n)}S_C.
\]
Together with
$\operatorname{AV}_{\mathfrak g_0}V
=\operatorname{AV}_{\mathfrak g_0}S$,
this proves the asserted product formula. 
\end{proof}

\begin{Prop}\label{prop:odd-shuffle-BD}
Let $\Lambda_{\Large\Ydown}$ and $\Lambda_{\mid}$ be the highest weights of the same simple module $V$ with respect to $\Pi_{\Large\Ydown}$ and $\Pi_{\mid}$, respectively. Initialize $\ell_i=\lambda_{\Large\Ydown,i}$ and $u_j=\mu_{\Large\Ydown,j}$. For $i=1,\ldots,m$, successively process $j=n,n-1,\ldots,1$ using
\[
(\ell_i,u_j)\longmapsto
\begin{cases}
(\ell_i,u_j), & \text{if }\ell_i+u_j=0,\\
(\ell_i+1,u_j-1), & \text{if }\ell_i+u_j\neq0.
\end{cases}
\]
Each step uses the current coordinates and leaves all other coordinates unchanged. Upon completion,
\[
\Lambda_{\mid}=(\ell_1,\ldots,\ell_m\mid u_1,\ldots,u_n).
\]
Conversely, starting from the coordinates of $\Lambda_{\mid}$, process $i=m,m-1,\ldots,1$ and for each $i$, $j=1,\ldots,n$. Keep the pair unchanged when $\ell_i+u_j=0$, and replace it by $(\ell_i-1,u_j+1)$ otherwise. The final coordinates give $\Lambda_{\Large\Ydown}$. The same procedures will be applied to $B(m,n)$ specified in Subsection \ref{sec:B-GK-algorithm}.
\end{Prop}
\begin{proof}
The prescribed sequence of odd reflections at
$\delta_j-\epsilon_i$, with $i=1,\ldots,m$ and, for each fixed $i$,
$j=n,\ldots,1$, transforms $\Pi_{\Large\Ydown}$ into $\Pi_{\mid}$.
At each step, $\beta=\delta_j-\epsilon_i$ is an isotropic odd simple
root of the current system. For the current highest weight $\Lambda$,
we have
\[
(\Lambda,\beta)=-(\ell_i+u_j).
\]
By \cite[Lemma 1.40]{CW12}, the highest weight remains unchanged if
$\ell_i+u_j=0$, and becomes $\Lambda-\beta$ otherwise.
Since subtracting $\beta$ increases $\ell_i$ by one and decreases
$u_j$ by one, this gives the stated update.
The inverse algorithm follows by reversing the reflection sequence,
because each update preserves $\ell_i+u_j$.
The same argument applies to the two Borel systems in type $B(m,n)$.
\end{proof}

\begin{Exam}
Let $\mathfrak g=\mathfrak{osp}(8\mid12)$ and let $V$ be a simple highest weight module with highest weight
\[
\Lambda^{(0)}=\Lambda_{\Large\Ydown}=(3,2,1,0\mid1,0,1,0,0,0).
\]
Denote by $\Lambda^{(k)}$ the highest weight after the $k$-th odd reflection. At the step corresponding to $\beta=\delta_j-\epsilon_i$,
we use the current coordinates to compute
\[
\Lambda^{(k)}
=
\begin{cases}
\Lambda^{(k-1)}, & \ell_i+u_j=0,\\
\Lambda^{(k-1)}-\beta, & \ell_i+u_j\ne0.
\end{cases}
\]
The value of $\ell_i+u_j$ in each table is taken immediately before the corresponding reflection.

For $i=1$, the six reflections give
\begin{table}[htbp]
    \centering
    \begin{tabular}{cccl} 
        \toprule
        $k$ & $\beta$ & $\ell_1+u_j$ & $\Lambda^{(k)}$ \\
        \midrule
        1 & $\delta_6-\epsilon_1$ & 3 & $(4,2,1,0 \mid 1,0,1,0,0,-1)$ \\
        2 & $\delta_5-\epsilon_1$ & 4 & $(5,2,1,0 \mid 1,0,1,0,-1,-1)$ \\
        3 & $\delta_4-\epsilon_1$ & 5 & $(6,2,1,0 \mid 1,0,1,-1,-1,-1)$ \\
        4 & $\delta_3-\epsilon_1$ & 7 & $(7,2,1,0 \mid 1,0,0,-1,-1,-1)$ \\
        5 & $\delta_2-\epsilon_1$ & 7 & $(8,2,1,0 \mid 1,-1,0,-1,-1,-1)$ \\
        6 & $\delta_1-\epsilon_1$ & 9 & $(9,2,1,0 \mid 0,-1,0,-1,-1,-1)$ \\
        \bottomrule
    \end{tabular}
\end{table}

For $i=2$, we obtain

\begin{table}[htbp]
    \centering
    \begin{tabular}{cccl} 
        \toprule
        $k$ & $\beta$ & $\ell_2+u_j$ & $\Lambda^{(k)}$ \\
        \midrule
        7  & $\delta_6-\epsilon_2$ & 1 & $(9,3,1,0 \mid 0,-1,0,-1,-1,-2)$ \\
        8  & $\delta_5-\epsilon_2$ & 2 & $(9,4,1,0 \mid 0,-1,0,-1,-2,-2)$ \\
        9  & $\delta_4-\epsilon_2$ & 3 & $(9,5,1,0 \mid 0,-1,0,-2,-2,-2)$ \\
        10 & $\delta_3-\epsilon_2$ & 5 & $(9,6,1,0 \mid 0,-1,-1,-2,-2,-2)$ \\
        11 & $\delta_2-\epsilon_2$ & 5 & $(9,7,1,0 \mid 0,-2,-1,-2,-2,-2)$ \\
        12 & $\delta_1-\epsilon_2$ & 7 & $(9,8,1,0 \mid -1,-2,-1,-2,-2,-2)$ \\
        \bottomrule
    \end{tabular}
\end{table}

For $i=3$, we have
\begin{table}[htbp]
    \centering
    \begin{tabular}{cccl} 
        \toprule
        $k$ & $\beta$ & $\ell_3+u_j$ & $\Lambda^{(k)}$ \\
        \midrule
        13 & $\delta_6-\epsilon_3$ & $-1$ & $(9,8,2,0 \mid -1,-2,-1,-2,-2,-3)$ \\
        14 & $\delta_5-\epsilon_3$ & 0  & $(9,8,2,0 \mid -1,-2,-1,-2,-2,-3)$ \\
        15 & $\delta_4-\epsilon_3$ & 0  & $(9,8,2,0 \mid -1,-2,-1,-2,-2,-3)$ \\
        16 & $\delta_3-\epsilon_3$ & 1  & $(9,8,3,0 \mid -1,-2,-2,-2,-2,-3)$ \\
        17 & $\delta_2-\epsilon_3$ & 1  & $(9,8,4,0 \mid -1,-3,-2,-2,-2,-3)$ \\
        18 & $\delta_1-\epsilon_3$ & 3  & $(9,8,5,0 \mid -2,-3,-2,-2,-2,-3)$ \\
        \bottomrule
    \end{tabular}
\end{table}

Finally for $i=4$,  we have
\begin{table}[htbp]
    \centering
    \begin{tabular}{cccl} 
        \toprule
        $k$ & $\beta$ & $\ell_4+u_j$ & $\Lambda^{(k)}$ \\
        \midrule
        19 & $\delta_6-\epsilon_4$ & $-3$ & $(9,8,5,1 \mid -2,-3,-2,-2,-2,-4)$ \\
        20 & $\delta_5-\epsilon_4$ & $-1$ & $(9,8,5,2 \mid -2,-3,-2,-2,-3,-4)$ \\
        21 & $\delta_4-\epsilon_4$ & 0  & $(9,8,5,2 \mid -2,-3,-2,-2,-3,-4)$ \\
        22 & $\delta_3-\epsilon_4$ & 0  & $(9,8,5,2 \mid -2,-3,-2,-2,-3,-4)$ \\
        23 & $\delta_2-\epsilon_4$ & $-1$ & $(9,8,5,3 \mid -2,-4,-2,-2,-3,-4)$ \\
        24 & $\delta_1-\epsilon_4$ & 1  & $(9,8,5,4 \mid -3,-4,-2,-2,-3,-4)$ \\
        \bottomrule
    \end{tabular}
\end{table}

Thus
\[
\Lambda_{\mid}=\Lambda^{(24)}=(9,8,5,4\mid-3,-4,-2,-2,-3,-4).
\]
\end{Exam}

\begin{Thm}\label{thm:D-GK-all}
Let $\frg=\mathfrak{osp}(2m|2n)$, $m\geq2$, $n\geq1$, and let $V=\widetilde{L}(\Lambda_{\Ydown})$. Obtain $\mu_{\mid}$ by Proposition \ref{prop:odd-shuffle-BD}. For every complex highest weight, without an integrality or typicality hypothesis, we have
\begin{equation}\label{eq:D-GK-all}
\begin{aligned}
\GKD V={}&m(m-1)+n^2
-F_D\bigl((\lambda_{\Ydown}^{\rho_0})_0^-\bigr)
-F_D\bigl((\lambda_{\Ydown}^{\rho_0})_{\frac12}^-\bigr)\\
&-F_B\bigl((\mu_{\mid}^{\rho_0})_0^-\bigr)
-F_D\bigl((\mu_{\mid}^{\rho_0})_{\frac12}^-\bigr)-\sum_{\lambda_{\Ydown},\mu_{\mid}}
 F_A\bigl(\widetilde{(\lambda_{\Ydown}^{\rho_0})}_z||\widetilde{(\mu_{\mid}^{\rho_0})}_z\bigr).
\end{aligned}
\end{equation}
\end{Thm}
\begin{proof}
It follows from Proposition \ref{D(m,n)} and \cite[Theorem 5.7]{BXX}.
\end{proof}

\begin{Cor}\label{cor:D-GK-integral}
If $\Lambda_{\Ydown}$ is integral, then
\begin{equation}\label{eq:D-GK-integral}
\GKD V=m(m-1)+n^2
-F_D\bigl((\lambda_{\Ydown}^{\rho_0})^-\bigr)
-F_B\bigl((\mu_{\mid}^{\rho_0})^-\bigr).
\end{equation}
\end{Cor}
\begin{proof}
Formula \eqref{eq:D-GK-all} therefore reduces to \eqref{eq:D-GK-integral}.
\end{proof}

\begin{Exam}\label{ex:D-GK-integral}
For $\frg=\mathfrak{osp}(6|4)$ and $\Lambda_{\Ydown}=(2,-1,0\mid0,-2)$, the algorithm gives $\Lambda_{\mid}=(3,1,2\mid-3,-4)$. Then $\lambda_{\Ydown}^{\rho_0}=(4,0,0)$ and $\mu_{\mid}^{\rho_0}=(-1,-3)$, whose doubled RS shapes are $[4,1,1]$ and $[2,2]$. Hence $\GKD V=6+4-2-1=7$.
\end{Exam}

\begin{Exam}\label{ex:D-GK-nonintegral}
For $\frg=\mathfrak{osp}(4|2)$ and $\Lambda_{\Ydown}=(\frac13,\frac13\mid\frac12)$, both sums are nonzero, giving $\Lambda_{\mid}=(\frac43,\frac43\mid-\frac32)$. Thus $\lambda_{\Ydown}^{\rho_0}=(\frac43,\frac13)$ and $\mu_{\mid}^{\rho_0}=(-\frac12)$. Then $\GKD V=2+1-1=2$.
\end{Exam}

\begin{Lem}\label{lem:root}
Let $V$ be a simple highest weight $\g$-module and let $\beta\in\Delta_0$ be simple in $\Pi_s$.  Write $v_s$ for a highest weight vector and $k_\beta= (\Lambda_s,\beta^\vee )$.
For a nonzero negative even root vector $f_\beta$:
\begin{enumerate}
\item If $k_\beta\in\mathbb{Z}_{\geq0}$, then $f_\beta$ is locally nilpotent on $V$.
\item If $k_\beta\notin\mathbb{Z}_{\geq0}$, then $f_\beta$ acts injectively on $V$.
\end{enumerate}
\end{Lem}
\begin{proof}
    \begin{enumerate}
        \item If $k_\beta=r\in\mathbb{Z}_{\geq0}$, $\mathfrak{sl}_2(\beta)$-relations show that $f_\beta^{r+1}v_s$ is annihilated by $e_\beta$.  The other positive simple generators commute with $f_\beta$, so this vector,
if nonzero, would be a lower weight vector of $\frg$.  Simplicity forces $f_\beta^{r+1}v_s=0$. Since $\operatorname{ad}f_\beta$ is locally nilpotent on $U(\g)$, for each $u\in U(\g)$ the identity
\[
f_\beta^Nuv_s=\sum_{i=0}^N\binom Ni(\operatorname{ad}f_\beta)^i(u)f_\beta^{N-i}v_s
\]
vanishes for sufficiently large $N$. As $V=U(\g)v_s$, this proves local nilpotence of $f_\beta$ on $V$.

\item Suppose that $k_\beta\notin\mathbb Z_{\geq 0}$. The
$\mathfrak{sl}_2(\beta)$-relations give
\[
e_\beta f_\beta^{a}v_s=a(k_\beta-a+1)f_\beta^{a-1}v_s \text{ for } a\geq 1.
\]
Choose minimal $a$ such that $f_\beta^{a}v_s=0$. Then $f_\beta^{a-1}v_s\neq 0$, and $k_\beta=a-1\in\mathbb Z_{\geq 0}$, a contradiction. Hence $f_\beta^{a}v_s\neq 0$ for all $a\geq 0$.

Assume that $f_\beta$ is not injective on $V$, and choose a nonzero homogeneous vector $w\in\ker f_\beta$. Since $V$ is simple, $V=U(\mathfrak g)w$. As $f_\beta w=0$ and $\operatorname{ad}f_\beta$ acts locally nilpotently on
$U(\mathfrak g)$, for every $u\in U(\mathfrak g)$ we have $(\operatorname{ad}f_\beta)^{N}(u)w=0$ for $N\gg0$. Thus $f_\beta$ acts locally nilpotently on $V$, contradicting $f_\beta^{a}v_s\neq 0$ for all $a\geq 0$. Therefore   $f_\beta$ acts injectively on $V$.
    \end{enumerate}
\end{proof}

Let $V$ be as in \eqref{eq:samemodule}, denote 
$$ J(V)=\{\beta\in\Pi_D:(\Lambda_{\Ydown},\beta^\vee)\notin\mathbb{Z}_{\geq0}\}\cup\{\beta\in\Pi_C:
           (\Lambda_{|},\beta^\vee)\notin\mathbb{Z}_{\geq0}\}.$$

\begin{Thm}\label{thm--main}
Let $V$ be as in \eqref{eq:samemodule}, with no typicality or integrality hypothesis.  Then
\begin{equation*}
    \operatorname{GK}\dim V\geq |J(V)|.
\end{equation*}
The following conditions are equivalent:
\begin{enumerate}
\item $\operatorname{GK}\dim V=0$;
\item $V$ is finite-dimensional;
\item $J(V)=\varnothing$;
\item both $L_0^{\Ydown}(\Lambda_{\Ydown})$ and $L_0^{|}(\Lambda_{|})$ are finite-dimensional.
\end{enumerate}
\end{Thm}
\begin{proof}
     Write $J(V)=\{\beta_1,\ldots,\beta_r\}$. By Lemma \ref{lem:root}, each $f_{\beta_i}$ acts injectively on $V$. Thus $f_{\beta_1}^{a_1}\cdots f_{\beta_r}^{a_r}v_{\Ydown}\neq 0$ for all $a_i\in\mathbb Z_{\geq0}$. Their weights are distinct because $J(V)\subseteq\Pi_D\sqcup\Pi_C$ is linearly independent. Consequently,
\[
\dim U_d(\mathfrak g)v_{\Ydown}\geq\#\{(a_1,\ldots,a_r)\in\mathbb Z_{\geq0}^r:a_1+\cdots+a_r\leq d\}=\binom{d+r}{r}.
\]
Taking growth exponents gives $\operatorname{GK}\dim V\geq r=|J(V)|$. The equivalence of $(1)$ and $(2)$ is clear.
     
     $(2)\Rightarrow(3)$ If $V$ is finite-dimensional, the action of every negative even root vector is locally nilpotent; Lemma \ref{lem:root} then makes $J(V)$ empty.

     $(3)\Rightarrow(4)$ Suppose that $J(V)=\varnothing$, and set $\Pi_0=\Pi_D\sqcup\Pi_C$. By Lemma \ref{lem:root}, each negative root vector $f_\beta$, $\beta\in\Pi_0$, acts locally nilpotently on $V$.

Fix $s\in\{\Ydown,|\}$, and let $v_s$ be a highest weight vector of weight $\Lambda_s$ for the corresponding Borel subalgebra. We have $e_\beta v_s=0$ for every $\beta\in\Pi_0$. Choose the minimal positive integer $N$ such that $f_\beta^N v_s=0$. The $\mathfrak{sl}_2(\beta)$-relations give
\[
0=e_\beta f_\beta^N v_s=N\bigl( (\Lambda_s,\beta^\vee )-N+1\bigr)f_\beta^{N-1}v_s.
\]
By the minimality of $N$, we have $f_\beta^{N-1}v_s\neq0$, and hence
\[
 (\Lambda_s,\beta^\vee )=N-1\in\mathbb Z_{\geq0}.
\]
Thus $\Lambda_s$ is dominant integral for $\mathfrak g_0$, and in particular for the Levi subalgebra $\mathfrak g^{(s)}(0)$. The finite-dimensional highest weight classification for reductive Lie algebras therefore implies that $L_0^s(\Lambda_s)$ is finite-dimensional for both $s\in\{\Ydown,|\}$.

    $(4)\Rightarrow(1)$ The finite dimensionality of $L_0^{\Large\Ydown}(\Lambda_{\Large\Ydown})$ and
$L_0^{\mid}(\Lambda_{\mid})$ implies that $L_{\mathfrak{so}(2m)}(\lambda_{\Large\Ydown})$ and
$L_{\mathfrak{sp}(2n)}(\mu_{\mid})$ are finite-dimensional. Proposition \ref{D(m,n)} therefore gives
\[
\operatorname{GKdim}V=\operatorname{GKdim}L_{\mathfrak{so}(2m)}(\lambda_{\Large\Ydown})+\operatorname{GKdim}
L_{\mathfrak{sp}(2n)}(\mu_{\mid})=0.
\]
\end{proof}

\subsection{GK dimensions of $\widetilde{L}(\Lambda)$ for $B(m,n)$}\label{sec:B-GK-algorithm}
Let $\frg=\mathfrak{osp}(2m+1|2n)$, $m,n\geq1$. We retain the subscripts ${\Ydown}$ and $\mid$ for the orderings $\delta_1\cdots\delta_n\eps_1\cdots\eps_m$ and $\eps_1\cdots\eps_m\delta_1\cdots\delta_n$, respectively. In type $B$, ${\Ydown}$ is an ordering label inherited from type $D$, and both diagrams are linear. The two simple systems are
\begin{align}\label{eq:B-two-simple-systems}
\Pi_{\Ydown}&=\{\delta_j-\delta_{j+1}:j<n\}\cup\{\delta_n-\eps_1\}
 \cup\{\eps_i-\eps_{i+1}:i<m\}\cup\{\eps_m\},\notag\\
\Pi_{\mid}&=\{\eps_i-\eps_{i+1}:i<m\}\cup\{\eps_m-\delta_1\}
 \cup\{\delta_j-\delta_{j+1}:j<n\}\cup\{\delta_n\}.
\end{align}
They have the same positive even roots. The first is the standard distinguished system; the second has an isotropic odd simple root $\eps_m-\delta_1$ and a nonisotropic odd simple root $\delta_n$. Write the highest weights of the same simple module as $\Lambda_{\Ydown}=(\lambda_{\Ydown}\mid\mu_{\Ydown})$ and $\Lambda_{\mid}=(\lambda_{\mid}\mid\mu_{\mid})$. Proposition \ref{prop:odd-shuffle-BD} computes either weight from the other.

\subsubsection{The support and growth formulas}
The simple roots $\delta_j-\delta_{j+1}$ and $\delta_n$ of $\Pi_{\mid}$ generate $\mathfrak t=\mathfrak{osp}(1|2n)$, whose even part is $\mathfrak{sp}(2n)$.

The $B_m$-simple roots are contained in $\Pi_{\Ydown}$, while the simple roots of $\mathfrak t$ are contained in $\Pi_{\mid}$. Hence the two highest weight vectors generate the simple modules of these two subalgebras. Although $\mathfrak t_0=\mathfrak{sp}(2n)$, the restriction of $L_{\mathfrak t}(\mu_{\mid})$ need not be the simple module $L_{\mathfrak{sp}(2n)}(\mu_{\mid})$. This accounts for the final factor in the following formula; the formula is an equality of associated varieties and growth, rather than a tensor product decomposition of $V$.

\begin{Prop}\label{prop:B-two-Borel}
For every complex highest weight, we have
\[
\AV_{\frg_0}V
=\AV_{\mathfrak{so}(2m+1)}L_{\mathfrak{so}(2m+1)}(\lambda_{\Ydown})
 \times\AV_{\mathfrak{sp}(2n)}L_{\frt}(\mu_{\mid}),
\]
and
\begin{equation}\label{eq:B-two-Borel}
\GKD V=\GKD L_{\mathfrak{so}(2m+1)}(\lambda_{\Ydown})
 +\GKD L_{\mathfrak t}(\mu_{\mid}).
\end{equation}
\end{Prop}
\begin{proof}
Let $v_{\Ydown},v_{\mid}$ be homogeneous highest weight vectors. The simple roots of $B_m$ are contained in $\Pi_{\Ydown}$, and those of $\mathfrak t$ are contained in $\Pi_{\mid}$. The   argument in Proposition \ref{D(m,n)}, applied to these two subsystems, gives $U(\mathfrak{so}(2m+1))v_{\Ydown}\simeq L_{\mathfrak{so}(2m+1)}(\lambda_{\Ydown})$ and $U(\mathfrak t)v_{\mid}\simeq L_{\mathfrak t}(\mu_{\mid})$; see also \cite[Lem. 1.2.2(i)]{GG21}.

Choose a simple $\frg_0$-submodule $S$ of $U(\frg_0)v_{\Ydown}$. It lies in one parity, so Proposition \ref{lem:even-submodule-support} gives $\AV_{\frg_0}V=\AV_{\frg_0}S$. Write $S=S_B\otimes_{\mathbb{C}} S_C$ and choose its highest vector $w$. By  Lemma \ref{lem:cyclic-support}, we have
\begin{align*}
\AV_{\mathfrak{so}(2m+1)}S_B
 &=\AV_{\mathfrak{so}(2m+1)}U(\mathfrak{so}(2m+1))v_{\Ydown},\\
\AV_{\mathfrak{sp}(2n)}S_C
 &=\AV_{\mathfrak{sp}(2n)}U(\mathfrak{sp}(2n))v_{\mid}.
\end{align*}
Hence the support of $V$ is the product of these supports.

For $L_{\frt}(\mu_{\mid})=U(\frt)v_{\mid}$, PBW Theorem gives $U(\frt)=E_{\frt}U(\mathfrak{sp}(2n))$, where $E_{\frt}$ is finite-dimensional and stable under the adjoint action of $\mathfrak{sp}(2n)$. Multiplication defines a surjection
\[
E_{\frt}\otimes U(\mathfrak{sp}(2n))v_{\mid}\twoheadrightarrow L_{\frt}(\mu_{\mid}),
\]
and $U(\mathfrak{sp}(2n))v_{\mid}$ is a submodule of this restriction. Corollary \ref{cor:finite-tensor} gives
\[
\AV_{\mathfrak{sp}(2n)}L_{\frt}(\mu_{\mid})
=\AV_{\mathfrak{sp}(2n)}U(\mathfrak{sp}(2n))v_{\mid}.
\]
This proves the asserted product formula. 
\end{proof}

\subsubsection{GK formula for $\mathfrak{osp}(1|2n)$}
For $\mathfrak t=\mathfrak{osp}(1|2n)$, put $\rho_{\mathfrak t}=(n-\frac12,n-\frac32,\ldots,\frac12)$.

\begin{Lem}\label{lem:osp-terminal-GK}
Let $\mu\in\mathbb C^n$ and $\xi=\mu+\rho_{\mathfrak t}$. Then we have
\begin{equation}\label{eq:osp-terminal-GK}
\GKD L_{\mathfrak t}(\mu)=n^2-F_B(\xi_{\frac12}^-)-F_D(\xi_0^-)-\sum_{\xi}F_A(\widetilde\xi_z).
\end{equation}
\end{Lem}
\begin{proof}
Put $\xi=\mu+\rho_{\frt}$ and set
\[
R^+=\{\delta_i-\delta_j,\delta_i+\delta_j:i<j\}
\cup\{\delta_i:1\leq i\leq n\},
\qquad R=R^+\cup(-R^+).
\]
Let
\[
\Phi_\mu=\{\alpha\in\Delta_0(\frt):
(\mu,\alpha^\vee)\in\Z\},
\]
and obtain $\Psi\subseteq R$ by replacing each root
$\pm2\delta_i$ of $\Phi_\mu$ by $\pm\delta_i$.
Write $\Psi^+=\Psi\cap R^+$, and let $\Pi_\Psi$ and $W_\Psi$
be its simple system and Weyl group.

Since $\rho_{\frt,i}\pm\rho_{\frt,j}\in\Z$, we have
\[
\mu_i\pm\mu_j\in\Z
\Longleftrightarrow \xi_i\pm\xi_j\in\Z,
\qquad
\mu_i\in\Z
\Longleftrightarrow \xi_i\in\tfrac12+\Z.
\]
Thus $\Psi$ has a $B$ block on $\xi_{1/2}$, a $D$ block
on $\xi_0$, and type $A$ blocks on the remaining pairs
of cosets. In particular, $\xi$ is integral for $\Psi$.
Empty blocks and $D_1$ have no roots.

Choose $\eta\in W_\Psi\xi$ dominant for $\Psi^+$, and put
\[
\nu=\eta-\rho_{\frt},
\qquad
J=\{\alpha\in\Pi_\Psi:(\eta,\alpha^\vee)=0\}.
\]
Choose $w$ with $w\eta=\xi$ to be the longest element
of $wW_J$. Then $w\alpha\in-\Psi^+$
for every $\alpha\in J$.

The defining integrality conditions imply
$W_\Psi\xi\subseteq\xi+\Z^n$: reflections in
$\delta_i\pm\delta_j$ change $\xi$ by integer coordinates,
while a single-coordinate reflection occurs only when
$\xi_i\in\frac12+\Z$. Hence
\[
\nu\in\mu+\Z^n,\qquad \Phi_\nu=\Phi_\mu.
\]
For $\mathcal C=\mu+\Z^n$, dominance of $\eta$ and the
choice of $w$ therefore give
\[
\nu\in\mathcal C^+,\qquad
\widetilde B_\nu^0\subseteq\tau_{\mathcal C}(w)
\]
in the notation of \cite[\S0.5]{Mu97}.

Choose \(\zeta \in \mathcal{C}^{++}\) and \(\theta \in \mathcal{C}_0^{++}\) as regular comparison parameters for \(W_\Psi\).
Their integral Weyl groups are \(W_\Psi\).
By \cite[Theorem A]{Mu97}, their regular Verma composition matrices agree. The ordinary Kazhdan--Lusztig description identifies this matrix with the regular matrix for the positive integral Coxeter system of \(\Psi\). Let \((a_{u,v})\) denote its inverse, with dominant-parameter indexing.

Since \(\nu \in \mathcal{C}^+\) and \(\widetilde{B}_\nu^0 \subseteq \tau_{\mathcal{C}}(w)\), \cite[\S 3.8]{Mu97} gives
\[
\begin{aligned}
\operatorname{ch} L_{\mathfrak{t}}(\mu) &= \sum_{y \in W_\Psi} a_{w,y} \operatorname{ch} M_{\mathfrak{t}}(y\eta - \rho_{\mathfrak{t}}).
\end{aligned}
\]

Put $\rho_\Psi=\frac12\sum_{\alpha\in\Psi^+}\alpha$.
For the ordinary reductive Lie algebra with Cartan
subalgebra $\h$ and root system $\Psi$, the same inverse
matrix and the ordinary translation principle give
\[
\ch L_\Psi(\xi-\rho_\Psi)
=\sum_{y\in W_\Psi}a_{w,y}
  \ch M_\Psi(y\eta-\rho_\Psi).
\]
Using ordinary weight characters, PBW and
\[
\frac{1+e^{-\delta_i}}{1-e^{-2\delta_i}}
=\frac{1}{1-e^{-\delta_i}}
\]
give
\[
\ch M_{\frt}(\gamma)
=e^\gamma\prod_{\alpha\in R^+}(1-e^{-\alpha})^{-1}.
\]
Comparing the two character expansions yields
\[
\ch L_{\frt}(\mu)
=e^{\rho_\Psi-\rho_{\frt}}
 \ch L_\Psi(\xi-\rho_\Psi)
 \prod_{\alpha\in R^+\setminus\Psi^+}
       (1-e^{-\alpha})^{-1}.
\]

Choose the height $H(\delta_i)=n-i+1$. For either highest weight module, let \(F_k = U_k(\mathfrak{n}^-)v\) and let \(\mathcal{H}_N\) be the sum of the weight spaces of height at most \(N\) below its highest weight. Every negative root generator has height between \(1\) and \(2n\). PBW Theorem gives
\[
F_{\lfloor N/(2n) \rfloor} \subseteq \mathcal{H}_N \subseteq F_N.
\]
The eventual Hilbert polynomial of \(\dim F_k\) therefore shows that the height truncation has the same logarithmic growth exponent as the PBW filtration. Since all positive roots have positive heights, bounded
above and below by positive constants, the height and
PBW filtrations have the same growth exponent.
Put
\[
d=\GKD L_\Psi(\xi-\rho_\Psi),
\qquad s=|R^+\setminus\Psi^+|.
\]
Let $A(N)$ and $C(N)$ be the cumulative weight multiplicities
up to height $N$ below the respective highest weights of
$L_\Psi(\xi-\rho_\Psi)$ and $L_{\frt}(\mu)$.
Let $B(N)$ count monomials of height at most $N$ in variables
indexed by $R^+\setminus\Psi^+$.
The character identity has nonnegative coefficients and gives
\[
A(\lfloor N/2\rfloor)B(\lfloor N/2\rfloor)
\leq C(N)\leq A(N)B(N).
\]
The Hilbert polynomial and weighted monomial counting imply
\[
\lim_{N\to\infty}\frac{\log A(N)}{\log N}=d,
\qquad
\lim_{N\to\infty}\frac{\log B(N)}{\log N}=s.
\]
Consequently,
\[
\GKD L_{\frt}(\mu)=d+s=n^2-|\Psi^+|+d.
\]

Finally, the ordinary integral RS formulas \cite{BX1,BXX},
applied to the components of $\Psi$ with shifted parameter
$\xi$, give
\[
d=|\Psi^+|-F_B(\xi_{1/2}^-)-F_D(\xi_0^-)
-\sum_{\xi}F_A(\widetilde\xi_z),
\]
where each remaining pair of cosets is counted once.
Substitution proves the formula.
If $\Psi$ is empty, $L_\Psi(\xi-\rho_\Psi)$ is
one-dimensional, and the same argument applies.
\end{proof}

\begin{Cor}\label{cor:osp-terminal-lattices}
For $\mathfrak t=\mathfrak{osp}(1|2n)$,
\begin{enumerate}
    \item If $\mu\in\mathbb Z^n$, we have
    $$\GKD L_{\mathfrak t}(\mu)=n^2-F_B((\mu+\rho_{\mathfrak t})^-).$$

    \item If $\mu\in(\frac12+\mathbb Z)^n$,  we have
\begin{equation}\label{eq:osp-half-integer-GK}
\GKD L_{\mathfrak t}(\mu)=n^2-F_D((\mu+\rho_{\mathfrak t})^-)=n+\GKD L_{\mathfrak{so}(2n)}\bigl(\mu+\tfrac12(1,\ldots,1)\bigr).
\end{equation}
\end{enumerate}
In particular,
\[
\GKD L_{\mathfrak{osp}(1|2)}(\mu)=
\begin{cases}
0,&\mu\in\mathbb Z_{\geq0},\\
1,&\mu\notin\mathbb Z_{\geq0}.
\end{cases}
\]

\end{Cor}

\subsubsection{GK formula of $B(m,n)$}
\begin{Thm}\label{thm:B-GK-all}
Let $\frg=B(m,n)$, $m\geq0$, $n\geq1$, and let $V=\widetilde{L}(\Lambda_{\Ydown})$. Compute $\mu_{\mid}$ by Proposition \ref{prop:odd-shuffle-BD}. Put
\begin{equation}\label{eq:B-shifted-input}
\lambda_{\Ydown}^{\rho_0}=\lambda_{\Ydown}+(m-\tfrac12,\ldots,\tfrac12),\qquad
\mu_{\mid}^{\rho_{\mathfrak t}}=\mu_{\mid}+(n-\tfrac12,\ldots,\tfrac12).
\end{equation}
For every complex highest weight,
\begin{equation}\label{eq:B-GK-all}
\begin{aligned}
\GKD V={}&m^2+n^2
-F_B\bigl((\lambda_{\Ydown}^{\rho_0})_0^-\bigr)
-F_B\bigl((\lambda_{\Ydown}^{\rho_0})_{\frac12}^-\bigr)\\
&-F_D\bigl((\mu_{\mid}^{\rho_{\mathfrak t}})_0^-\bigr)
-F_B\bigl((\mu_{\mid}^{\rho_{\mathfrak t}})_{\frac12}^-\bigr)-\sum_{\lambda_{\Ydown},\mu_{\mid}}
 F_A\bigl(\widetilde{(\lambda_{\Ydown}^{\rho_0})}_z||\widetilde{(\mu_{\mid}^{\rho_{\mathfrak t}})}_z\bigr).
\end{aligned}
\end{equation}
If $\Lambda_{\Ydown}$ is integral for $\frg_0$, this simplifies to
\begin{equation}\label{eq:B-GK-integral}
\GKD V=m^2+n^2
-F_B\bigl((\lambda_{\Ydown}^{\rho_0})^-\bigr)
-F_B\bigl((\mu_{\mid}^{\rho_{\mathfrak t}})^-\bigr).
\end{equation}
The second shift is $\rho_{\mathfrak t}$, rather than the ordinary symplectic shift $\rho_{C}$. The formula has no typicality hypothesis. For $m=0$, the formula is exactly Lemma \ref{lem:osp-terminal-GK} for $B(0,n)=\mathfrak{osp}(1|2n)$.
\end{Thm}
\begin{proof}
Apply $B_m$ formula \cite[Theorem 5.7]{BXX} to the first summand of \eqref{eq:B-two-Borel} and Lemma \ref{lem:osp-terminal-GK} to the second. This proves \eqref{eq:B-GK-all}. The first assertion of Corollary \ref{cor:osp-terminal-lattices} then gives \eqref{eq:B-GK-integral}.
\end{proof}

\begin{Exam}\label{ex:B-GK-integral}
For $\frg=\mathfrak{osp}(5|2)$ and $\Lambda_{\Ydown}=(1,0\mid0)$, we have $\Lambda_{\mid}=(2,1\mid-2)$. The ordinary orthogonal factor is finite-dimensional. The final parameter is $\xi=-\frac32$, whose doubled RS shape is ${\bf p}(\xi^-)=[2]$. Thus $\GKD V=1$.
\end{Exam}

\begin{Exam}\label{ex:B-GK-half-integer}
For $\frg=\mathfrak{osp}(3|4)$ and $\Lambda_{\Ydown}=(\frac12\mid\frac12,-\frac12)$, we have $\Lambda_{\mid}=(\frac32\mid-\frac12,-\frac12)$. The orthogonal module of highest weight $\frac12$ is finite-dimensional. The final parameter is $\xi=(1,0)$, whose doubled RS shape is ${\bf p}(\xi^-)=[2,1,1]$. Hence $\GKD V=4-2=2$.
\end{Exam}

\subsection{General comparisons for type II highest weight modules}
The preceding subsections give explicit GK dimension formulas for $D(m,n)$ and $B(m,n)$. We first retain a general type II bound, then give a component comparison used in the exact exceptional formulas of the following three subsections.
\begin{Prop}
     Let $\frg$ be of type II. Fix a $\mathbb Z_2$-consistent five-grading and a compatible Borel $\frb\subseteq\frp=\frg(0)\oplus\frg(1)\oplus\frg(2)$. Put $d=\GKD L_{\frg_0}(\Lambda)$ for the induced positive even roots. Then $d\leq\GKD\widetilde L(\Lambda)\leq d+\dim\frg(2)$.
\end{Prop}
\begin{proof}
Put $\frp=\frg(0)\oplus\frg(1)\oplus\frg(2)$ and use the generalized Verma module
$$\mathcal N(\Lambda)=U(\frg)\otimes_{U(\frp)}L_0(\Lambda).$$
It has a surjection onto $\widetilde L(\Lambda)$ for every highest weight. Hence
\begin{align*}
     \text{GK}\dim\widetilde L(\Lambda)&\leq\text{GK}\dim\mathcal N(\Lambda)\\&\overset{\text{Prop.} \ref{GK-dim-induced-mod}}{=}\text{GK}\dim L_0(\Lambda)+\dim\frg(-2)\leq\text{GK}\dim L_{\frg_0}(\Lambda)+\dim\frg(2),
\end{align*}
where the last inequality follows from Lemma \ref{GK-dim-submodule} and $\dim\frg(-2)=\dim\frg(2)$. The same lemma gives the lower bound $d\leq\text{GK}\dim\widetilde L(\Lambda)$.

\end{proof}

For the exceptional cases, the following consequence of Lemma \ref{lem:cyclic-support} and  Proposition \ref{lem:even-submodule-support} allows the even factors to be evaluated at independently chosen highest weight vectors.

\begin{Prop}\label{prop:exceptional-component-GK}
Let $V$ be a simple highest weight $\frg$-module and suppose that $\frg_0=\frk_1\oplus\cdots\oplus\frk_r$ is semisimple, where the $\frk_i$ are its simple ideals. For arbitrary nonzero homogeneous vectors $v_1,\ldots,v_r\in V$,
\[
 \AV_{\frg_0}V=\prod_{i=1}^r\AV_{\frk_i}U(\frk_i)v_i,\qquad\GKD V=\sum_{i=1}^r\GKD U(\frk_i)v_i.
\]
\end{Prop}
\begin{proof}
Choose a nonzero homogeneous highest weight vector $v$ of $V$. The module $U(\mathfrak g_{0})v$ is a quotient of a Verma module, so it has finite length and contains a nonzero simple submodule $S$. All nonzero vectors of $S$ have the same parity as $v$. Hence Proposition \ref{lem:even-submodule-support} gives $\operatorname{AV}_{\mathfrak g_{\bar0}}V=\operatorname{AV}_{\mathfrak g_{\bar0}}S$.

Write $S\simeq S_1\otimes_{\mathbb{C}}\cdots\otimes_{\mathbb{C}} S_r$, where $S_i$ is a simple highest weight $\mathfrak k_i$-module. Choose highest weight vectors $w_i\in S_i$ and put $w=w_1\otimes\cdots\otimes w_r$.
Then $U(\mathfrak k_i)w\simeq S_i$, and Lemma \ref{lem:cyclic-support} gives
\[
\operatorname{AV}_{\mathfrak k_i}U(\mathfrak k_i)v_i
=
\operatorname{AV}_{\mathfrak k_i}U(\mathfrak k_i)w
=
\operatorname{AV}_{\mathfrak k_i}S_i.
\]
Tensor products of good filtrations therefore give
\[
\operatorname{AV}_{\mathfrak g_{0}}V=\operatorname{AV}_{\mathfrak g_{0}}S=\prod_{i=1}^r\operatorname{AV}_{\mathfrak k_i}S_i=\prod_{i=1}^r\operatorname{AV}_{\mathfrak k_i}U(\mathfrak k_i)v_i.
\]
\end{proof}

For a complex number $c$, put
\begin{equation}\label{eq:exceptional-kappa}
 \varkappa(c)=
 \begin{cases}
 0,&c\in\mathbb Z_{\geq0},\\
 1,&c\notin\mathbb Z_{\geq0}.
 \end{cases}
\end{equation}

\begin{Lem}\label{lem:exceptional-rank-one-GK}
For every $c\in\mathbb C$,
\[
 \GKD L_{\frsl(2)}(c)=\GKD L_{\mathfrak{osp}(1|2)}(c\delta) =\varkappa(c).
\]
Here $\delta$ is the positive odd root. If $v$ is a nonzero homogeneous highest weight vector of
$L_{\mathfrak{osp}(1|2)}(c\delta)$, then $\GKD U(\mathfrak{sl}_2)v=\varkappa(c)$.
\end{Lem}
\begin{proof}
A rank-one simple module has GK dimension one when $c\notin\mathbb Z_{\geq0}$; otherwise it  is finite-dimensional. The super assertion is the rank-one case of Lemma \ref{lem:osp-terminal-GK}. Applying Proposition \ref{lem:even-submodule-support} to its nonzero even cyclic submodule $U(\frsl(2))v$ proves the last assertion.
\end{proof}

\subsection{GK dimensions of $\widetilde{L}(\Lambda)$ for $G(3)$}
\label{sec:G3-GK-algorithm}
Use the roots, fundamental weights and simple systems $\Pi_1,\ldots,\Pi_4$ of Subsection \ref{notation-G3}. In particular, $\Pi_1$ contains the $G_2$ simple roots, while $\Pi_4$ contains the odd root $\delta$ generating $\mathfrak{osp}(1|2)$. Write the highest weight as
\[
 \Lambda=(\lambda\mid\mu)=\lambda\delta+\mu,
 \qquad
 \mu=b_1\omega_1+b_2\omega_2,\qquad \lambda=b'_1.
\]

For a condition $P$, let $\mathbf1_{\{P\}}$ be $1$ when $P$ holds and $0$ otherwise. 

\begin{Thm}\label{thm:GK-G3-exceptional}
Define $d_1,d_2,d_3\in\{0,1\}$ by
\begin{align}
 d_1&=\mathbf1_{\{2\lambda+3b_1+2b_2\neq0\}},\notag\\
 d_2&=\mathbf1_{\{2\lambda+3b_1+b_2-d_1\neq0\}},
       \label{eq:G3-GK-tests}\\
 d_3&=\mathbf1_{\{2\lambda+b_2-d_1-3d_2\neq0\}},\notag
\end{align}
and put $\lambda_4=\lambda-d_1-d_2-d_3$. The highest weight of the same simple module for $\Pi_4$ is $\Lambda_4=(\lambda_4\mid\mu_4)$, where
\[
 \mu_4=(b_1+d_2-d_3)\omega_1
       +(b_2+d_1-d_2+2d_3)\omega_2.
\]
For $V=\widetilde L(\Lambda)$ relative to $\Pi_1$, with arbitrary
complex highest weight,
\begin{equation}\label{eq:G3-GK-exact}
\begin{aligned}
 \GKD V
 &=\GKD L_{G_2}(\mu)+\varkappa(\lambda_4)\\
 &=|\Phi_{G_2}^+|-{\bf a}_{[\mu^{\rho_0}]}(w_{\mu^{\rho_0}})
       +\varkappa(\lambda_4),
\end{aligned}
\end{equation}
where $\mu^{\rho_0}=\mu+\omega_1+\omega_2$ and $w_{\mu^{\rho_0}}$ is the minimal length element of Lemma \ref{GKdim-uniform1}.
\end{Thm}
\begin{proof}
Starting from $\Pi_1$, we apply odd reflections at
\[
\beta_1=\delta+\epsilon_3,\qquad
\beta_2=\delta-\epsilon_2,\qquad
\beta_3=\delta-\epsilon_1.
\]
These roots are isotropic simple roots in $\Pi_1,\Pi_2,\Pi_3$, respectively, and the resulting systems are $\Pi_2,\Pi_3,\Pi_4$. The positive even roots remain unchanged. By the odd-reflection rule, the highest weight is unchanged if its pairing with the reflected root is zero, and otherwise decreases by that root; see \cite[Corollary 8.6.2]{Mu12}. Write $d_i=0$ or $1$ according as the $i$-th reflection leaves the highest weight unchanged or subtracts $\beta_i$. With the bilinear form fixed in Subsection \ref{notation-G3}, we have
\[
(\mu,\epsilon_1)=b_2,\qquad
(\mu,\epsilon_2)=3b_1+b_2,\qquad
(\mu,\epsilon_3)=-3b_1-2b_2.
\]
The pairings that determine the three reflections are therefore
\begin{align*}
(\Lambda,\beta_1)&=-(2\lambda+3b_1+2b_2),\\
(\Lambda-d_1\beta_1,\beta_2)&=-(2\lambda+3b_1+b_2-d_1),\\
(\Lambda-d_1\beta_1-d_2\beta_2,\beta_3)
&=-(2\lambda+b_2-d_1-3d_2).
\end{align*}
Here each pairing is taken with the highest weight resulting from the preceding reflections. The final highest weight is thus
$\Lambda_4=\Lambda-d_1\beta_1-d_2\beta_2-d_3\beta_3$.

Let $v_1$ and $v_4$ be homogeneous highest weight vectors for the initial and final systems, respectively. The subalgebras generated by the indicated simple roots give
\[
U(G_2)v_1\simeq L_{G_2}(\mu),\qquad
U(\mathfrak{osp}(1|2))v_4
\simeq L_{\mathfrak{osp}(1|2)}(\lambda_4\delta),
\]
by the argument in Proposition \ref{D(m,n)}; see also \cite[Lemma 1.2.2(i)]{GG21}. Since $\frg_0=G_2\oplus\frsl(2)$, Proposition \ref{prop:exceptional-component-GK} and Lemma \ref{lem:exceptional-rank-one-GK} give the first equality in \eqref{eq:G3-GK-exact}. The second equality follows from the  GK dimension formula in Lemma \ref{GKdim-uniform1}.

\end{proof}

\begin{Cor}\label{cor:G3-GK-spectrum}
The GK dimensions of simple highest weight $G(3)$-modules are exactly $\{0,1,3,4,5,6,7\}$.
\end{Cor}
\begin{proof}
By Proposition \ref{prop:G3-even-GK-values}, the possible values of $\GKD L_{G_2}(\mu)$ are $\{0,3,4,5,6\}$. Fix $\mu=b_1\omega_1+b_2\omega_2$ and let $d=\GKD L_{G_2}(\mu)$. Choose an integer $N\ge3$ sufficiently large that
\[
2\lambda+3b_1+2b_2,\qquad
2\lambda+3b_1+b_2-1,\qquad
2\lambda+b_2-4
\]
are nonzero for both $\lambda=N$ and $\lambda=N+\tfrac12$. Then $d_1=d_2=d_3=1$ in both cases, so
\[
\lambda_4=N-3\in\mathbb Z_{\ge0}
\quad\text{or}\quad
\lambda_4=N-\tfrac52\notin\mathbb Z,
\]
respectively. Thus $\varkappa(\lambda_4)$ takes values in $\{0,1\}$, equation \eqref{eq:G3-GK-exact} gives
\[
\GKD V\in\{0,1,3,4,5,6,7\}.
\]
\end{proof}

\begin{Exam}\label{ex:G3-GK-rank-one}
For $\mu=0$, the weights $\Lambda=\delta,\delta/2,2\delta$ give, respectively,
\[
 (d_1,d_2,d_3)=(1,1,1),\quad(1,0,0),\quad(1,1,0).
\]
The values of $\lambda_4$ are $-2,-1/2,0$, respectively, and their GK dimensions are $1,1,0$.
\end{Exam}

\subsection{GK dimensions of $\widetilde{L}(\Lambda)$ for $F(4)$}
\label{sec:F4-GK-algorithm}
Choose orthogonal coordinates satisfying
\[
 (\delta,\delta)=-3,\qquad
 (\epsilon_i,\epsilon_j)=\delta_{ij},\qquad
 (\delta,\epsilon_i)=0,\qquad 1\leq i,j\leq3.
\]
The positive even and odd roots are
\begin{align*}
 \Delta_0^+&=\{\delta\}\cup\{\epsilon_i:1\leq i\leq3\}
       \cup\{\epsilon_i\pm\epsilon_j:i<j\},\\
 \Delta_1^+&=\{\tfrac12(\delta\pm\epsilon_1\pm\epsilon_2
                              \pm\epsilon_3)\}.
\end{align*}
All odd positive roots are isotropic, and $\frg_0=\frso(7)\oplus\frsl(2)$. The initial simple system is
\[
 \Pi^{(0)}=
 \{\tfrac12(\delta-\epsilon_1-\epsilon_2-\epsilon_3),
       \epsilon_3,\epsilon_2-\epsilon_3,\epsilon_1-\epsilon_2\}.
\]
These root data are those of \cite[Section 4.1]{ML}. The Dynkin diagram associated to $\Pi^{(0)}$ is depicted as follows:
  \begin{center}
$\dynkin[labels={\alpha_1,\alpha_2,\alpha_3,\alpha_4},reverse arrows]F{tooo}$
  \end{center}

Write
\[
 \Lambda=(\lambda\mid\mu)=\sum_{i=1}^3\lambda_i\epsilon_i+\mu\delta,\qquad \lambda=(\lambda_1,\lambda_2,\lambda_3).
\]
 Weyl vectors are
\[
 \rho_0=(\frac52,\frac32,\frac12\mid\frac12),\qquad\rho=(\frac52,\frac32,\frac12\mid-\frac32).
\]
In particular $(\Lambda,\delta^\vee)=2\mu$.

\begin{Thm}\label{thm:GK-F4-exceptional}
Start with $\lambda^{(0)}=\lambda$ and $\mu^{(0)}=\mu$. For $r=1,2,3,4$, take
\[
 \sigma^{(1)}=(-1,-1,-1),\quad
 \sigma^{(2)}=(-1,-1,1),\quad
 \sigma^{(3)}=(-1,1,-1),\quad
 \sigma^{(4)}=(-1,1,1)
\]
and define
\begin{equation}\label{eq:F4-GK-updates}
\begin{aligned}
 d_r&=\mathbf1_{\{3\mu^{(r-1)}
                   -\sum_{i=1}^3\sigma_i^{(r)}
                                      \lambda_i^{(r-1)}\neq0\}},\\
 \mu^{(r)}&=\mu^{(r-1)}-\tfrac12d_r,\\
 \lambda_i^{(r)}&=\lambda_i^{(r-1)}
                         -\tfrac12d_r\sigma_i^{(r)}.
\end{aligned}
\end{equation}
For the simple highest weight module $V=\widetilde L(\Lambda)$ relative to $\Pi^{(0)}$,
\begin{equation}\label{eq:F4-GK-exact}
 \GKD V=\GKD L_{\frso(7)}(\lambda)
                         +\varkappa(2\mu^{(4)}).
\end{equation}
Put $\lambda^{\rho_0}=\lambda+(\tfrac52,\tfrac32,\tfrac12)$.
For arbitrary complex highest weights, the explicit   formula is
\begin{equation}\label{eq:F4-GK-insertion}
\begin{aligned}
 \GKD V={}&|\Phi_{B_3}^+|-F_B((\lambda^{\rho_0})_0^-)-F_B((\lambda^{\rho_0})_{\frac12}^-)-\sum_{\lambda}F_A(\widetilde{(\lambda^{\rho_0})}_z) +\varkappa(2\mu^{(4)}).
\end{aligned}
\end{equation}
\end{Thm}
\begin{proof}
Put $\beta_r=\tfrac12(\delta+\sum_i\sigma_i^{(r)}\epsilon_i)$ and $\gamma=\tfrac12(\delta+\epsilon_1-\epsilon_2-\epsilon_3)$. Successive odd reflections at $\beta_1,\ldots,\beta_4$ give
\begin{align*}
 \Pi^{(1)}&=\{-\beta_1,\beta_2,\epsilon_2-\epsilon_3,
                                            \epsilon_1-\epsilon_2\},\\
 \Pi^{(2)}&=\{\epsilon_3,-\beta_2,\beta_3,
                                            \epsilon_1-\epsilon_2\},\\
 \Pi^{(3)}&=\{\beta_4,\epsilon_2-\epsilon_3,-\beta_3,\gamma\},\\
 \Pi^{(4)}&=\{-\beta_4,\epsilon_2-\epsilon_3,\epsilon_3,\delta\}.
\end{align*}
Thus each reflected root is isotropic and simple at its step. For example, adding $\beta_3$ to the three other roots of $\Pi^{(2)}$ produces $\beta_4,\epsilon_2-\epsilon_3,\gamma$. At the last step,
$\gamma+\beta_4=\delta$ and $-\beta_3+\beta_4=\epsilon_3$, while $(\epsilon_2-\epsilon_3,\beta_4)=0$.
The positive even root system remains fixed.

At the $r$-th step, the current pairing is
\[
 \Big(\sum_i\lambda_i^{(r-1)}\epsilon_i +\mu^{(r-1)}\delta,\beta_r\Big) =-\frac12\Big(3\mu^{(r-1)}
-\sum_i\sigma_i^{(r)}\lambda_i^{(r-1)}\Big).
\]
Subtracting $d_r\beta_r$ is exactly \eqref{eq:F4-GK-updates}. The initial system $\Pi^{(0)}$ contains the $B_3$ simple roots, and the final system $\Pi^{(4)}$ contains the even root $\delta$. Their highest weight vectors therefore generate
the simple modules $L_{\frso(7)}(\lambda)$ and $L_{\frsl(2)}(2\mu^{(4)})$ for the respective factors. Apply Proposition \ref{prop:exceptional-component-GK} and Lemma \ref{lem:exceptional-rank-one-GK} to prove
\eqref{eq:F4-GK-exact}. The formula in \cite[Theorem 5.7(1)]{BXX} gives \eqref{eq:F4-GK-insertion}.
\end{proof}

\begin{Cor}\label{cor:F4-GK-spectrum}
The GK dimensions of simple highest weight $F(4)$-modules are
exactly $\{0,1,4,5,6,7,8,9,10\}$.
\end{Cor}
\begin{proof}
The integer coordinates of $\lambda^{\rho_0}$ form one subsequence, and the coordinates in $\tfrac12+\mathbb Z$ form another. For either nonempty subsequence $x$ of length $k$, the doubled sequence $x^-$ has length $2k$. Thus its RS shape is a partition of $2k$.

For the remaining coordinates, consider the sets $z+\mathbb Z$ and $-z+\mathbb Z$ together, with each pair considered only once. Form $\widetilde{x}_z$ by listing the coordinates in $z+\mathbb Z$ in their original order, followed by the negatives of the coordinates in $-z+\mathbb Z$ in reverse order. If these two sets contain $k$
coordinates in total, then $\widetilde{x}_z$ has length $k$.

In both cases, $k$ counts coordinates of $\lambda^{\rho_0}$, so $1\le k\le3$. Enumerating the partitions of $2k$ and $k$ in the respective defining sums gives the following sets containing all possible results:
\[
\begin{array}{c|c|c}
k&F_B(x^-)&F_A(\widetilde{x}_z)\\ \hline
1&\{0,1\}&\{0\}\\
2&\{0,1,4\}&\{0,1\}\\
3&\{0,1,2,3,4,9\}&\{0,1,3\}
\end{array}
\]
The positive lengths $k$ sum to three, so the only possibilities are $3$, $2+1$, and $1+1+1$. In these cases, the sums of the corresponding $F_A$ and $F_B$ terms belong to
\[
\{0,1,2,3,4,9\},\qquad
\{0,1,2,4,5\},\qquad
\{0,1,2\},
\]
respectively. Thus the values of {\bf a} belongs to $\{0,1,2,3,4,9\}\cup\{0,1,2,4,5\}\cup\{0,1,2\}=\{0,1,2,3,4,5,9\}$. By Lemma \ref{GKdim-uniform1}, we have $\GKD_{\frso(7)}(\lambda)\in\{0,4,5,6,7,8,9\}$. Thus equation \ref{eq:F4-GK-exact} gives 
$$\GKD V\in\{0,1,4,5,6,7,8,9,10\}.$$

\end{proof}

\begin{Exam}\label{ex:F4-GK-rank-one}
For $\lambda=(0,0,0)$ and $\mu=1$, we have $(d_1,d_2,d_3,d_4)=(1,1,0,0)$ and $\mu^{(4)}=0$. Thus $\GKD\widetilde L(\Lambda)=0$. For $\lambda=(1,0,0)$ and $\mu=0$, we have $(d_1,d_2,d_3,d_4)=(1,0,0,1)$ and $2\mu^{(4)}=-2$, so $\GKD\widetilde L(\Lambda)=1$.
\end{Exam}

\begin{Rem}\label{rem:exceptional-GK-stabilization}
For $G(3)$, let $c=\lambda$ and $d=\GKD L_{G_2}(\mu)$, and put $r=3$. For $F(4)$, let $c=2\mu$ and
$d=\GKD L_{\frso(7)}(\lambda)$, and put $r=4$. Define $N=\sum_{i=1}^rd_i\in\{0,\ldots,r\}$, and hence
\[
 \GKD V=
 \begin{cases}
 d+1,&c\notin\mathbb Z\text{ or }c\in\mathbb Z_{<0},\\
 d,&c\in\mathbb Z_{\geq r}.
 \end{cases}
\]
Only the cases $c\in\{0,\ldots,r-1\}$ need to be considered separately.
\end{Rem}

\subsection{GK dimensions of $\widetilde{L}(\Lambda)$ for $D(2,1;\alpha)$}\label{sec:D21-GK-algorithm}
Let $\alpha\in\mathbb C\setminus\{0,-1\}$ and take orthogonal coordinates satisfying
\[
 (\epsilon_1,\epsilon_1)=-(1+\alpha),\qquad
 (\epsilon_2,\epsilon_2)=1,\qquad
 (\epsilon_3,\epsilon_3)=\alpha.
\]
The roots are
\[
 \Delta_0=\{\pm2\epsilon_i:1\leq i\leq3\},\qquad\Delta_1=\{\pm\epsilon_1\pm\epsilon_2\pm\epsilon_3\},
\]
and $\frg_0=\frsl(2)\oplus\frsl(2)\oplus\frsl(2)$. All odd roots are isotropic. Fix the initial simple system
\[
 \Pi^{(0)}=\{\epsilon_1-\epsilon_2-\epsilon_3,  2\epsilon_2,2\epsilon_3\}.
\]
These are the usual exceptional root data \cite{Mu12,FSS2000}. The Dynkin diagram associated to $\Pi^{(0)}$ is
depicted as follows:
\[
\begin{dynkinDiagram}A{t}
    \dynkinPlaceRootRelativeTo{2}{1}{northeast}{right}{left}
    \dynkinPlaceRootRelativeTo{3}{1}{southeast}{right}{left}

    \dynkinRootMark{o}{2}
    \dynkinRootMark{o}{3}

    \dynkinDefiniteSingleEdge{1}{2}
    \dynkinDefiniteSingleEdge{1}{3}

    \node[left=2mm] at (root 1)
        {$\epsilon_1-\epsilon_2-\epsilon_3$};
    \node[right=2mm] at (root 2) {$2\epsilon_2$};
    \node[right=2mm] at (root 3) {$2\epsilon_3$};
\end{dynkinDiagram}
\]
Write the highest weight as
\[
 \Lambda=(\lambda_1,\lambda_2,\lambda_3) =\sum_{i=1}^3\lambda_i\epsilon_i.
\]
Then $(\Lambda,(2\epsilon_i)^\vee)=\lambda_i$, $\rho_0=(1,1,1)$ and $\rho=(-1,1,1)$.

\begin{Thm}\label{thm:GK-D21-exceptional}
Write $\Lambda^{(r)}=(\lambda_1^{(r)},\lambda_2^{(r)},\lambda_3^{(r)})$ and put $\Lambda^{(0)}=\Lambda$. Define
\begin{equation}\label{eq:D21-GK-updates}
\begin{aligned}
 d_1&=\mathbf1_{\{-(1+\alpha)\lambda_1-\lambda_2-\alpha\lambda_3\neq0\}},\\
\Lambda^{(1)}&=(\lambda_1-d_1,\lambda_2+d_1,\lambda_3+d_1),\\
 d_2&=\mathbf1_{\{-(1+\alpha)\lambda_1^{(1)}+\lambda_2^{(1)} -\alpha\lambda_3^{(1)}\neq0\}},\\
 \Lambda^{(2)} &=(\lambda_1^{(1)}-d_2,\lambda_2^{(1)}-d_2, \lambda_3^{(1)}+d_2).
\end{aligned}
\end{equation}
For $V=\widetilde L(\Lambda)$ with respect to $\Pi^{(0)}$,
\begin{equation}\label{eq:D21-GK-exact}
 \GKD V=\varkappa(\lambda_1-d_1-d_2) +\varkappa(\lambda_2)+\varkappa(\lambda_3).
\end{equation}
This holds for every complex highest weight and every allowed parameter $\alpha$. For each fixed $\alpha$, the GK dimensions of simple highest weight $D(2,1;\alpha)$-modules are exactly $\{0,1,2,3\}$.
\end{Thm}
\begin{proof}
In the initial simple system $\Pi^{(0)}$, $2\epsilon_2$ and $2\epsilon_3$ are simple roots. The initial
highest weight vector therefore generates simple $\mathfrak{sl}_2$-modules with highest weights $\lambda_2$ and $\lambda_3$, respectively.

The coefficient of $\epsilon_1$ in $\Lambda^{(2)}$ is $\lambda_1-d_1-d_2$. The same argument shows that the final
highest weight vector generates a simple $\mathfrak{sl}_2$-module with this highest weight.

By Lemma \ref{lem:exceptional-rank-one-GK}, these three modules have GK dimensions
\[
\varkappa(\lambda_2),\qquad
\varkappa(\lambda_3),\qquad
\varkappa(\lambda_1-d_1-d_2),
\]
respectively. Proposition \ref{prop:exceptional-component-GK} then gives \eqref{eq:D21-GK-exact}.

Let $d_1,d_2$ be defined by \eqref{eq:D21-GK-updates}. First apply the odd reflection at $\beta_1=\epsilon_1-\epsilon_2-\epsilon_3$. The highest weight becomes $\Lambda^{(1)}=\Lambda-d_1\beta_1$, and the corresponding simple system is
\[
\Pi^{(1)}
=\{-\beta_1,\epsilon_1+\epsilon_2-\epsilon_3,\epsilon_1-\epsilon_2+\epsilon_3\}.
\]
Next apply the odd reflection at $\beta_2=\epsilon_1+\epsilon_2-\epsilon_3$, which is simple in $\Pi^{(1)}$.
The highest weight becomes $\Lambda^{(2)}=\Lambda^{(1)}-d_2\beta_2$. Thus the final simple system is
\[
\Pi^{(2)}=\{-\beta_2,2\epsilon_2,2\epsilon_1\}.
\]
Both reflections leave the positive even roots unchanged.

It remains to show that every value in $\{0,1,2,3\}$ occurs. Choose
\[
\lambda_1\in\left\{3,\frac72\right\},\qquad\lambda_2,\lambda_3\in\left\{0,\frac12\right\}.
\]
Since $0\leq d_1+d_2\leq2$, the number $\lambda_1-d_1-d_2$ is a nonnegative integer when $\lambda_1=3$ and a noninteger when $\lambda_1=\frac72$. Consequently,
\[
\varkappa(\lambda_1-d_1-d_2)=
\begin{cases}
0,&\lambda_1=3,\\
1,&\lambda_1=\frac72.
\end{cases}
\]
For $i=2,3$, $\lambda_i=0$ or $\frac12$ gives $\varkappa(\lambda_i)=0$ or $1$, respectively. Hence the three terms in the GK dimension formula can independently take the values $0$ and $1$. Their sum therefore takes every value in $\{0,1,2,3\}$ for every $\alpha\neq0,-1$.
\end{proof}

\begin{Cor}\label{cor:D21-GK-parameter}
Fix the highest weight $\Lambda=(\lambda_1,\lambda_2,\lambda_3)$ with $\alpha\in\mathbb C\setminus\{0,-1\}$, and let $V=\widetilde L(\Lambda)$ relative to $\Pi^{(0)}$. Its GK dimension is nonconstant in $\alpha$ if and only if
\[
 \lambda_1=1,\qquad \lambda_2\neq-1,\qquad \lambda_3\neq-1.
\]
In this case put
\[
 E_\Lambda=\left\{\frac{\lambda_2+1}{\lambda_3+1}, -\frac{\lambda_2+1}{\lambda_3+1}\right\}\bigg\backslash\{0,-1\}.
\]
Then
\begin{equation}\label{eq:D21-GK-parameter}
 \GKD V=\varkappa(\lambda_2)+\varkappa(\lambda_3)+\mathbf1_{\{\alpha\notin E_\Lambda\}}.
\end{equation}
If $(\lambda_2+1)/(\lambda_3+1)$ equals $1$ or $-1$, then $E_\Lambda=\{1\}$. Otherwise, $E_\Lambda$ contains
two distinct values.
\end{Cor}
\begin{proof}
By \eqref{eq:D21-GK-exact}, only $\varkappa(\lambda_1-d_1-d_2)$ can change as $\alpha$ varies, since $\lambda_2$ and $\lambda_3$ are fixed. Write
\[
\begin{aligned}
A(\alpha)&=-(\lambda_1+\lambda_2) -\alpha(\lambda_1+\lambda_3),\\
B(\alpha)&=-\lambda_1+\lambda_2+2d_1-\alpha(\lambda_1+\lambda_3).
\end{aligned}
\]
Here $d_1=0$ precisely when $A(\alpha)=0$, and $d_2=0$ precisely when $B(\alpha)=0$; otherwise, the corresponding
$d_i$ equals $1$.

Since $0\leq d_1+d_2\leq2$, if $\lambda_1\notin\mathbb Z_{\geq0}$, then $\lambda_1-d_1-d_2$ cannot be a nonnegative integer. Thus $\varkappa(\lambda_1-d_1-d_2)=1$. If $\lambda_1\in\mathbb Z_{\geq2}$, then
\[
\lambda_1-d_1-d_2\geq\lambda_1-2\geq0,
\]
so $\varkappa(\lambda_1-d_1-d_2)=0$.

Suppose that $\lambda_1=0$. Then $\varkappa(-d_1-d_2)=0$ requires $d_1=d_2=0$. The equations $A(\alpha)=B(\alpha)=0$ become
\[
-\lambda_2-\alpha\lambda_3=0,\qquad\lambda_2-\alpha\lambda_3=0.
\]
Since $\alpha\neq0$, these equations give $\lambda_2=\lambda_3=0$. Conversely, at $\Lambda=(0,0,0)$, we have $d_1=d_2=0$ for every allowed $\alpha$. Hence
\[\varkappa(-d_1 - d_2) = 
\begin{cases} 
0, & (\lambda_2, \lambda_3) = (0, 0), \\ 
1, & (\lambda_2, \lambda_3) \neq (0, 0).
\end{cases}\]
Therefore, a change in GK dimension is possible only when $\lambda_1=1$.

Now let $\lambda_1=1$. Then
\begin{align*}
    A(\alpha)&=-(\lambda_2+1)-\alpha(\lambda_3+1),\\
B(\alpha)&=\lambda_2-1+2d_1-\alpha(\lambda_3+1).
\end{align*}
If $\lambda_3=-1$, then $A(\alpha)$ and $d_1$ are independent of $\alpha$. The same is then true of $B(\alpha)$ and $d_2$, so the GK dimension is constant. If $\lambda_3\neq-1$ but $\lambda_2=-1$, then
\[
A(\alpha)=-\alpha(\lambda_3+1)\neq0,
\]
so $d_1=1$. Similar process gives $B(\alpha)=A(\alpha)\neq0$, so $d_2=1$ as well. The GK dimension is again constant.

It remains to consider $\lambda_1=1$ and $\lambda_2,\lambda_3\neq-1$. Put
\[
r=\frac{\lambda_2+1}{\lambda_3+1}\neq0.
\]
If $\alpha\neq\pm r$, then $A(\alpha)=-(\lambda_3+1)(\alpha+r)\neq0$, so $d_1=1$. Similarly, $B(\alpha)=(\lambda_3+1)(r-\alpha)\neq0$, and $d_2=1$. Thus
\[
\varkappa(1-d_1-d_2)=\varkappa(-1)=1.
\]
If $\alpha=r$, we have $A(r)=-2(\lambda_2+1)\neq0$ and $B(r)=0$. Hence $d_1=1$, $d_2=0$, and the first term is $\varkappa(1-d_1-d_2)=0$. If $\alpha=-r$, we have $d_1=0$. Regardless of whether $d_2=0$ or $1$,
\[
1-d_1-d_2\in\{0,1\},
\]
so $\varkappa(1-d_1-d_2)=0$.

Therefore, with $E_\Lambda=\{r,-r\}\setminus\{0,-1\}$, we obtain
\[
\varkappa(1-d_1-d_2)=
\begin{cases}
0,&\alpha\in E_\Lambda,\\
1,&\alpha\notin E_\Lambda.
\end{cases}
\]
Adding $\varkappa(\lambda_2)+\varkappa(\lambda_3)$ proves \eqref{eq:D21-GK-parameter}.

Recall that $r\neq0$. If $r=\pm1$, deleting the forbidden parameter $-1$ gives $E_\Lambda=\{1\}$.
Otherwise, both $r$ and $-r$ are allowed parameters, so $E_\Lambda=\{r,-r\}$. Thus $E_\Lambda$ contains one or two elements and is always nonempty.

Since there are infinitely many allowed parameters, there are also parameters outside $E_\Lambda$. By \eqref{eq:D21-GK-parameter}, the GK dimension is one less for $\alpha\in E_\Lambda$ than for $\alpha\notin E_\Lambda$. It therefore takes two different values as $\alpha$ varies, proving the stated criterion.
\end{proof}

\begin{Exam}\label{ex:D21-GK-jumps}
For the fixed highest weight $\Lambda=\epsilon_1$, we have
\[
 \GKD\widetilde L(\Lambda)=
 \begin{cases}0,&\alpha=1,\\1,&\alpha\neq1.\end{cases}
\]
For $\Lambda=\epsilon_1+\tfrac12\epsilon_2+\tfrac12\epsilon_3$, we have
\[
 \GKD\widetilde L(\Lambda)=
 \begin{cases}2,&\alpha=1,\\3,&\alpha\neq1.\end{cases}
\]
\end{Exam}

\section{Finite-dimensional representations}\label{sec:BD-finite-dimensional}
\subsection{Finite-dimensional representations of $D(m,n)$ and $B(m,n)$}
In this subsection, let $\frg=D(m,n)$ with $m\geq2$, $n\geq1$, or $\frg=B(m,n)$ with $m,n\geq1$. For brevity, write
\[
\Lambda=\Lambda_{\Ydown}=(\lambda\mid\mu),\qquad \lambda=\lambda_{\Ydown},\quad \mu=\mu_{\Ydown},
\]
and put $V=\widetilde{L}(\Lambda)$. Retain $\mu_{\mid}$ obtained by Proposition \ref{prop:odd-shuffle-BD}. The weight $\Lambda$ is dominant integral for $\frg_0$ precisely when
\[
\mu\in\Z^n,\qquad \mu_1\geq\cdots\geq\mu_n\geq0,
\]
and $\lambda\in\Z^m\cup(\frac12+\Z)^m$ satisfies
\[
\begin{cases}
\lambda_1\geq\cdots\geq\lambda_m\geq0, & \text{if }\frg=B(m,n),\\[2pt]
\lambda_1\geq\cdots\geq\lambda_{m-1}\geq|\lambda_m|, & \text{if }\frg=D(m,n).
\end{cases}
\]

For an $(n\mid m)$-hook partition $\tau$, namely $\tau_{n+1}\leq m$, set $q=(\tau_{n+1},\tau_{n+2},\ldots)'$ and define
\[
\tau^\natural=(q_1,\ldots,q_m\mid\tau_1,\ldots,\tau_n),\qquad \tau^\natural_-=(q_1,\ldots,q_{m-1},-q_m\mid\tau_1,\ldots,\tau_n).
\]
Here a prime denotes transpose of partitions. The second weight is used only in type $D$.

\begin{Thm}[$B$ and $D$]\label{thm:BD-zero-GK}
The following conditions are equivalent:
\begin{enumerate}[label=\textup{(\roman*)}]
\item $\GKD V=0$.
\item $V$ is finite-dimensional.
\item $\Lambda$ is dominant integral for $\frg_0$ and $\mu_{\mid}\in\Z^n$ satisfies $\mu_{\mid,1}\geq\cdots\geq\mu_{\mid,n}\geq0$.
\item $\Lambda$ is dominant integral for $\frg_0$ and, whenever $k:=\mu_n<m$, one has $\lambda_{k+1}=\cdots=\lambda_m=0$.
\end{enumerate}
Under the assumption that $\Lambda$ is dominant integral for $\frg_0$, these conditions admit the following descriptions:
\begin{enumerate}[label=\textup{(\alph*)}]
\item If $\lambda\in(\frac12+\Z)^m$, they hold if and only if $\mu_n\geq m$.
\item If $\lambda\in\Z^m$, they hold if and only if, for some $(n\mid m)$-hook partition $\tau$,
\[
\begin{cases}
\Lambda=\tau^\natural, & \frg=B(m,n),\\[2pt]
\Lambda\in\{\tau^\natural,\tau^\natural_-\}, & \frg=D(m,n).
\end{cases}
\]
\end{enumerate}
\end{Thm}

\begin{proof}
\textup{(i)}$\Leftrightarrow$\textup{(ii)} is clear. Finite-dimensionality forces $\Lambda$ to be dominant integral for $\frg_0$. Thus all four conditions imply or require this assumption, which we impose for the remainder of the proof.

Odd reflections change coordinates by integers, so $\mu_{\mid}\in\Z^n$. Theorems \ref{thm:D-GK-all} and \ref{thm:B-GK-all} give
\[
\GKD V=
\begin{cases}
n^2-F_B\bigl((\mu_{\mid}^{\rho_0})^-\bigr), & \text{if }\frg=D(m,n),\\[2pt]
n^2-F_B\bigl((\mu_{\mid}^{\rho_{\mathfrak t}})^-\bigr), & \text{if }\frg=B(m,n),
\end{cases}
\]
since $\lambda$ is dominant integral.  These are the GK dimensions of the ordinary $C_n$ and $B_n$ simple highest weight modules of highest weight $\mu_{\mid}$, respectively. Since $\mu_{\mid}\in\Z^n$, Corollary \ref{cor:osp-terminal-lattices} and the ordinary finite-dimensional highest weight criterion therefore give
\[
\GKD V=0 \quad\Longleftrightarrow\quad \mu_{\mid,1}\geq\cdots\geq\mu_{\mid,n}\geq0.
\]
This proves \textup{(i)}$\Leftrightarrow$\textup{(iii)}.

We first show that each row of the algorithm in Proposition \ref{prop:odd-shuffle-BD}  preserves the weakly decreasing order of the integer sequence $(u_1,\ldots,u_n)$. Suppose that $u_j\ge u_{j+1}$ at the beginning of a row. If $u_j>u_{j+1}$, then $u_j-1\ge u_{j+1}$, so their order is preserved because each entry decreases by at most one. If $u_j=u_{j+1}$, the entry $u_{j+1}$ is processed first. If it remains unchanged, then $\ell_i+u_{j+1}=0$ and $\ell_i$
also remains unchanged; hence the subsequent update leaves $u_j$ unchanged as well. If $u_{j+1}$ decreases by one, then $u_j$, which decreases by at most one, cannot become smaller than it. Thus each row preserves the weakly decreasing order. Since $\mu_1\ge\cdots\ge\mu_n$, induction over the rows gives
$\mu_{\mid,1}\ge\cdots\ge\mu_{\mid,n}$. As these coordinates are integers, $\mu_{\mid}$ is dominant if and only if $\mu_{\mid,n}\ge0$.

Put $k=\mu_n$. If $k\ge m$, each coordinate decreases by at most $m$ throughout the algorithm in Proposition \ref{prop:odd-shuffle-BD}, so $\mu_{\mid,j}\ge\mu_j-m\ge k-m\ge0$ for every $j$. Hence condition (iii) holds. Suppose that $0\le k<m$. For $1\le i\le k$, row $i$ starts with $\ell_i=\lambda_i\ge0$ and, inductively, $u_n=k-i+1>0$. Since $u_n$ is processed first, the sum $\ell_i+u_n$ is positive, and this row decreases $u_n$ by one. Thus $u_n=0$ after the first $k$ rows. If $\lambda_{k+1}\ne0$, the update gives $u_n=-1$. Since subsequent updates cannot increase $u_n$, we obtain $\mu_{\mid,n}<0$. If $\lambda_{k+1}=0$, it implies that $\lambda_{k+1}=\cdots=\lambda_m=0$. Every remaining row then starts with $(\ell_i,u_n)=(0,0)$, so its first update leaves $u_n$ unchanged. The later updates in that row do not involve $u_n$, and hence $\mu_{\mid,n}=0$. Consequently, for $0\le k<m$, we have $\mu_{\mid,n}\ge0$ if and only if $\lambda_{k+1}=\cdots=\lambda_m=0$.
This proves the equivalence of (iii) and (iv).

To prove (a), assume $\lambda \in (\frac{1}{2} + \mathbb{Z})^m$. During the algorithm in Proposition \ref{prop:odd-shuffle-BD}, $\ell_i$ only increases by integers, so it is always a half-integer; $u_j$ only decreases by integers, so it is always an integer. Therefore $\ell_i + u_j \in \frac{1}{2} + \mathbb{Z}$, which can never be zero. Thus, every row decreases each $u_j$ by one. After $m$ rows, $\mu_| = \mu - m(1, \dots, 1)$. Consequently,
\[
\mu_{|,n} \ge 0 \iff \mu_n \ge m.
\]
This proves (a).

Finally, suppose $\lambda\in\Z^m$ and set
\[
\nu=
\begin{cases}
(\lambda_1,\ldots,\lambda_m), & \frg=B(m,n),\\[2pt]
(\lambda_1,\ldots,\lambda_{m-1},|\lambda_m|), & \frg=D(m,n).
\end{cases}
\]
Then $\nu$ is a partition and $\nu'_1=\#\{i:\nu_i>0\}\leq m$. Thus \textup{(iv)} is equivalent to $\nu'_1\leq\mu_n$. This inequality makes
\[
\tau=(\mu_1,\ldots,\mu_n,\nu'_1,\nu'_2,\ldots)
\]
a partition, and $\tau_{n+1}=\nu'_1\leq m$ makes it an $(n\mid m)$-hook partition. Its  second factor $\nu'$ has conjugate $\nu$, so it gives $\Lambda=\tau^\natural$ in type $B$ and $\Lambda\in\{\tau^\natural,\tau^\natural_-\}$ in type $D$. Conversely,  hook partition expression gives $\nu'_1=\tau_{n+1}\leq\tau_n=\mu_n$, hence (iv) holds. This proves \textup{(b)} and completes the proof.
\end{proof}

\subsection{Finite-dimensional representations of $G(3)$, $F(4)$ and $D(2,1;\alpha)$}

\begin{Thm}[$G_3$]\label{cor:G3-GK-zero}
The module $V$ in Theorem \ref{thm:GK-G3-exceptional} is finite-dimensional if and only if $b_1,b_2\in\mathbb Z_{\geq0}$ and one of the following holds:
\begin{enumerate}
    \item  $\lambda=0,\ b_1=b_2=0$;
 \item $\lambda=2,\ b_2=0$;
 \item $\lambda\in\mathbb Z_{\geq3}$.
\end{enumerate}
This recovers the $G(3)$-conditions in \cite[(3.3)]{SZ}, with $b_2=0$ corresponding to equality of the first two $G_2$ coordinates.
\end{Thm}
\begin{proof}
   By \eqref{eq:G3-GK-exact}, we have
\[
 \GKD V=0
 \quad\Longleftrightarrow\quad
 b_1,b_2\in\mathbb Z_{\geq0},\qquad
 \lambda-d_1-d_2-d_3\in\mathbb Z_{\geq0}.
\]
The latter condition forces $\lambda\in\mathbb Z_{\geq0}$. Conversely, if $\lambda\in\mathbb Z_{\geq3}$, it always holds because $d_1+d_2+d_3\leq3$. For $\lambda=0$, the condition is equivalent to $d_1=d_2=d_3=0$. By Theorem \ref{thm:GK-G3-exceptional}, $d_1=0$ then gives $3b_1+2b_2=0$, hence $b_1=b_2=0$. Conversely, $\lambda=b_1=b_2=0$ gives $d_1=d_2=d_3=0$. For $\lambda=1$, we have $2+3b_1+2b_2>0$ and $1+3b_1+b_2>0$. Thus $d_1=d_2=1$, and $\lambda-d_1-d_2-d_3<0$. This contradicts the latter condition. For $\lambda=2$, the similar discussion gives $d_1=d_2=1$ and $d_3=\mathbf1_{\{b_2\neq0\}}$. Therefore $\lambda-d_1-d_2-d_3=-d_3$ belongs to $\mathbb Z_{\geq0}$ if and only if $b_2=0$.
\end{proof}

\begin{Thm}[$F_4$]\label{cor:F4-GK-zero}
The module $V$ in Theorem \ref{thm:GK-F4-exceptional} is finite-dimensional if and only if $\lambda$ is dominant integral for $\frso(7)$ and $2\mu^{(4)}\in\mathbb Z_{\geq0}$. It has GK dimension one precisely when $\lambda$ is dominant integral and $2\mu^{(4)}\notin\mathbb Z_{\geq0}$. This recovers the $F(4)$-conditions in \cite[(3.3)]{SZ}.

\end{Thm}
\begin{proof}
Put $N=d_1+d_2+d_3+d_4$. By \eqref{eq:F4-GK-updates} and \eqref{eq:F4-GK-exact} we have
\[
 \GKD V=0
 \quad\Longleftrightarrow\quad
 \lambda\text{ is dominant integral for }\frso(7),\qquad
 2\mu-N\in\mathbb Z_{\geq0}.
\]
The first condition is $\lambda_1-\lambda_2,\lambda_2-\lambda_3,2\lambda_3\in\mathbb Z_{\geq0}$, and in particular
$\lambda_1\geq\lambda_2\geq\lambda_3\geq0$. The second condition forces $2\mu\in\mathbb Z_{\geq0}$. If $\mu\geq2$ and $2\mu$ is an integer, it always holds because $N\leq4$. It remains to consider $\mu=0,\frac12,1,\frac32$.

Write $d_r=\mathbf1_{\{Q_r\neq0\}}$, where
\[
\begin{aligned}
 Q_1&=3\mu+\lambda_1+\lambda_2+\lambda_3,\\
 Q_2&=3\mu+\lambda_1+\lambda_2-\lambda_3-d_1,\\
 Q_3&=3\mu+\lambda_1-\lambda_2+\lambda_3-d_1-2d_2,\\
 Q_4&=3\mu+\lambda_1-\lambda_2-\lambda_3-2d_1-d_2-d_3.
\end{aligned}
\]
For $\mu=0$, the condition $2\mu-N\geq0$ requires $N=0$. Then $Q_1=0$ forces $\lambda_1=\lambda_2=\lambda_3=0$.

For $\mu>0$, the integrality condition gives $\mu\geq\frac12$. We have $Q_1>0$, so $d_1=1$, and then
\[
 Q_2=3\mu+\lambda_1+\lambda_2-\lambda_3-1\geq3\mu-1>0.
\]
Thus $d_2=1$ as well. If $\mu=\frac12$, this gives $2\mu-N\leq-1$, this contradicts the latter condition. If $\mu=1$, nonnegativity of $2\mu-N$ requires $d_3=d_4=0$. Now
\[
 Q_3=\lambda_1-\lambda_2+\lambda_3=0,
\]
which vanishes precisely when $\lambda_1=\lambda_2$ and $\lambda_3=0$. Under these conditions $Q_4=0$ as well,
and $2\mu-N=0$.
Finally, if $\mu=\frac32$, then
\[
 Q_3=\tfrac32+\lambda_1-\lambda_2+\lambda_3>0,
\]
so $d_3=1$. Consequently $2\mu-N=-d_4$ is nonnegative if and only if
\[
 Q_4=\lambda_1-\lambda_2-\lambda_3+\tfrac12=0.
\]
By \cite[Table 1]{BMXX}, minimal (nonzero) GK dimensions of simple module in $B_3$ is $4$; the possible values are computed in the proof of Corollary \ref{cor:F4-GK-spectrum}. Thus \eqref{eq:F4-GK-exact} equals $1$ exactly when $\GKD L_{\frso(7)}(\lambda)= 0$ and $2\mu^{(4)}\notin\mathbb Z_{\geq0}$, proving the remaining assertion.
\end{proof}

\begin{Thm}[$D(2,1;\alpha)$]\label{thm-D(2,1,alpha)-GK-zero}
    The module $V$ in Theorem \ref{thm:GK-D21-exceptional} is finite-dimensional if and only if $\lambda_1-d_1-d_2,\lambda_2,\lambda_3\in\mathbb Z_{\geq0}$. This recovers the $D(2,1;\alpha)$-conditions in \cite[(3.3)]{SZ}. 
\end{Thm}
\begin{proof}
To make the finite-dimensional criterion explicit, assume $\lambda_2,\lambda_3\in\mathbb Z_{\geq0}$. Then $\lambda_1-d_1-d_2\in\mathbb Z_{\geq0}$ forces $\lambda_1\in\mathbb Z_{\geq0}$. If $\lambda_1\geq2$, it holds
automatically. If $\lambda_1=0$, it requires $d_1=d_2=0$; \eqref{eq:D21-GK-updates} implies $\lambda_2=\lambda_3=0$ since $\alpha\neq0$. If $\lambda_1=1$, the condition is $d_1+d_2\leq1$. If $d_1=0$, then
$\lambda_2+1=-\alpha(\lambda_3+1)$, and the condition holds for either value of $d_2$. If $d_1=1$, the condition holds precisely when $d_2=0$, which is equivalent to $\lambda_2+1=\alpha(\lambda_3+1)$. Hence it is equivalent to
\[
\lambda_2+1=\pm\alpha(\lambda_3+1).
\]
The subcase $d_1=d_2=0$ is included in the first alternative; it forces $\lambda_2=0$ and $\alpha=-1/(\lambda_3+1)$.
\end{proof}



\bibliographystyle{alpha}
\bibliography{reference.bib}

@Article{AJ1,
  author  = {Joseph, A.},
  title   = {{G}elfand--{K}irillov {D}imension for the {A}nnihilators of {S}imple {Q}uotients of {V}erma {M}odules},
  journal = {J. Lond. Math. Soc.},
  year    = {1978},
  volume  = {18},
  number  = {1},
  pages   = {50-60},
}

@Article{AJ2,
  author  = {Joseph, A.},
  title   = {On the variety of a highest weight module},
  journal = {J. Algebra},
  year    = {1984},
  volume  = {88},
  number  = {1},
  pages   = {238-278},
}

@incollection {AJ3,
    AUTHOR = {Joseph, A.},
     TITLE = {On the classification of primitive ideals in the enveloping
              algebra of a semisimple {L}ie algebra},
 BOOKTITLE = {Lie group representations, {I} ({C}ollege {P}ark, {M}d.,
              1982/1983)},
    SERIES = {Lecture Notes in Math.},
    VOLUME = {1024},
     PAGES = {30--76},
 PUBLISHER = {Springer, Berlin},
      YEAR = {1983},
      ISBN = {3-540-12725-9},
   MRCLASS = {17B35},
  MRNUMBER = {727849},
MRREVIEWER = {Wim\ H.\ Hesselink},
       DOI = {10.1007/BFb0071431},
       URL = {https://doi.org/10.1007/BFb0071431},
}

@article {AJ6,
    AUTHOR = {Joseph, A.},
     TITLE = {On the associated variety of a primitive ideal},
   JOURNAL = {J. Algebra},
  FJOURNAL = {Journal of Algebra},
    VOLUME = {93},
      YEAR = {1985},
    NUMBER = {2},
     PAGES = {509--523},
      ISSN = {0021-8693},
   MRCLASS = {17B35},
  MRNUMBER = {786766},
MRREVIEWER = {David\ Collingwood},
       DOI = {10.1016/0021-8693(85)90172-3},
       URL = {https://doi.org/10.1016/0021-8693(85)90172-3},
}

@Article{BB,
  author  = {W. Borho and J. L. Brylinski},
  title   = {Differential {O}perators on {H}omogeneous {S}paces. {I}. {I}rreducibility of the {A}ssociated {V}ariety for {A}nnihilators of {I}nduced {M}odules},
  journal = {Invent. Math.},
  year    = {1982},
  volume  = {69},
  number  = {3},
  pages   = {437-476},
}

@article {CLW,
    AUTHOR = {Cheng, S.-J. and Lam, N. and Wang, W. Q.},
     TITLE = {Super duality and irreducible characters of ortho-symplectic
              {L}ie superalgebras},
   JOURNAL = {Invent. Math.},
  FJOURNAL = {Inventiones Mathematicae},
    VOLUME = {183},
      YEAR = {2011},
    NUMBER = {1},
     PAGES = {189--224},
      ISSN = {0020-9910,1432-1297},
   MRCLASS = {17B10},
  MRNUMBER = {2755062},
       DOI = {10.1007/s00222-010-0277-4},
       URL = {https://doi.org/10.1007/s00222-010-0277-4},
}

@BOOK{Ca85,
  title = {Finite groups of {L}ie type},
  publisher = {John Wiley \& Sons, Inc., New York},
  year = {1985},
  author = {Carter, R. W.},
  pages = {xii+544},
  series = {Pure and Applied Mathematics (New York)},
  note = {Conjugacy classes and complex characters, A Wiley-Interscience Publication},
  mrclass = {20G40 (20-02 20C15)},
  mrnumber = {794307},
  mrreviewer = {David B. Surowski}
}

@article {Dong23,
    AUTHOR = {Dong, C.-P.},
     TITLE = {On the {H}elgason-{J}ohnson bound},
   JOURNAL = {Israel J. Math.},
  FJOURNAL = {Israel Journal of Mathematics},
    VOLUME = {254},
      YEAR = {2023},
    NUMBER = {1},
     PAGES = {373--397},
      ISSN = {0021-2172,1565-8511},
   MRCLASS = {22E46},
  MRNUMBER = {4591839},
MRREVIEWER = {William\ M.\ McGovern},
       DOI = {10.1007/s11856-022-2403-6},
       URL = {https://doi.org/10.1007/s11856-022-2403-6},
}

@article {Su,
    AUTHOR = {Su, Y. C.},
     TITLE = {Composition factors of {K}ac modules for the general linear {L}ie superalgebras},
   JOURNAL = {Math. Z.},
    VOLUME = {252},
      YEAR = {2006},
    NUMBER = {4},
     PAGES = {731--754},
      ISSN = {1432-1823},
       URL = {https://doi.org/10.1007/s00209-005-0874-x},
}

@article {CM21,
    AUTHOR = {Chen, C.-W. and Mazorchuk, V.},
     TITLE = {Simple supermodules over {L}ie superalgebras},
   JOURNAL = {Trans. Amer. Math. Soc.},
  FJOURNAL = {Transactions of the American Mathematical Society},
    VOLUME = {374},
      YEAR = {2021},
    NUMBER = {2},
     PAGES = {899--921},
      ISSN = {0002-9947,1088-6850},
   MRCLASS = {17B10 (16E30)},
  MRNUMBER = {4196382},
MRREVIEWER = {Iwan\ Praton},
       DOI = {10.1090/tran/8303},
       URL = {https://doi.org/10.1090/tran/8303},
}

@article {Geck,
    AUTHOR = {Geck, M.},
     TITLE = {$\textsf{PyCox}$: computing with (finite) {C}oxeter groups and
              {I}wahori-{H}ecke algebras},
   JOURNAL = {LMS J. Comput. Math.},
  FJOURNAL = {LMS Journal of Computation and Mathematics},
    VOLUME = {15},
      YEAR = {2012},
     PAGES = {231--256},
      ISSN = {1461-1570},
   MRCLASS = {20C40 (20C08 20F55)},
  MRNUMBER = {2988815},
MRREVIEWER = {Chi\ Kin\ Mak},
       DOI = {10.1112/S1461157012001064},
       URL = {https://doi.org/10.1112/S1461157012001064},
}

@article{GG21,
  author  = {Gorelik, M. and Grantcharov, D.},
  title   = {Simple bounded highest weight modules of basic classical {L}ie superalgebras},
  journal = {Transform. Groups},
  volume  = {26},
  number  = {3},
  pages   = {893--914},
  year    = {2021}
}

@article{Mu97,
  author  = {Musson, I. M.},
  title   = {The enveloping algebra of the {Lie} superalgebra $\mathfrak{osp}(1,2r)$},
  journal = {Represent. Theory},
  volume  = {1},
  pages   = {405--423},
  year    = {1997},
  doi     = {10.1090/S1088-4165-97-00020-4}
}

@article {CM18,
    AUTHOR = {Coulembier, K. and Musson, I. M.},
     TITLE = {The primitive spectrum for {$\mathfrak{gl}(m|n)$}},
   JOURNAL = {Tohoku Math. J. (2)},
  FJOURNAL = {The Tohoku Mathematical Journal. Second Series},
    VOLUME = {70},
      YEAR = {2018},
    NUMBER = {2},
     PAGES = {225--266},
      ISSN = {0040-8735,2186-585X},
   MRCLASS = {16S30 (16D60 17B10)},
  MRNUMBER = {3810240},
MRREVIEWER = {Iwan\ Praton},
       DOI = {10.2748/tmj/1527904821},
       URL = {https://doi.org/10.2748/tmj/1527904821},
}

@article {SZ,
    AUTHOR = {Su, Y. C. and Zhang, R. B.},
     TITLE = {Generalised {J}antzen filtration of exceptional {L}ie superalgebras},
   JOURNAL = {Isr. J. Math.},
    VOLUME = {212},
      YEAR = {2016},
    NUMBER = {2},
     PAGES = {635--676},
       URL = {https://doi.org/10.1007/s11856-016-1301-1},
}

@article{BGXW,
  author    = {Bai, Z. Q. and Gao, F. and Wang, Y. T. and Xie, X.},
  title     = { Gelfand--{K}irillov dimensions and annihilator varieties of highest
 weight modules of exceptional type {L}ie algebras},
  journal   = {ArXiv:2509.24346v1},
year = {2025},
}

@article{ML,
title = {The representation theory of the exceptional {L}ie superalgebras {F}(4) and {G}(3)},
journal = {J. Algebra},
volume = {419},
pages = {167-222},
year = {2014},
issn = {0021-8693},
doi = {https://doi.org/10.1016/j.jalgebra.2014.07.016},
url = {https://www.sciencedirect.com/science/article/pii/S0021869314004049},
author = {L. Martirosyan},
}

@book {Mu12,
    AUTHOR = {Musson, I. M.},
     TITLE = {Lie superalgebras and enveloping algebras},
    SERIES = {Graduate Studies in Mathematics},
    VOLUME = {131},
 PUBLISHER = {American Mathematical Society, Providence, RI},
      YEAR = {2012},
     PAGES = {xx+488},
      ISBN = {978-0-8218-6867-6},
   MRCLASS = {17-02 (16S30 17B35)},
  MRNUMBER = {2906817},
MRREVIEWER = {Aleksandr\ Nikolaevich\ Sergeev},
       DOI = {10.1090/gsm/131},
       URL = {https://doi.org/10.1090/gsm/131},
}

@book {CW12,
    AUTHOR = {Cheng, S.-J. and Wang, W. Q.},
     TITLE = {Dualities and representations of {L}ie superalgebras},
    SERIES = {Graduate Studies in Mathematics},
    VOLUME = {144},
 PUBLISHER = {American Mathematical Society, Providence, RI},
      YEAR = {2012},
     PAGES = {xviii+302},
      ISBN = {978-0-8218-9118-6},
   MRCLASS = {17B10},
  MRNUMBER = {3012224},
MRREVIEWER = {Aleksandr\ Nikolaevich\ Sergeev},
       DOI = {10.1090/gsm/144},
       URL = {https://doi.org/10.1090/gsm/144},
}

@article {CW22,
    AUTHOR = {Cheng, S.-J. and Wang, W. Q.},
     TITLE = {Character formulae in category $\mathcal{O}$ for exceptional {L}ie superalgebra {G}(3)},
    VOLUME = {62},
   JOURNAL = {Kyoto J. Math.},
      YEAR = {2022},
     PAGES = {719--751},
       DOI = {10.1090/gsm/144},
       URL = {https://doi.org/10.1090/gsm/144},
}

@article {PS94,
    AUTHOR = {Penkov, I. and Serganova, V.},
     TITLE = {Generic irreducible representations of finite-dimensional
              {L}ie superalgebras},
   JOURNAL = {Internat. J. Math.},
  FJOURNAL = {International Journal of Mathematics},
    VOLUME = {5},
      YEAR = {1994},
    NUMBER = {3},
     PAGES = {389--419},
      ISSN = {0129-167X,1793-6519},
   MRCLASS = {17B10},
  MRNUMBER = {1274125},
MRREVIEWER = {Joris\ Van der Jeugt},
       DOI = {10.1142/S0129167X9400022X},
       URL = {https://doi.org/10.1142/S0129167X9400022X},
}

@incollection {Kac78,
    AUTHOR = {Kac, V. G.},
     TITLE = {Representations of classical {L}ie superalgebras},
 BOOKTITLE = {Differential geometrical methods in mathematical physics, {II}
              ({P}roc. {C}onf., {U}niv. {B}onn, {B}onn, 1977)},
    SERIES = {Lecture Notes in Math.},
    VOLUME = {676},
     PAGES = {597--626},
 PUBLISHER = {Springer, Berlin},
      YEAR = {1978},
      ISBN = {3-540-08935-7},
   MRCLASS = {17B70 (81E10)},
  MRNUMBER = {519631},
MRREVIEWER = {H.\ de Vries},
}

@article {Mu06,
    AUTHOR = {Musson, I. M.},
     TITLE = {Lie superalgebras, {C}lifford algebras, induced modules and
              nilpotent orbits},
   JOURNAL = {Adv. Math.},
  FJOURNAL = {Advances in Mathematics},
    VOLUME = {207},
      YEAR = {2006},
    NUMBER = {1},
     PAGES = {39--72},
      ISSN = {0001-8708,1090-2082},
   MRCLASS = {17B35 (16W55)},
  MRNUMBER = {2264065},
MRREVIEWER = {William\ M.\ McGovern},
       DOI = {10.1016/j.aim.2005.03.016},
       URL = {https://doi.org/10.1016/j.aim.2005.03.016},
}

@article {CL13,
    AUTHOR = {Cao, B. T. and Luo, L.},
     TITLE = {Trivial module for ortho-symplectic {L}ie superalgebras and
              {L}ittlewood's formula},
   JOURNAL = {Sci. China Math.},
  FJOURNAL = {Science China. Mathematics},
    VOLUME = {56},
      YEAR = {2013},
    NUMBER = {11},
     PAGES = {2251--2260},
      ISSN = {1674-7283,1869-1862},
   MRCLASS = {17B20 (17B10)},
  MRNUMBER = {3123569},
MRREVIEWER = {Stephen\ Slebarski},
       DOI = {10.1007/s11425-012-4535-3},
       URL = {https://doi.org/10.1007/s11425-012-4535-3},
}

@article {EH04,
    AUTHOR = {Enright, T. J. and M. Hunziker},
     TITLE = {Resolutions and {H}ilbert series of determinantal varieties
              and unitary highest weight modules},
   JOURNAL = {J. Algebra},
  FJOURNAL = {Journal of Algebra},
    VOLUME = {273},
      YEAR = {2004},
    NUMBER = {2},
     PAGES = {608--639},
      ISSN = {0021-8693,1090-266X},
   MRCLASS = {17B10 (22E46)},
  MRNUMBER = {2037715},
       DOI = {10.1016/S0021-8693(03)00159-5},
       URL = {https://doi.org/10.1016/S0021-8693(03)00159-5},
}

@book {FSS2000,
    AUTHOR = {Frappat, L. and Sciarrino, A. and Sorba, P.},
     TITLE = {Dictionary on {L}ie algebras and superalgebras},
      NOTE = {With 1 CD-ROM (Windows, Macintosh and UNIX)},
 PUBLISHER = {Academic Press, Inc., San Diego, CA},
      YEAR = {2000},
     PAGES = {xxii+410},
      ISBN = {0-12-265340-8},
   MRCLASS = {17-00 (17Bxx)},
  MRNUMBER = {1773773},
MRREVIEWER = {Aleksandr\ A.\ Mikhalev},
}

@article {Kac77,
    AUTHOR = {Kac, V. G.},
     TITLE = {Lie superalgebras},
   JOURNAL = {Adv. Math.},
  FJOURNAL = {Advances in Mathematics},
    VOLUME = {26},
      YEAR = {1977},
    NUMBER = {1},
     PAGES = {8--96},
      ISSN = {0001-8708},
   MRCLASS = {17B10 (17B65)},
  MRNUMBER = {486011},
MRREVIEWER = {H.\ de\ Vries},
       DOI = {10.1016/0001-8708(77)90017-2},
       URL = {https://doi.org/10.1016/0001-8708(77)90017-2},
}

@article {BJ24,
    AUTHOR = {Bai, Z. Q. and Jiang, J.},
     TITLE = {Gelfand--{K}irillov dimension and reducibility of scalar
              generalized {V}erma modules for classical {L}ie algebras},
   JOURNAL = {Acta Math. Sin. (Engl. Ser.)},
  FJOURNAL = {Acta Mathematica Sinica (English Series)},
    VOLUME = {40},
      YEAR = {2024},
    NUMBER = {3},
     PAGES = {658--706},
      ISSN = {1439-8516,1439-7617},
   MRCLASS = {17B10 (17B20 22E47)},
  MRNUMBER = {4712827},
MRREVIEWER = {Wei\ Xiao},
       DOI = {10.1007/s10114-024-2676-2},
       URL = {https://doi.org/10.1007/s10114-024-2676-2},
}

@article {SZ12,
    AUTHOR = {Su, Y. C. and Zhang, R. B.},
     TITLE = {Generalised {V}erma modules for the orthosymplectic {L}ie
              superalgebra {$\mathfrak{osp}_{k|2}$}},
   JOURNAL = {J. Algebra},
  FJOURNAL = {Journal of Algebra},
    VOLUME = {357},
      YEAR = {2012},
     PAGES = {94--115},
      ISSN = {0021-8693,1090-266X},
   MRCLASS = {17B10 (17B56)},
  MRNUMBER = {2905244},
MRREVIEWER = {J\"org\ Feldvoss},
       DOI = {10.1016/j.jalgebra.2012.01.026},
       URL = {https://doi.org/10.1016/j.jalgebra.2012.01.026},
}

@book {MR2000,
    AUTHOR = {Krause, G. R. and Lenagan, T. H.},
     TITLE = {Growth of algebras and {G}elfand-{K}irillov dimension},
    SERIES = {Graduate Studies in Mathematics},
    VOLUME = {22},
   EDITION = {Revised},
 PUBLISHER = {American Mathematical Society, Providence, RI},
      YEAR = {2000},
     PAGES = {x+212},
      ISBN = {0-8218-0859-1},
   MRCLASS = {16P90},
  MRNUMBER = {1721834},
MRREVIEWER = {Martha\ K.\ Smith},
       DOI = {10.1090/gsm/022},
       URL = {https://doi.org/10.1090/gsm/022},
}

@Article{BX1,
  author  = {Z. Q. Bai and X. Xie},
  title   = {{G}elfand--{K}irillov dimensions of highest weight {H}arish-{C}handra modules for {${\rm SU}(p,q)$}},
  journal = {Int. Math. Res. Not. IMRN},
  year    = {2019},
  number  = {14},
  pages   = {4392-4418},
}

@Book{lusztig2003hecke,
  title      = {Hecke algebras with unequal parameters},
  publisher  = {American Mathematical Society, Providence, RI},
  year       = {2003},
  author     = {Lusztig, G.},
  volume     = {18},
  series     = {CRM Monograph Series},
  isbn       = {0-8218-3356-1},
  mrclass    = {20C08 (20F55)},
  mrnumber   = {1974442 (2004k:20011)},
  mrreviewer = {G{\"o}tz Pfeiffer},
  pages      = {vi+136},
}

@InCollection{lusztig1985A_n,
  author    = {Lusztig, G.},
  title     = {The two-sided cells of the affine {W}eyl group of type {$\tilde{A}_n$}},
  booktitle = {Infinite-dimensional groups with applications ({B}erkeley, {C}alif., 1984)},
  publisher = {Springer, New York},
  year      = {1985},
  volume    = {4},
  series    = {Math. Sci. Res. Inst. Publ.},
  pages     = {275--283},
  doi       = {10.1007/978-1-4612-1104-4_10},
  mrclass   = {20G15 (20G05)},
  mrnumber  = {823323 (87i:20081)},
  url       = {http://dx.doi.org/10.1007/978-1-4612-1104-4_10},
}

@InCollection{lusztig1985cellsI,
  author     = {Lusztig, G.},
  title      = {Cells in affine {W}eyl groups},
  booktitle  = {Algebraic groups and related topics ({K}yoto/{N}agoya, 1983)},
  publisher  = {North-Holland, Amsterdam},
  year       = {1985},
  volume     = {6},
  series     = {Adv. Stud. Pure Math.},
  pages      = {255--287},
  mrclass    = {20G15},
  mrnumber   = {803338 (87h:20074)},
  mrreviewer = {Bhama Srinivasan},
}

@Article{BM,
  author  = {Z. Q. Bai and M. Hunziker},
  title   = {The {G}elfand–{K}irillov dimension of a unitary highest weight module},
  journal = {Science China-mathematics.},
  year    = {2015},
  volume  = {58},
  number  = {12},
  pages   = {2489-2498},
}

@Article{BX,
  author  = {Z. Q. Bai and W. Xiao},
  title   = {Gelfand--{K}irillov dimension and reducibility of scalar generalized {V}erma modules},
  journal = {Acta Math Sin Engl Ser.},
  year    = {2019},
  volume  = {35},
  number  = {11},
  pages   = {1854-1860},
}

@Article{BXX,
  author  = {Z. Q. Bai and W. Xiao and X. Xie},
  journal = {IMRN},
  title   = {Gelfand--{K}irillov {D}imensions and {A}ssociated {V}arieties of {H}ighest {W}eight {M}odules},
  year    = {2023},
  number  = {10},
  pages   = {8101--8142},
}

@Article{Vo78,
  author     = {D. A. Vogan},
  title      = {Gelfand--{K}irillov dimension for {H}arish-{C}handra modules},
  journal    = {Invent. Math.},
  year       = {1978},
  volume     = {48},
  number     = {1},
  pages      = {75--98},
  issn       = {0020-9910},
  fjournal   = {Inventiones Mathematicae},
  mrclass    = {17B35},
  mrnumber   = {0506503},
  mrreviewer = {Anthony Joseph},
}

@Article{KL,
  author     = {Kazhdan, D. and Lusztig, G.},
  journal    = {Invent. Math.},
  title      = {Representations of {C}oxeter groups and {H}ecke algebras},
  year       = {1979},
  number     = {2},
  pages      = {165--184},
  volume     = {53},
  fjournal   = {Inventiones Mathematicae},
  mrclass    = {20H15 (17B35 20G05 22E47)},
  mrnumber   = {560412},
  mrreviewer = {Vinay V. Deodhar},
}

@article {GS10,
    AUTHOR = {Gruson, C. and Serganova, V.},
     TITLE = {Cohomology of generalized supergrassmannians and character
              formulae for basic classical {L}ie superalgebras},
   JOURNAL = {Proc. Lond. Math. Soc.},
  FJOURNAL = {Proceedings of the London Mathematical Society. Third Series},
    VOLUME = {101},
      YEAR = {2010},
    NUMBER = {3},
     PAGES = {852--892},
      ISSN = {0024-6115,1460-244X},
   MRCLASS = {17B10 (14F05)},
  MRNUMBER = {2734963},
       DOI = {10.1112/plms/pdq014},
       URL = {https://doi.org/10.1112/plms/pdq014},
}

@article {JJ,
    AUTHOR = {Jiang, J.},
     TITLE = {Reducibility of scalar generalized {V}erma modules of minimal
              parabolic type},
   JOURNAL = {Algebra Colloq.},
  FJOURNAL = {Algebra Colloquium},
    VOLUME = {32},
      YEAR = {2025},
    NUMBER = {4},
     PAGES = {623--634},
      ISSN = {1005-3867,0219-1733},
   MRCLASS = {17B10 (16S30 17B20)},
  MRNUMBER = {4990583},
       DOI = {10.1142/S1005386725000458},
       URL = {https://doi.org/10.1142/S1005386725000458},
}

@article {JW,
    AUTHOR = {Jiang, J. and Wu, S. Y.},
     TITLE = {Reducibility of scalar generalized {V}erma modules of minimal
              parabolic type {II}},
   JOURNAL = {J. Lie Theory.},
  FJOURNAL = {Journal of Lie Theory},
    VOLUME = {36},
      YEAR = {2026},
    NUMBER = {2-3},
     PAGES = {529--559},
     URL = {https://doi.org/10.5802/jolt.1439},
}

@article {BMXX,
    AUTHOR = {Bai, Z. Q. and Ma, J.-J. and Xiao, W. and Xie, X.},
     TITLE = {Associated varieties of minimal highest weight modules},
   JOURNAL = {Represent. Theory},
  FJOURNAL = {Representation Theory. An Electronic Journal of the American
              Mathematical Society},
    VOLUME = {28},
      YEAR = {2024},
     PAGES = {498--513},
      ISSN = {1088-4165},
   MRCLASS = {22E47 (17B08 17B20)},
  MRNUMBER = {4822987},
MRREVIEWER = {Yu-Feng\ Yao},
       DOI = {10.1090/ert/681},
       URL = {https://doi.org/10.1090/ert/681},
}

@article{Kit26,
  author  = {Kitagawa, M.},
  title   = {Cartan subalgebras for restrictions of
             $\mathfrak{g}$-modules},
  journal = {J. Algebra},
  volume  = {704},
  pages   = {165--201},
  year    = {2026},
  doi     = {10.1016/j.jalgebra.2026.04.052}
}

\end{document}